\documentclass[onefignum,onetabnum]{siamart171218}

\usepackage{tikz}
\usetikzlibrary{decorations.pathreplacing,angles,quotes, calc}
\usepackage{caption}
\usepackage{subcaption}

\usepackage{amsmath}
\usepackage{bm}
\usepackage{amsfonts}
\usepackage{commath}
\usepackage{lipsum}  
\usepackage{graphicx}
\usepackage{epstopdf}
\usepackage{algorithmic}
\ifpdf
  \DeclareGraphicsExtensions{.eps,.pdf,.png,.jpg}
\else
  \DeclareGraphicsExtensions{.eps}
\fi

\newsiamremark{remark}{Remark}
\newsiamremark{hypothesis}{Hypothesis}
\crefname{hypothesis}{Hypothesis}{Hypotheses}
\newsiamthm{claim}{Claim}

\headers{Structure-preserving domain decomposition for data-driven models}{Jiang, Actor, Roberts, Trask}

\title{A structure-preserving domain decomposition method for data-driven modeling
\thanks{{\bfseries Funding:} This article has been authored by an employee of National Technology \& Engineering Solutions of Sandia, LLC under Contract No. DE-NA0003525 with the U.S. Department of Energy (DOE). The employee owns all right, title and interest in and to the article and is solely responsible for its contents. The United States Government retains and the publisher, by accepting the article for publication, acknowledges that the United States Government retains a non-exclusive, paid-up, irrevocable, world-wide license to publish or reproduce the published form of this article or allow others to do so, for United States Government purposes. The DOE will provide public access to these results of federally sponsored research in accordance with the DOE Public Access Plan \url{https://www.energy.gov/downloads/doe-public-access-plan.}}}

 \author{Shuai Jiang\thanks{Center for Computing Research, Sandia National Laboratories, Albuquerque, NM. USA}\and Jonas Actor\footnotemark[2]\and Scott Roberts\thanks{Engineering Sciences Center, Sandia National Laboratories, 1515 Eubank SE, Albuquerque, NM 87123, USA}\and Nathaniel Trask\thanks{Department of Mechanical Engineering and Applied Mechanics, University of Pennsylvania,
Philadelphia, PA, USA. \email{ntrask@seas.upenn.edu}}}

\makeatletter
\newcommand*{\addFileDependency}[1]{
  \typeout{(#1)}
  \@addtofilelist{#1}
  \IfFileExists{#1}{}{\typeout{No file #1.}}
}
\makeatother

\newcommand{\mat}[1]{\bm{#1}}
\renewcommand{\div}{\nabla \cdot}
\newcommand{\ipbc}[2]{\langle #1, #2 \rangle}
\usepackage{tikz,pgfplots}
\pgfplotsset{compat=1.14}

\usetikzlibrary{calc}

\newcommand{\logLogSlopeTriangle}[5]
{

    \pgfplotsextra
    {
        \pgfkeysgetvalue{/pgfplots/xmin}{\xmin}
        \pgfkeysgetvalue{/pgfplots/xmax}{\xmax}
        \pgfkeysgetvalue{/pgfplots/ymin}{\ymin}
        \pgfkeysgetvalue{/pgfplots/ymax}{\ymax}

        \pgfmathsetmacro{\xArel}{#1}
        \pgfmathsetmacro{\yArel}{#3}
        \pgfmathsetmacro{\xBrel}{#1-#2}
        \pgfmathsetmacro{\yBrel}{\yArel}
        \pgfmathsetmacro{\xCrel}{\xArel}

        \pgfmathsetmacro{\lnxB}{\xmin*(1-(#1-#2))+\xmax*(#1-#2)} 
        \pgfmathsetmacro{\lnxA}{\xmin*(1-#1)+\xmax*#1} 
        \pgfmathsetmacro{\lnyA}{\ymin*(1-#3)+\ymax*#3} 
        \pgfmathsetmacro{\lnyC}{\lnyA+#4*(\lnxA-\lnxB)}
        \pgfmathsetmacro{\yCrel}{(\lnyC-\ymin)/(\ymax-\ymin)} 

        \coordinate (A) at (rel axis cs:\xArel,\yArel);
        \coordinate (B) at (rel axis cs:\xBrel,\yBrel);
        \coordinate (C) at (rel axis cs:\xCrel,\yCrel);

        \draw[#5]   (A)-- node[pos=0.5,anchor=north] {1}
                    (B)-- 
                    (C)-- node[pos=0.5,anchor=west] {#4}
                    cycle;
    }
}

\begin{document}

\maketitle

\begin{abstract}
We present a domain decomposition strategy for developing structure-preserving finite element discretizations from data when exact governing equations are unknown.
On subdomains, trainable Whitney form elements are used to identify structure-preserving models from data, providing a Dirichlet-to-Neumann map which may be used to globally construct a mortar method. 
The reduced-order local elements may be trained offline to reproduce high-fidelity Dirichlet data in cases where first principles model derivation is either intractable, unknown, or computationally prohibitive. 
In such cases, particular care must be taken to preserve structure on both local and mortar levels without knowledge of the governing equations, as well as to ensure well-posedness and stability of the resulting monolithic data-driven system. 
This strategy provides a flexible means of both scaling to large systems and treating complex geometries, and is particularly attractive for multiscale problems with complex microstructure geometry.
While consistency is traditionally obtained in finite element methods via quasi-optimality results and the Bramble-Hilbert lemma as the local element diameter $h\rightarrow0$, our analysis establishes notions of accuracy and stability for finite $h$ with accuracy coming from matching data. 
Numerical experiments and analysis establish properties for $H(\operatorname{div})$ problems in small data limits ($\mathcal{O}(1)$ reference solutions). 
\end{abstract}



\section{Introduction}
We consider the problem of identifying a model from data when the governing equations are unknown, but the conservation structure is known. Namely, one may know that fluxes associated with mass, momentum, or energy are conserved, but be unable to derive specific expressions for those fluxes.

We assume a class of models of the form
\begin{equation}
    \begin{aligned}
    \div \mat u &= - f  &\text{on }& \Omega, &\\
    \mat u &= h(p;\theta) &\text{on }& \Omega,& \\
    p &= g              &\text{on }& \partial \Omega&
    \end{aligned}
\end{equation}
where $\Omega \in \mathbb{R}^d$ is a Lipschitz domain, $f \in L^2(\Omega)$ forcing term, $g$ Dirichlet data, and $h$ a closure for the flux of unknown functional form approximated by a family of non-parametric regressors parameterized by $\theta$. 
We demonstrate on $\Omega \in \mathbb{R}^2$ exclusively, but the techniques shown here generalize to higher dimensions and arbitrary manifolds. 
For this class of problems, data is provided in the form $\mathcal{D} = \left\{(\mat u_k,f_k, h_k, g_k)\right\}_{k=1}^{N}$ and one identifies parameters $\theta$ which minimize error in a suitable norm, providing a model which may generalize by solving for choices of $f$ and $h$ outside the training set. 

By casting data-driven modeling in such a \textit{structure-preserving} framework, one aims to identify a model which balances a trade-off between rigorous preservation of physical/algebraic/stability structure while maintaining ``black-box'' approximation of as large a class of models as possible. 
This lies on a spectrum of methods in the literature spanning a trade-off between expressivity and exploitable structure. 
For example, operator regression methods aim to directly identify a solution map $(f,h)\rightarrow \mat u$ via interpolation in unconstrained Hilbert spaces (high expressivity), while PDE-constrained optimization \cite{biegler2003large, hinze2008optimization} assumes a known functional form for $h$ which requires only estimation of material parameters (highly structured with simplified analysis). 

For the purposes of this work we consider elliptic systems of $H(\operatorname{div})$-type where structure-preservation amounts to preserving notions of flux continuity. 
In the literature, preservation of other types of structure is a key challenge for data-driven models: gauge invariances associated with non-trivial null-spaces \cite{trask2022enforcing}, geometric structure associated with bracket dynamics \cite{gruber2023reversible,greydanus2019hamiltonian,desai2021port,hernandez2021structure}, group equivariance \cite{bergomi2019towards,villar2021scalars} and other structures \cite{celledoni2021structure}. 
Many of these approaches aim to enforce the invariances by construction rather than rely on data or training to ``learn'' them, allowing better performance in small-data limits and improved theoretical properties.

In our previous works \cite{actoradata,trask2022enforcing}, we have developed structure-preserving machine learning frameworks generalizing the discrete exterior calculus (DEC) and finite element exterior calculus (FEEC) (see \cref{sec:decfeec}). 
Both frameworks pose the learning of physics as identifying maps between cochains associated with a de Rham complex, and provide a number of desirable theoretical guarantees: preservation of exact sequence structure (e.g. $\div (\nabla \times) = 0$), exact local conservation of generalized fluxes, an exact Hodge decomposition, a Lax-Milgram stability theory for Hodge Laplacians, well-posedness theory for nonlinear problems, and a framework for treating problems with non-trivial null-spaces (e.g. electromagnetism). 
In the FEEC setting, a Dirichlet-to-Neumann map prescribing the exchange of generalized fluxes between subdomains is expressed in terms of parameterized Whitney forms, allowing the machine learning of geometric control volumes which optimally admit integral balance laws. 
While effective for providing rigorous structure-preservation, the scheme provides poor computational scaling whereby the number of degrees of freedom scale as $\mathcal O(N^k)$, where $N$ is the number of partitions and $k$ is the order of the Whitney form. 

The current work applies a divide-and-conquer strategy to mitigate this by partitioning the domain into disjoint, non-overlapping subdomains $\Omega = \cup_i \Omega_i$, whose exact specifications will be discussed later, and seeks local models restricted to each $\Omega_i$ of the form
\begin{equation}\label{eq:localmortargeneralcase}
    \begin{aligned}
    \div \mat u_i &= - f_i, \\
    \mat u_i &= h(p_i;\theta_i),\\
    p_i &= g_i&\text{ on } \partial \Omega_i, 
    \end{aligned}
\end{equation}
with the subscript $\cdot_i$ denoting appropriate restrictions of fields to $\Omega_i$. 
The framework for regressing local models is introduced in \cref{sec:whitneysec}. 
To train subdomain models, we can perform offline training over data $\mathcal{D}_i = \left\{(u_{i,k},f_{i,k},h_{i,k}, g_{i, k})\right\}_{k=1}^{N_{i}}$. 
This can be obtained either by taking the restriction of global data onto the subdomain ($g_i=p|_{\partial \Omega}$), or by performing simulations directly on each subdomain to  identify the local response to a representative mortar space (e.g. $g_i \in \mathbb{P}_m(\partial \Omega_i)$ the space of $m$\textsuperscript{th}-order polynomials). 
After obtaining local models, a mortar method is presented in \cref{sec:mortarsec} which is used to assemble local models into a global model on $\Omega$. 

For this data-driven mortar strategy, we impose two desired requirements:
\begin{enumerate}
    \item \textbf{R1: Preservation of structure across both scales:} For the $H(\operatorname{div})$ problems under consideration, the Whitney form construction admits interpretation as an integral balance law where fluxes are discretely treated as equal and opposite, providing a local conservation principle on each subdomain $\Omega_i$. 
    We require that the mortar formulation be compatible with this, so that when local elements are stitched together through the mortar we preserve conservation globally on $\Omega$.
    \item \textbf{R2: Stability of error at global scale:} 
    If, during pretraining, local models may be obtained to a given optimization error, we would like to quantify the error induced at a global level by the coupling process. 
    Ideally this would be bound by a constant independent of the number of subdomains, so that the global error remains comparable to that of the locally trained models as many elements are coupled together and performance does not degenerate in the limit of many data-driven elements.
\end{enumerate}
We demonstrate both requirements either in analytical proofs in \cref{sec:mortarsec}, or via numerical example in \cref{sec:numerical-results}.
Finally, the technical proofs and more details regarding training are shown in the appendix \cref{sec:appendix}.

\section{Relation to previous work}
The proposed strategy exploits a connection to structure-preserving PDE discretization to ensure that physics are enforced by construction, rather than via the penalty formulation typically pursued in the physics-informed machine learning literature. We summarize the relationship between this approach and the literature, as well as how our strategy relates to classical domain decomposition methods.

\subsection{Data-driven DEC/FEEC and Dirichlet-to-Neumann maps}\label{sec:decfeec}
In traditional numerical analysis the discrete exterior calculus (DEC) is a framework for constructing and analyzing staggered finite volume schemes \cite{hirani2003discrete,nicolaides1992direct}. 
The generalized Stokes theorem is used to define discrete vector calculus operators (e.g. grad/curl/div) which map between differential forms on a pair of primal/dual computational meshes. 
The finite element exterior calculus (FEEC) generalizes DEC by constructing finite element spaces which interpolate differential forms and provides variational extensions \cite{arnold2018finite}. 

In the data-driven exterior calculus (DDEC) \cite{trask2022enforcing}, DEC operators are parameterized in a manner allowing the learning of well-posed models on graphs, where data is used to identify the inner-product associated with codifferential operators.
In \cite{actoradata}, it was shown that a family of data-driven Whitney forms may be constructed from parameterized partitions-of-unity (POUs).
The Whitney forms admit a de Rham complex which encodes POU geometry as differentiable control volumes and their higher order boundaries (faces/edges/etc) without reference to a traditional mesh. 
An inner-product is induced by the geometry of the control volumes, supporting the discovery of models in terms of control volume balances.
This allow a data-driven FEEC extension of DDEC which we use extensively in this work. 
Furthermore, by posing integral balances as relationships between domains and fluxes on their boundaries, we work with degrees of freedom which naturally conform to the trace spaces necessary for a mortar strategy.

\subsection{Structure-preserving ML vs. physics-informed ML} 

In the recent scientific machine learning literature, physics-informed methods broadly encompass frameworks where physical constraints are incorporated by adding (typically collocation) residuals to a loss function as a Tikhonov regularization with a penalty parameter \cite{cai2021physics}. 
This technique is simple to implement and, when used together with automatic differentiation, admits a simple treatment of inverse problems, discovery of ``missing physics'' or closures \cite{karniadakis2021physics,patel2022thermodynamically}, and uncertainty quantification \cite{yang2019adversarial,zhang2019quantifying}. 

The flexibility of the framework comes at the expense of solving a multi-objective optimization problem whereby the physics residual must be empirically weighted against the data loss, and can only be enforced to within optimization error \cite{wang2021understanding}. 
For certain classes of problems it is necessary to enforce physics to machine precision to obtain qualitatively correct answers; e.g. subsurface transport and lubrication flows depend crucially on exact conservation of mass \cite{trask2018compatible}, while electromagnetic problems which fail to provide an exactly divergence-free magnetic field predict qualitatively incorrect spectra \cite{arnold2018finite}. 
In the context of physics-informed learning, some works have pursued a penalty-based domain decomposition strategy with the goal of efficient distributed computation and more flexibility in neural network approximation \cite{jagtap2021extended}.
While effective, the collocation scheme and penalty formulation complicate analysis and preclude exact conservation, respectively. Because the desired conservation structure only holds to within optimization error, penalization may be insufficient for certain classes of applications.

\subsection{Choice of mortar scheme} 
Domain decomposition is a mature field, with many established options for how to couple solutions across arbitrary finite element subdomains \cite{toselli2004domain, smith1997domain}.
Representative rigorous methods range from (e.g. finite element tearing and interconnecting (FETI) \cite{farhat2001feti}, mortar methods \cite{bernardi1993domain}, and hybridizable discontinuous Galerkin methods \cite{cockburn2009unified}) impose continuity of fluxes and state at subdomain interfaces either strongly via Lagrange multipliers or weakly by using Nietsche's trick to introduce a variational penalty. 

For the div-grad problem, there is also a choice of working in either $H^1$- or $H(\operatorname{div})$-conforming spaces (e.g. $\mathbb{P}_1$/Nedelec or Raviart-Thomas/$\mathbb{P}_0$ mixed spaces), and whether one chooses to apply a mortar on the state or flux variables. 
Working in $H(\operatorname{div})$ is perhaps most natural, as the mortar space admits interpretation as a conservative flux that trivially preserves conservation structure \cite{arbogast2007multiscale}. 
However, this requires working with $d$- and $(d-1)$- dimensional Whitney forms. 
Our Whitney form construction scales with computational complexity $\mathcal{O}(N^k)$, where $N$ is the dimension of $0$\textsuperscript{th}-order Whitney forms and $k$ is the maximal order Whitney form. 
It is therefore preferable to exploit primal/dual structure and work in $H^1$, meaning only $0$\textsuperscript{th}- and $1$\textsuperscript{st}-order Whitney forms are used. 
This forces us to adopt an $H^1$ domain decomposition strategy similar to that developed by Glowinski and Wheeler \cite[\S 7]{glowinski1987domain}.
Further extensions are needed to easily incorporate and analyze the case where data-driven FEEC elements are used as the local solvers. 



\section{Local learning of Whitney form elements}\label{sec:whitneysec}
For brevity, we discuss only the fundamental aspects of data-driven DEC/FEEC necessary to describe the local element construction. 
For a complete exposition we direct readers to references for: data-driven exterior calculus \cite{trask2022enforcing}, data-driven finite element exterior calculus \cite{actoradata}, classical finite element exterior calculus for forward simulation \cite{arnold2018finite}, and Whitney forms \cite{gillette2016construction}.

Given a compact domain $\omega \in \mathbb{R}^2$ with finite open cover $\{U_i \}_{i=1}^N$, a \textit{partition of unity} (POU) is a collection of functions $\phi_i: \omega \rightarrow [0,1]$ such that $\phi_i(\bm{x})\geq 0$, $\operatorname{supp}(\phi_i) \subseteq U_i$, $\phi_i < \infty$ and $\sum_{i} \phi_i(\bm{x}) = 1$ for all $\bm{x} \in \omega$. 
We assume access to a \textit{parameterized POU} (PPOU) $\left\{\phi_i(\bm{x};\theta)\right\}_{i=1}^N$, which is continuous with respect to a parameter $\theta$. 

To construct Whitney forms, any trainable PPOU may be used, although in this work we adopt the same used in \cite{actoradata}. Starting with tensor-product B-splines on the unit domain, we refer to trainable vertex locations as \emph{fine-scale nodes/knots}. To approximate complex geometries, we consider a coarsening via convex combinations of the knots into our ultimate PPOU, $\left\{\phi_i(\bm{x};\theta)\right\}_{i=1}^N$ where $\theta$ denotes parameters corresponding to both knot locations and trainable entries of the convex combination tensor; see \cref{fig:convex-comb} for an illustrative figure of this process.

In \cite{actoradata}, the tensor-product grid points are parameterized using the distances between the grid points to avoid inversion of elements.
In particular, we can define the grid points in one dimension $\{t_i\}_{i=0}^n$ with $t_0 = 0$ and $t_1 = 1$ by parameterizing 
\begin{align}
    \sigma\left( \delta\right)_i= t_{i + 1} - t_i, \qquad i \in \{0, \ldots, n-1\}
\end{align}
where $\delta_i$ is a trainable parameter, and $\sigma$ is a sigmoid activation enforcing positivity. To parameterize a map of convex combination of knots, we consider a trainable two-tensor with softmax activation applied to each row; for details we refer to \cite{actoradata}. In what follows we adopt the simplified notation $\phi_i(\bm{x};\theta) = \phi_i$.
\begin{figure}
    \centering
    \includegraphics[width=.5\textwidth]{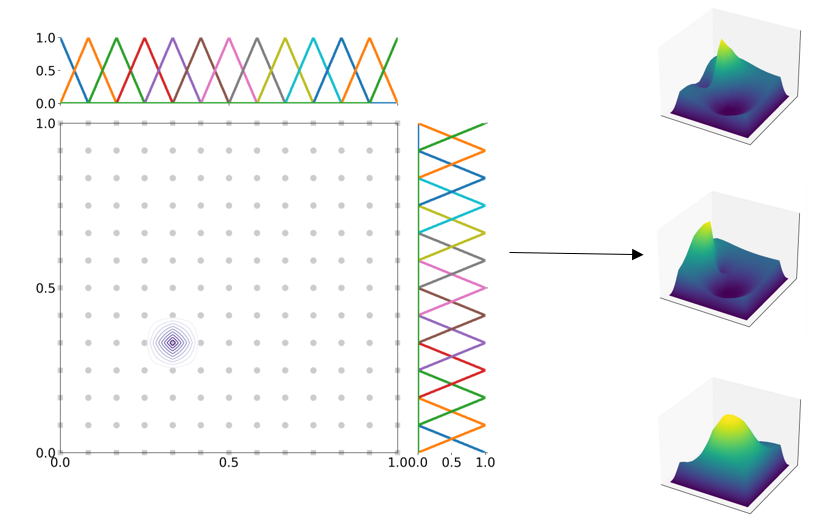}
    \caption{To construct a PPOU, we first consider an underlying tensor product grid of B-splines with trainable vertex locations. By taking a trainable convex combination of these shape functions, we arrive at more complex geometries. 
    Noting that B-splines form a partition of unity, and that partitions of unity are closed under convex combination, this process provides a trainable partition of unity which may be integrated exactly via a pull-back onto the fine grid.
    For purposes of illustration, the underlying tensor product is shown to be uniform in the figure, but they are allowed to shift in the general case.
    }
    \label{fig:convex-comb}
\end{figure}

We construct finite element spaces consisting of the $0$\textsuperscript{th}-, $1$\textsuperscript{st}- and $2$\textsuperscript{nd}-order Whitney forms from $\phi_i$:
\begin{equation} \label{eqn:whitneyspaces}
    \begin{aligned}
V^0 &:= \operatorname{span}\left\{\phi_i \mid 1 \le i \le N \right\}, \\
V^1 &:= \operatorname{span}\left\{ \phi_i \nabla \phi_j - \phi_j \nabla \phi_i \mid 1 \le i,j \le N \right\}, \\
V^2 &:=  \operatorname{span}\left\{ \phi_i \nabla \phi_j \times \nabla \phi_k - \phi_j \nabla \phi_i \times \nabla \phi_k - \phi_k \nabla \phi_j \times \nabla \phi_i \mid  1 \le i,j,k \le N \right\},
\end{aligned}
\end{equation} 
adopting the notation $\psi_{j_{1},\dots,j_K} \in V^{k-1}$,  to identify elements of spaces by their constituent 0-forms (e.g. $\psi_{ij} \in V^1$). 
As shown in \cite{actoradata}, the tensor used to parameterize convex combinations of B-splines may be manipulated to obtain modifications of these spaces with zero trace
\begin{equation}\label{eqn:whitney-bc}
    V^k_0 := \left\{ u \in V^k \,\big\vert\, u|_{\partial \omega} = 0\right\}.
\end{equation}

Consider now the variational form of divergence $(q,\nabla\cdot \mat {u})$ and curl $(\mat {v},\nabla \times \mat{w})$, where $q\in V^0$, $\mat{u}\in V^1_0$, $\mat{v} \in V^1$, and $\mat{w} \in V^2_0$. After integration by parts, Whitney forms induce the following discrete vector calculus operators \cite[\S 3]{actoradata}
\begin{equation}\label{eqn:discretederivatives}
    \begin{aligned}
(\mathsf{DIV})_{i, (ab)} &:= (\psi_{ab}, -\nabla \psi_i) =  \sum_{j \neq i}(\psi_{ab}, \psi_{ij}), \\
(\mathsf{CURL})_{(ij), (abc)} &:=(\psi_{abc}, \nabla \times \psi_{ij}) = 2 \sum_{k \neq i,j} (\psi_{abc}. \psi_{ijk}).
\end{aligned}
\end{equation}

These discrete exterior derivatives maintain a powerful connection to the graph exterior calculus from combinatorial Hodge theory. 
Consider a complete graph $\mathcal{G} = (\mathcal{V}, \mathcal{E})$ with the vertex set $\mathcal{V}$, edge set $\mathcal{E}$, and higher-order $k$-cliques denoted by the oriented tuples $(i_1,\dots,i_k)$.
The standard $k$\textsuperscript{th}-order coboundary operator $\delta_k$ is simply associated with the oriented incidence matrix between $k+1$- and $k$-cliques. 
Specifically, the graph gradient $\delta_0$ and graph curl $\delta_1$ are defined by
\begin{align*}
    (\delta_0 u)_{ij} &= u_j - u_i \\
    (\delta_1 u)_{ijk} &= u_{ij} + u_{jk} + u_{ki},
\end{align*}
where $u_i$ denote a scalar value associated with the node $i$, $u_{ij} = -u_{ji}$ denotes a scalar associated with the edge $(i,j) \in \mathcal{E}$, and $u_{ijk}$ a value associated with the 3-cliques (e.g. faces) which is anti-symmetric with respect to the index ordering 
\begin{align*}
    u_{ijk} = -u_{ikj} = -u_{jik} = -u_{kji} = u_{kij} = u_{jki}.
\end{align*}

The adjoint of coboundary operators induces the so-called codifferential operators, which in this setting provide definitions of graph divergence and curl:
\begin{equation}
    \begin{aligned}
        (DIV \, u)_i  &:= (\delta_0^T u)_i = \sum_{j \neq i} u_{ij},\\
    (CURL \, u)_{ij}&:= (\delta_1^T u)_i= \sum_{k \neq i,j} u_{ijk}.
\end{aligned}
\end{equation}
These graph operators have a number of properties mimicking the familiar vector calculus, but follow only from the topological properties of graphs. 
For example, the \textit{exact sequence} property $DIV \circ CURL = 0$ discretely parallels $\nabla\cdot \nabla \times= 0$, and conservation structure is reflected in $DIV$ calculating the sum of anti-symmetric generalized fluxes. 

The connection between the parameterized Whitney form space and the combinatorial Hodge theory follows by rewriting \cref{eqn:discretederivatives} as
\begin{equation*}
	\mathsf{DIV} = DIV \, \mathbf{M}_1, \quad \mathsf{CURL} = CURL \, \mathbf{M}_2
\end{equation*} 
where $(\mathbf{M}_1)_{(ij),(ab)} = (\psi_{ab}, \psi_{ij})$ and ($\mathbf{M}_2)_{(ijk), (abc)}=(\psi_{abc}, \psi_{ijk})$ are mass matrices associated with the finite element spaces $V^1$ and $V^2$, respectively. 
Therefore, we see that the geometry of the PPOUs implicitly induces a weighting on the graph exterior calculus, with the boundaries of learned partitions inducing a topology associated with conservation structure. 

We may finally revisit the original task of identifying a model of the form \cref{eq:localmortargeneralcase}. 
Let the Whitney forms associated with subdomain $\Omega_i$ be $V^0(\Omega_i)$ and $V^1(\Omega_i)$ by taking $\omega = \Omega_i$. 
Mirroring \cref{eq:localmortargeneralcase}, the model on each individual subdomain is equivalent to the following variational problem: find $(p_i,\mat {u}_i) \in V^0(\Omega_i) \times V^1(\Omega_i)$ such that for all $(w_i, \mat v_i) \in V^0_0(\Omega_i) \times V^1(\Omega_i) $,
\begin{align*}
    (\mat u_i,\mat v_i) - (h(p_i;\theta_i), \mat v_i) &= 0 \\
    (\mat u_i, \nabla w_i) &=(f_i, w_i) 
\end{align*}
with Dirichlet boundary condition $p_i = g_i$ on $\partial \Omega$, which is enforced by using a standard lift.

Following the theory laid out in \cite{trask2022enforcing}, we could assume the unknown fluxes take the form of a nonlinear perturbation of a diffusive flux while maintaining a tractable stability analysis, e.g. 
\begin{align*}
    h(p_i;\theta_i) = \nabla p_i + N[p_i;\theta_i]
\end{align*}
However in the current work, we will consider only the linear case ($N[p_i;\theta_i]=0$).
In this setting the Whitney forms will identify the geometry and properties associated with material heterogeneities under an assumed diffusion process, providing the following variational problem on each element.
\begin{equation}\label{eqn:early-local-element}
    \begin{aligned}
    (\mat u_i,\mat v_i) - (\nabla p_i, \mat v_i) &= 0 \\
    (\mat u_i, \nabla w_i) &=(f_i, w_i).
\end{aligned}
\end{equation}

Finally we substitute in the discrete exterior derivatives associated with the PPOUs to obtain a discrete parametric model, posing the following equality constrained optimization problem to calibrate the POU geometry to data,
\begin{equation}\label{eqn:h1loss} \begin{split}
\min_{W, \bm{B}_0, \bm{B}_1, \bm{D}_0 \bm{D}_1}\,
& \norm{ p_{\text{data}} - \sum_i \hat{p}_i \psi_i}_2^2
+ \alpha^2 \norm{ F_\text{data} - \sum_{ij} \hat{F}_{ij} \psi_{ij} }_2^2 \\
\text{such that }&
\begin{bmatrix} {\mathbf{M}_1} & - {\mathbf{M}_1} \mathbf{D}_1^{-1} \delta_0 \mathbf{D}_0 \\ - \mathbf{B}_0^{-1} \delta_0^T \mathbf{B}_1 {\mathbf{M}_1} & \mathbf{0} \end{bmatrix} \begin{bmatrix} \widehat{\mathbf{F}} \\ \widehat{\mathbf{p}} \end{bmatrix} = \begin{bmatrix} \mathbf{b}_D \\ - \mathbf{b}_f \end{bmatrix} \\
\end{split}, \end{equation}
where $\mathbf{B_k}$ and $\mathbf{D}_k$ are diagonal matrices with trainable positive coefficients, $\mathbf{b}_D, \mathbf{b}_f$ the terms arising from the Dirichlet boundary condition and forcing term respectively, $\alpha$ a normalization parameter, and $W$ the remaining weights associated with the POUs such as the location of knots and the convex combination tensor. 
As shown in \cite{actoradata}, $\mathbf{B_k}$ and $\mathbf{D}_k$ infer metric information from data without impacting the topological structure of the model.
For further details regarding the specific construction of POUs we refer to \cite{actoradata}.

\begin{remark}
    The Whitney form construction supports a number of theoretical constructions: a Hodge decomposition, Poincare inequality, a corresponding Lax-Milgram theory, a well-posedness theory for certain nonlinear elliptic problems, and discrete preservation of exact sequence properties which exactly preserve conservation structure. When we use the Whitney form elements to construct the subdomain spaces $V_i$ in the mortar method in the following section, we aim to carefully construct the mortar space so that this structure is not lost at the global level.
\end{remark}

\section{Mortar Method}\label{sec:mortarsec}
After the local models are trained, we seek to construct a mortar method which is flexible enough to couple FEEC elements on the different subdomains together. 
Note that since the fine-scale knots are able to move during pre-training, the mortar is necessarily non-conforming, with possible ``hanging'' mortar nodes which do not coincide with the neighboring local element nodes; this necessitates an analysis of stability associated with projecting between local and mortar spaces.
Furthermore, we would like the mortar method to preserve the conservation and stability properties outlined in the introduction ({\bf R1}, {\bf R2}). 

As discussed in \cref{sec:whitneysec}, we assume that our data $\{(\mat u(x_k), p(x_k)), g_k\}_{k=0}^N$ (with $x_k \in \Omega$ sampled randomly) satisfy the following variational equation: seek solution $(\mat u, p) \in (L^2(\Omega)^2, H^1_g(\Omega))$ such that
\begin{equation}\label{eqn:primal-big}
\begin{aligned}
        	(\mat u, \mat v) - (K\nabla p, \mat v) &= 0, &\forall \mat v \in L^2(\Omega)^2  \\
	(\mat u, \nabla w)  &= (f, w), &\forall w \in H^1_0(\Omega)
\end{aligned}
\end{equation}
where $H^1$ is the standard Sobolev space and $H^1_g(\Omega) = \{u \in H^1(\Omega) \mid u|_{\partial \Omega} = g_k\}$ \cite{braess2007finite}, and the tensor $K \in L^\infty$ is a positive-definite matrix.
Finally, we assume the problem is of at least $p \in H^{3/2}(\Omega)$ regularity, which arises naturally if, for example, $f \in L^2(\Omega), g \in H^{3/2}(\partial \Omega)$ with Lipschitz coefficients $K$ and $\Omega$ is convex \cite{grisvard2011elliptic}.
We will see in our numerical results that the above regularity result is a sufficient condition for the error analysis, and not a necessary one. 

Let $\Omega$ be divided into $n$ non-overlapping, polygonal subdomain blocks $\Omega_i$ of similar aspect ratios.
Let $\Gamma_{i}$ be the edges of $\Omega_i$, $\Gamma = \cup_i \Gamma_i$ the set of all boundaries of the subdomains (including those intersecting $\partial \Omega$), and let $\Gamma_{ij} = \Gamma_i \cap \Gamma_j$ for all $i, j$ be the boundary between two adjacent subdomains. 
See \cref{fig:gamma} for an illustrative figure. 

Define 
\begin{align}\label{eqn:lambda-space-def}
	\Lambda := \{v \in L^2(\Gamma) \mid \exists u \in H^1(\Omega), u|_{\Gamma} = v \}
\end{align}
as the space of $L^2$ functions on the interfaces which are the traces of $H^1$ functions, and the subspaces 
\begin{align*}
    \Lambda_0 &:= \{\lambda \in \Lambda \mid \lambda|_{\partial \Omega} = 0\}, \\
    \Lambda_g &:= \{\lambda \in \Lambda \mid \lambda|_{\partial \Omega} = g\}.
\end{align*}
Note that since $\Lambda$ consists of the trace of $H^1$ functions, we may endow $\Lambda$ with the $H^{1/2}$ norm on $\Gamma$. 

\begin{figure}[tb]
    \centering
\begin{tikzpicture}[scale=2.5]
\draw[dashed] (0, 0) grid (2, 2);
\draw[thick] (0, 0) -- (2, 0) -- (2, 2) -- (0, 2) -- cycle;
\draw[line width=5pt, gray, opacity=.4] (0, 1) -- (1, 1) -- (1, 0) -- (0, 0) -- cycle;
\draw (.5, .5) node {$\Omega_1$};
\draw (1.5, .5) node {$\Omega_2$};
\draw (.5, 1.5) node {$\Omega_3$};
\draw (1.6, 1.5) node {$\Omega_4$};
\draw[ultra thin, decoration={brace,mirror,raise=1pt},decorate]
  (1,1) -- node[scale=.4, right=12] {$\Gamma_{3,4}$} (1, 2);
\end{tikzpicture}
    \caption{Figure of a square domain $\Omega$ divided into four subdomains.
    The edge $\Gamma_{3,4}$ is denoted explicitly and the highlighted boundary is $\Gamma_1$.}
    \label{fig:gamma}
\end{figure}
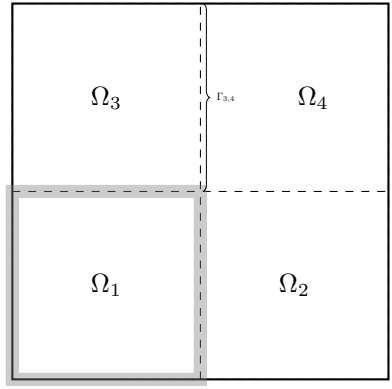

\subsection{Stability analysis for continuous case}
Before proceeding to the model discovery problem and the discrete, we first consider smooth solutions coming from solutions from diffusion problem to guide the design of a suitable mortar method.
It is straightforward to decompose \cref{eqn:primal-big} into problems on the subdomains $\{\Omega_i\}_{i=1}^n$ by introducing a mortar representing the pressure on the space $\Lambda$:
\begin{lemma}\label{lem:cont-subdomains}
	For $1 \le i \le n$, let $(\mat u_i, p_i, \lambda) \in (L^2(\Omega_i)^2, H^1(\Omega_i), \Lambda_g)$ such that

\begin{equation}\label{eqn:original-subdomain}
\begin{aligned}
    	(\mat u_i, \mat v_i)_{\Omega_i} - (K\nabla p_i, \mat v_i)_{\Omega_i} &= 0, &\forall \mat v_i \in L^2(\Omega_i)^2 \\
	(\mat u_i, \nabla w_i)_{\Omega_i}  &= (f, w_i)_{\Omega_i}, &\forall w_i \in H^1_{0}(\Omega_i)
\end{aligned}
\end{equation}	
	with continuity of state and flux enforced via the boundary condition $p_i|_{\Gamma_i} = \lambda|_{\Gamma_i}$
	and weak flux continuity condition
\begin{align}\label{eqn:basic-h10}
	\sum_{i=1}^n (\mat u_i, \nabla w)_{\Omega_i}  &= (f, w), &\forall w \in H^1_{0}(\Omega).
\end{align}
	Then $\mat u = \sum_{i=1}^n \mat u_i \in L^2(\Omega)^2, p = \sum_{i=1}^n p_i \in H^1_g(\Omega)$ solves \cref{eqn:primal-big}.
\end{lemma}
\begin{proof}
    The existence of functions $(\mat u_i, p_i)$ and $\lambda$ comes trivially by restricting the solution from \cref{eqn:primal-big} to the individual subdomains and mortar space. 

    To see that \cref{eqn:original-subdomain,eqn:basic-h10} implies \cref{eqn:primal-big}, 
    we note that $L^2(\Omega)^2 = \bigoplus_{i=1}^n L^2(\Omega_i)^2$, and thus by summing the first equation of \cref{eqn:original-subdomain} and choosing $\mat {v}_i = \mat {v}|_{\Omega_i}$ as test functions, we have
    \begin{align*}
        \left(\sum_{i=1}^n \mat u_i, \mat v\right) - \left(K\nabla \sum_{i=1}^n p_i, \mat v\right) &= 0, &\forall \mat v \in L^2(\Omega)^2
    \end{align*}
    with $\sum_{i=1}^n p_i \in H^1_g(\Omega)$ since continuity is enforced with $\lambda$.
    As for the test functions arising in $w \in H^1_0(\Omega)$, we simply decompose $w$ into $\sum_{i=1}^n w_i + w_0$ where $w_i \in H_0^1(\Omega_i)$ for $1 \le i \le n$ and $w_0 := w - \sum_{i=1}^n w_i$, so that the summation of the second equation of \cref{eqn:original-subdomain} and \cref{eqn:basic-h10} gives us the desired result.
\end{proof}

The condition \cref{eqn:basic-h10} can be simplified. 
Consider the space $H^\gamma(\Omega)$ satisfying the decomposition 
\begin{align}\label{eqn:decomposition}
	H^1(\Omega) &= H_0^1(\Omega_1) \oplus \cdots \oplus H_0^1(\Omega_n) \oplus H^\gamma(\Omega)
\end{align}
with $H^\gamma \perp H_0^1(\Omega_i)$ relative to the $H^1$ norm for each $i$. 
Then, using to \cref{eqn:original-subdomain} and \cref{eqn:decomposition},  \cref{eqn:basic-h10} can be rewritten as
\begin{align}\label{eqn:good-condition}
	\sum_{i=1}^n (\mat u_i, \nabla w)  &= (f, w), &\forall w \in H^\gamma_0(\Omega)
\end{align}
where $H_0^\gamma := \{u \in H^\gamma(\Omega) \mid u|_{\partial \Omega} = 0 \}$.
We also define the subset $H_g^\gamma := \{u \in H^\gamma(\Omega) \mid u|_{\partial \Omega} = g \}$.
The space $H^\gamma$ corresponds to a minimal energy extension \cite{toselli2004domain} as the following lemma shows:
\begin{lemma}
    For all $u \in H^1(\Omega)$, there exists a unique decomposition $u = u_\gamma + \sum_{i=1}^n u_i$ such that $u_\gamma \in H^\gamma(\Omega), u_i \in H_0^1(\Omega_i)$. 
    Furthermore, one has 
    \begin{align*}
        \norm{u_\gamma}_{H^1(\Omega)} = \inf_{v \in H^1(\Omega), v|_{\Gamma} = u} \norm{v}_{H^1(\Omega)} \simeq \sum_{i=1}^n \norm{u_\gamma}_{H^{1/2}(\Gamma_i)}.
    \end{align*}
\end{lemma}
\begin{proof}
    Given $u$, consider $u_I \in H_0^1(\Omega_1) \oplus \cdots \oplus H_0^1(\Omega_n) $ such that for $1 \le i\le n$,
    \begin{align*}
        (u_I, v_i)_{H^1(\Omega_i)} = (u, v_i)_{H^1(\Omega_i)}, \qquad \forall v_i \in H_0^1(\Omega_i).
    \end{align*}
    Then the decomposition is simply $u = \sum_{i=1}^n u_I|_{\Omega_i} + u_\gamma$ where $u_\gamma = u - u_I$.
    The orthogonality is enforced since, for all $w_i$ in $H_0^1(\Omega_i)$ and $1 \le i \le n$,
    \begin{align*}
        (u_\gamma, w_i)_{H^1(\Omega)} = (u - u_I, w_i)_{H^1(\Omega_i)} = (u, w_i)_{H^1(\Omega_i)} - (u_I, w_i)_{H^1(\Omega_i)} = 0. 
    \end{align*}
    
    As for the minimal condition, let $v = u_\gamma + \sum_{i=1}^n v_i$ with $v_i \in H_0^1(\Omega_i)$ arbitrary, then by orthogonality
    \begin{align*}
         \norm{v}_{H^1(\Omega)}^2 =  \norm{u_\gamma}_{H^1(\Omega)}^2 + \norm{\sum_{i=1}^n v_i}^2_{H^1(\Omega)} \ge \norm{u_\gamma}_{H^1(\Omega)}^2
    \end{align*}
    and the $H^{1/2}$ equivalence is well known \cite{bertoluzza2000wavelet, cowsar1995balancing}. 
\end{proof}

With the above decomposition, we can further reduce \cref{eqn:primal-big} to be a variational problem only on $H^\gamma$ and $\Lambda$. 
Let $\lambda, \mu \in H^\gamma$, define the bilinear form and linear functional
\begin{align}\label{eqn:bilinear-form-1}
 	b(\lambda, \mu) &= \sum_{i=1}^n (\mat u^*(\lambda), \nabla \mu)_{\Omega_i}
 \end{align}
 and
 \begin{align}\label{eqn:linear-functional-1}
	L( \mu) =(f, \mu)_\Omega - \sum_{i=1}^n (\bar{\mat u}, \nabla \mu)_{\Omega_i}
\end{align}
where $(\mat u^*( \lambda), p^*( \lambda)) \in (L^2(\Omega)^2, H^1(\Omega))$ solves the local problems, for $1 \le i \le n$,
\begin{equation}
\begin{aligned}\label{eqn:star}
	(\mat u^*( \lambda), \mat v) - (K\nabla p^*( \lambda), \mat v) &= 0, &\forall \mat v \in L^2(\Omega_i)^2 \\
	(\mat u^*( \lambda), \nabla w)  &=0,  &\forall w \in H^1_{0}(\Omega_i)
\end{aligned}
\end{equation}
with boundary condition $p^*(\lambda)|_{\Gamma_i} = \lambda|_{\Gamma_i}$, and
where $(\bar{\mat u}, \bar p) \in (L(\Omega)^2, H_0^1(\Omega))$ solves, for $1 \le i \le n$, 
\begin{equation}
\begin{aligned}\label{eqn:bar}
	(\bar{\mat u}, \mat v) - (K\nabla \bar p, \mat v) &= 0, &\forall \mat v \in L^2(\Omega_i)^2 \\
	(\bar{\mat u}, \nabla w)  &= (f, w) , &\forall w \in H^1_0(\Omega_i)
\end{aligned}
\end{equation}
with boundary condition $\bar p|_{\Gamma_i} = 0$.
The bilinear form and linear functional closely resemble those of the $H(\operatorname{div})$ case from \cite{arbogast2007multiscale,arbogast2000mixed}. 
Note that the the problems \cref{eqn:star,eqn:bar} above are local in nature and can be solved in parallel. 

The following lemma shows that one can recover the original variational equations by working with the above bilinear form:
\begin{lemma}\label{lem:cont-sol}
    Let $\lambda \in H^\gamma_g$ be the solution to the variational equation, 
	\begin{align}\label{eqn:cont-bilinear-form}
		b(\lambda, \mu) = L(\mu), \qquad \forall \mu \in H^\gamma_0
	\end{align}
    then $\mat u := \mat u^*(\lambda) + \bar{\mat u}, p := p^*(\lambda) + \bar p$ is the solution to \cref{eqn:primal-big}. 
\end{lemma}
\begin{proof}
	Summing \cref{eqn:star,eqn:bar} results in
	\begin{align*}
			({\mat u}, \mat v) - (K\nabla p, \mat v) &= 0, &\forall \mat v \in  L^2(\Omega_i)^2  \\
	({\mat u}, \nabla w)  &= (f, w) , &\forall w \in H^1_0(\Omega_i)
	\end{align*}
	with $p|_{\Gamma} = \lambda$ for each $1 \le i \le n$. 

	It remains to check \cref{eqn:good-condition}, but this is simply because if \cref{eqn:cont-bilinear-form} holds, then 
	\begin{align*}
		\sum_{i=1}^n (\mat u, \nabla \mu)_{\Omega_i} = (f, \mu)
	\end{align*}
	for all $ \mu \in H^\gamma_0$ and the results follows from \cref{lem:cont-subdomains} and \cref{eqn:good-condition}.
\end{proof}

Finally, we note that the variational equation \cref{eqn:cont-bilinear-form} is well-defined as the bilinear form is coercive as shown in the following lemma, whose proof is delayed until the appendix: 
\begin{lemma}
    \label{lem:continuous-coer}
	The bilinear form \cref{eqn:bilinear-form-1} is symmetric and coercive on $\Lambda_0$.
\end{lemma}

\subsection{Discretized Case}
\begin{figure}[tb]
    \centering

    \includegraphics[width=.6\textwidth]{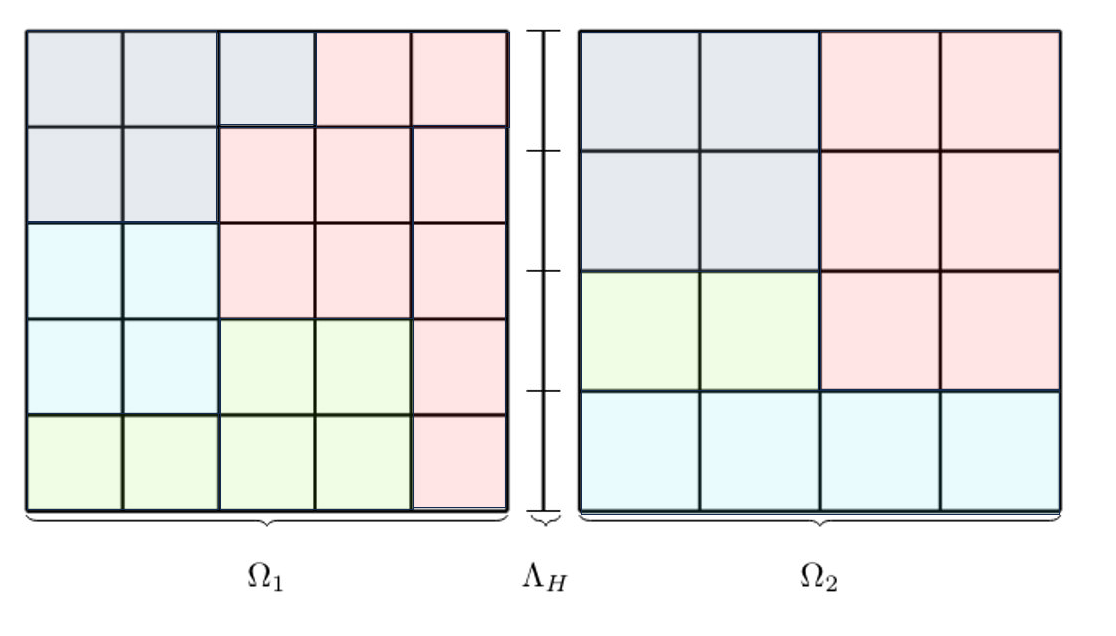}
    \caption{Sketch of a 4 element mortar $\Lambda_H$ and its two adjacent subdomains.
    The colors on the subdomains represent the PPOUs constructed as convex combinations of a fine-scale B-splines.
    We note that while the mortar matches the fine-scale nodes on $\Omega_2$, it is disjoint from $\Omega_1$, requiring analysis of a remap/projection between the two meshes. 
    Because the FEEC fine-scale nodes on $\Omega_i$ evolve during training, they will generally not coincide with mortar nodes.
    }
    \label{fig:mortar-simple}
\end{figure}

The discrete case is more technical, since both the spaces and the bilinear form are discretized as \cref{eqn:star} and \cref{eqn:bar} cannot be solved exactly. Further, care must be taken to treat the non-conforming grids that emerge naturally as nodes between adjacent subdomains evolve.

In what follows, the subscripts $h, H$ will denote a discretized version of a continuous space.
On each subdomain $\Omega_i$, let $W_{hi}, \mat V_{hi}$ be the discretized versions of $H^1(\Omega_i), L^2(\Omega_i)^2$ respectively.
We require the standard inf-sup compatibility between $W_{hi}, \mat V_{hi}$, which in this case is simply the condition $\nabla W_{hi} \subseteq \mat V_{hi}$ \cite{arnold2018finite, braess2007finite}.
In particular, we can choose $W_{hi}$ and $\mat V_{hi}$ to be the spaces $V^0$ and $V^1$ defined in \cref{eqn:whitneyspaces} in the case of FEEC elements; by construction then we have $\nabla V^0 = V^1$.
We will also use the case where local elements are taken to be traditional finite elements to show convergence; in this case we will consider $V^0$ and $V^1$ as continuous $\mathbb{Q}_1$ and lowest-order Nedelec elements, respectively.
Finally, we let $W_{hi, 0}$ be the subspace with homogeneous Dirichlet boundary condition (e.g. \cref{eqn:whitney-bc}).


On the interfaces, we choose $\Lambda_H \subset \Lambda$ to be the space of continuous, piecewise linear functions.
Let $\Lambda_{H, 0} := \{ \mu_H \in \Lambda_H \mid \mu_H |_{\partial \Omega} = 0\}\subset \Lambda_0$ and similarly let $\Lambda_{H, g} \subset \Lambda_g$ be the subset whereby the boundary is equal to $g$.
We allow the discretization between two subdomains to be different and also allow for the mortars to be non-matching.
See \cref{fig:mortar-simple} for a simplified figure where there are non-matching tensor-product grids. 

We define a projection for each subdomain $Q_i: \Lambda_H \to W_{hi}$ such that for all $\lambda_H \in \Lambda_H$
\begin{align}\label{eqn:projection-h12}
 	(\lambda_H - Q_i \lambda_H , p_h)_{L^{2}(\Gamma_i)} = 0, \qquad \forall p_h \in W_{hi}|_{\Gamma_i}
\end{align}
and 
\begin{align*}
    (Q_i\lambda_H,  w)_{H^1(\Omega_i)} &= 0, \qquad \forall w \in W_{hi, 0}.
\end{align*}
The first condition simply defines the boundary of $Q_i \lambda_H$ using the $L^2$-projection\footnote{We found in our numerical examples that using the interpolant suffices, however we will carry out the analysis using the projection.} on $\Omega_i$ while the second condition means that it is the discrete harmonic extension into $\Omega_i$ given the boundary $Q_i\lambda_H$ on $\Gamma_i$ \cite{toselli2004domain}. 
We note that in general, the projections to the left and right of that interface are different since the discretization can be different on either sides as can be seen in \cref{fig:mortar-simple}.

With the above in hand, we can define the discretized bilinear operator and linear functional similar to \cref{eqn:bilinear-form-1,eqn:linear-functional-1}. For $\lambda_H, \mu_H \in \Lambda_H$, let
\begin{align}\label{eqn:disc-bilinear}
	b_h( \lambda_H,  \mu_H) := \sum_{i=1}^n (\mat u_h^*(Q_i \lambda_H), \nabla (Q_i\mu_H))_{\Omega_i}
\end{align}
and
\begin{align}
	L_h( \mu_H) :=  \sum_{i=1}^n (f, Q_i\mu_H)_{\Omega_i} - (\bar{\mat u}_h, \nabla (Q_i\mu_H))_{\Omega_i}
\end{align}
where $p_h^*(Q_i \lambda_H) \in \oplus_{i=1}^n W_{hi}, \mat u^*_h(Q_i \lambda_H) \in \oplus_{i=1}^n \mat V_{hi}$ satisfies, for $1 \le i \le n$, 
\begin{align}\label{eqn:star-h}
	(\mat u^*_h(Q_i \lambda_H), \mat v_h) - (K\nabla p_h^*(Q_i \lambda_H), \mat v_h) &= 0, &\forall \mat v_h \in \mat V_{hi} \\
	(\mat u_h^*(Q_i \lambda_H), \nabla w_h)  &=0,  &\forall w_h \in W_{hi, 0}
\end{align}
with $p^*_h(Q_i\lambda_H) = Q_i\lambda_H$ on $\Gamma_i$, 
and $\bar p_h \in \oplus_{i=1}^n W_{hi}, \bar{\mat u}_h \in \oplus_{i=1}^n V_{hi}$ satisfying
\begin{align}\label{eqn:bar-h}
	(\bar{\mat u}_h, \mat v_h) - (K\nabla \bar p_h, \mat v_h) &= 0, &\forall \mat v_h \in \mat V_{hi}  \\
	(\bar{\mat u}_h, \nabla w_h)  &= (f, w_h) , &\forall w_h \in W_{hi, 0}
\end{align}
with $\bar p_h = 0$ on $\Gamma_i$.
As before, the above problems are defined locally and can be solved in parallel.

We state the discrete variational equation as follows. Find $\lambda_H \in \Lambda_{H, g}$ such that
\begin{align}\label{eqn:disc-bilinear-form}
    b_h(\lambda_H, \mu_H) = L_h(\mu_H), \qquad \forall \mu_H \in \Lambda_{H, 0}.
\end{align}
The well-posedness of the variational form can be deduced from Lax-Milgram if the coercivity condition
\begin{align}\label{eqn:discrete-coercive}
    b_h(\lambda_H, \lambda_H) \ge \alpha \sum_{i=1}^n\norm{\lambda_H}_{H^{1/2}(\Gamma_i)}^2
\end{align}
is true.
The coercivity condition \cref{eqn:discrete-coercive} will require two assumptions which excludes pathological discretizations: 
\begin{enumerate}
    \item Assumption 1 (injectivity): for all $\lambda_H \in \Lambda_H$, there exists a constant $C$ such that
    \begin{align}\label{eqn:injective-assumption}
	\sum_{i=1}^n \norm{Q_i\lambda_H}_{H^{1/2}(\Gamma_i)} \ge C \sum_{i=1}^n \norm{\lambda_H}_{H^{1/2}(\Gamma_i)}
\end{align}
meaning we have unisolvency when projecting from the mortar space onto the local subdomains. 

    \item Assumption 2 (strengthened triangle inequality): for each shared edge $\Gamma_{ij}$ and for all $\lambda_H \in \Lambda_H$, that 
\begin{align}\label{eqn:jump-assumption}
	\frac{C_p}{\abs{\Gamma_{ij}}}\norm{Q_i \lambda_H - Q_j \lambda_H}_{L^2(\Gamma_{ij})}^2 \le \frac{1}{2} (\norm{Q_i \lambda_H}_{H^{1/2}(\Gamma_{ij})}^2+ \norm{Q_j \lambda_H}_{H^{1/2}(\Gamma_{ij})}^2)
\end{align}
where $C_p$ is the Poincare constant arising in \cite[(1.3)]{brenner2003poincare} and $\abs{\Gamma_{ij}}$ is the length of the shared edge.
The condition means two adjacent subdomains cannot have too large of a difference in their discretization parameter.
In particular, if two adjacent subdomains have the same, symmetric discretization parameters, then the left side of \cref{eqn:jump-assumption} is trivially zero. 
\end{enumerate}

In the case of data-driven elements, extra care must be paid to Assumption 1 since a training procedure might move the fine-scale nodes such that unisolvency is lost.
However, this can be circumvented by either placing restrictions on the movement of the nodes, or, as in some of our numerical examples, using a very coarse mortar space. 

With the above assumptions, we can now state the stability result: 
\begin{lemma}\label{lem:dis-stability}
    With the above two assumptions, the discretized bilinear form \cref{eqn:disc-bilinear} is coercive (e.g. \cref{eqn:discrete-coercive}) over $\Lambda_{H, 0}$.
\end{lemma}
The proof of the above lemma is technical and is delayed to the appendix.

\cref{lem:dis-stability} means that one is allowed to apply Strang's second lemma to obtain error estimates. 
We assume that an a priori estimate exists: let $\delta$ be a constant such that the discrete approximations on each subdomain $1 \le i \le n$ satisfy
	\begin{align}\label{eqn:approximiation-assumption}
            \norm{p^*( \lambda) - p_h^*(Q_i \lambda)}_{\Omega_i} &\le \delta, &\qquad \norm{\mat u^*( \lambda) - \mat u_h^*(Q_i \lambda)}_{\Omega_i} &\le \delta \\
            \norm{\bar p - \bar p_h}_{\Omega_i} &\le \delta      &\qquad \norm{\bar{\mat u} - \bar{\mat u}_h}_{\Omega_i} &\le \delta 
  	\end{align}
for all $\lambda \in \Lambda$.
The constant $\delta$ corresponds to the ability of the local solvers to solve for $p^*, \mat u^*$ accurately for an arbitrary mortar.
In the case where standard FEM is used on the subdomain, then $\delta$ can be replaced with the respective a priori estimate whereas for the DDEC methods, this corresponds to an optimization threshold.

We can now state a simple convergence guarantee {\bf R2} on the mortar space, whose proof is delayed until the appendix:
\begin{theorem}\label{thm:err-est}
    Suppose the solution to the \cref{eqn:primal-big} is such that $p \in H^{2}(\Omega)$ with homogeneous Dirichlet boundary condition. Then there exists a constant $C$ independent of $\lambda^*$ such that
	\begin{align}\label{eqn:is-this-correct}
		\sum_{i=1}^n\norm{ \lambda^* -  \lambda^*_H}_{H^{1/2}(\Gamma_i)} &\le C n \abs{p}_{H^2(\Omega)} (H + h + \delta)
	\end{align}
    where $\lambda^*$ is the true solution to \cref{eqn:cont-bilinear-form}, and $\lambda_H^*$ is the solution to \cref{eqn:disc-bilinear-form}, and $H, h$ are the maximal mesh sizes on $\Lambda_H$ and the boundary of the subdomains $\Omega_i$, respectively.
\end{theorem}
\begin{remark}
    As mentioned, the constant $\delta$ associated with \cref{eqn:approximiation-assumption} corresponds to the accuracy of the local solvers while the $h$ term relates to the accuracy of projecting the mortar to the local subdomains using \cref{eqn:projection-h12}, though in general we can assume that $h < H$. 
    We also note that \cref{eqn:is-this-correct} implies that a combination of refinement of both the local solvers and the mortar space is needed to obtain convergence. 
\end{remark}
Finally, we can easily bound the error on the pressure and velocity explicitly.
\begin{lemma}\label{lem:error-pre-vel}
	With the same assumptions and constants as in \cref{thm:err-est}, there exists a constant $C$ independent of $\mat u$ and $p$ such that
 	\begin{align*}
            \norm{p - p_h}_{\Omega} + \norm{\mat u - \mat u_h}_{\Omega} &\le  C n \abs{p}_{H^2(\Omega)} (H + h + \delta)
	\end{align*}
 where $\mat u, p$ are the true solutions arising from \cref{eqn:original-subdomain} and $\mat u_h = \sum_{i=1}^n \mat u_h^*(Q_i \lambda_H^*) + \bar{\mat u}_h$, $p_h = \sum_{i=1}^n p_h^*(Q_i \lambda_H^*) + \bar{p}_h$.
\end{lemma}
\begin{proof}
    By \cref{lem:cont-sol}, we have 
\begin{align*}
	\norm{\mat u - \mat u_h}_{\Omega} &\le \sum_{i=1}^n \norm{\mat u^*( \lambda^*) - \mat u_h^*( Q_i\lambda_H^*)}_{\Omega_i} + \norm{\bar{\mat u} - \bar{\mat u}_h}_{\Omega_i}. 
\end{align*}
The latter term on the right hand side is bounded by $\delta$ by assumption. 
Thus the result follows by
\begin{align*}
	 \sum_{i=1}^n \norm{\mat u^*( \lambda^*) - \mat u_h^*( Q_i\lambda_H^*)}_{\Omega_i} &\le \sum_{i=1}^n \norm{\mat u^*( \lambda^*) - \mat u^*( \lambda_H^*)}_{\Omega_i} +  \norm{\mat u^*( \lambda_H^*) - \mat u_h^*(Q_i \lambda_H^*)}_{\Omega_i} \\
	 &\le n\delta +  \sum_{i=1}^n \norm{\mat u^*( \lambda^* -  \lambda^*_H)}_{\Omega_i} \\
	 &\le n\delta +  \sum_{i=1}^n \norm{ \lambda^* -  \lambda^*_H}_{H^{1/2}(\Gamma_i)}
\end{align*}
where we used standard regularity estimates at the last step. 
The same estimates also follow for the pressure and the result follows from applying \cref{thm:err-est}.   
\end{proof}

The above error analysis partially shows that requirement {\bf R2} from the introduction is met, as the total error is indeed controlled by a combination of the local optimization error, and coupling error from the mortars.
However, due to the use of crude bounds on the sum, it is not independent with the number of subdomains, though we will later observe it holds numerically (cf. \cref{sec:subdomain-refinement-examples}). 

\subsection{Data-Driven Elements with Mortar Method}
Classical finite elements such as Nedelec elements can be used for the local solvers in \cref{eqn:star-h,eqn:bar-h} on the subdomains $\Omega_i$ in a straightforward manner (see \cref{sec:example-1} for an example).
However, the true strength of the above mortar method is its ability to interface with the data-driven structure-preserving models discussed in \cref{sec:whitneysec}.
We briefly discuss combining the usage of the Whitney form elements with the mortar method. 

As before, we assume the data is of the form $\{(\mat u(x_k), p(x_k)), g_k\}_{k=0}^N$ with $x_k$ sampled randomly on $\Omega$.
This can either be supplied via physical data or high-fidelity PDE solvers.
Let $M$ be the total number of unique boundary conditions $g_k$ (e.g. $M
=1$ if all data points originate from the same boundary value problem). 
We assume $\Omega$ is divided into subdomains $\Omega_i$. 
As with most data-driven applications, a large number of data points $N$ is needed, however, only one boundary condition $M$ is needed (see \cref{sec:battery-example} for an example with $M = 1$), though more is always better.

The iterative solving process for the mortar \cref{eqn:disc-bilinear-form} involves different Dirichlet boundary conditions $\lambda_H$ being passed into \cref{eqn:star-h}, meaning that the ability for the data-driven Whitney form solvers to be able to correctly respond to different Dirichlet data is important. 
Ideally $M$ is large so that a good sampling of Dirichlet conditions around each $\Omega_i$ is achieved.

In cases where simulations on each $\Omega_i$ is possible, one should perform simulations to obtain responses to a possible mortar boundary conditions. 
In particular, in our numerical examples, we choose to use either nodal functions $\{(1-x)(1-y), x(1-y), (1-x)y, xy\}$ or edge Bernstein polynomials.
The Bernstein polynomials are chosen as they provide a complete basis on $\partial \Omega_i$ and their gradients are very smooth, however other boundary conditions can be chosen. 
We note that these data is usually cheaper to generate since the subdomains are smaller than $\Omega$, and they can be performed in parallel.

However, the ability to perform these simulations on each subdomain is not always possible. 
In this case, a simple approach consisting of taking $g_k := p|_{\Omega_i}$ and the corresponding data points $(\mat u(x_k), p(x_k))$ restricted to each $\Omega_i$ can be done.
While easier, this does lead to higher errors due to undersampling from certain mortar modes. 
Nevertheless, the structure-preserving nature of the data-driven elements ensures adherence to the underlying invariance.

With the data on each $\Omega_i$ chosen, we then solve the minimization problem \cref{eqn:h1loss} giving us fine-scale nodes, and a coarsening to POUs. 
\emph{These data-driven elements are then used as the local solvers for \cref{eqn:star-h,eqn:bar-h}.}
Some care must be exercised to ensure that Assumption 1 is satisfied; the projection from the mortar space onto the local solvers must be unique.
One can mix and match the local solvers, and only use the data-driven elements where the fluxes are unknown and use traditional finite elements elsewhere; see \cref{sec:cylinder} for an example.
Specific details regarding the training process for the numerical examples are given in \cref{sec:appendix}.

\subsection{Neumann Boundary Conditions and Conservation}
We briefly discuss modifications needed to solve the pure Neumann problem $\mat u \cdot \vec n = g$ on $\partial \Omega$, and show that the critical conservation and compatibility property of
\begin{align}
    \label{eqn:conservation}
    \int_{\Omega} f + \int_{\partial \Omega} g = 0
\end{align}
is satisfied by the discrete mortar method. 
Such conservation is exhibited in the FEEC elements also \cite{actoradata}, and thus by showing the mortar method exhibits this behavior as well, requirement {\bf R1} is satisfied. 

The assumed global model is now to find $(\mat u, p) \in (L^2(\Omega)^2, H^1(\Omega))$ satisfying
\begin{equation}
\begin{aligned}\label{eqn:primal-big-neumann}
    (\mat u, \mat v) - (K\nabla p, \mat v) &= 0, &\forall \mat v \in L^2(\Omega)^2 \\
    (\mat u, \nabla w)  &= (f, w) + (g, w)_{\partial \Omega}, &\forall w \in H^1(\Omega)
\end{aligned}
\end{equation}
with the condition that $(p, 1) = 0$ for uniqueness. 

Due to the differences in boundary conditions, a slightly different choice of spaces and decomposition akin to \cref{eqn:decomposition} is needed. 
Define 
\begin{align}
    H^\gamma_B(\Omega) &:= \{u \in H^\gamma \mid u|_{\Gamma_i \setminus \partial \Omega} = 0, \forall 1 \le  i \le n\}
\end{align}
and let $H^\gamma_D$ be such that 
\begin{align}
    H^\gamma = H^\gamma_{D} \oplus H^\gamma_B. 
\end{align}
The space $H^\gamma_B$ is simply the subspace which vanishes on the interior mortar spaces, while $H^\gamma_D$ is its complement. 
Finally, for each $1 \le i \le n$, let 
\begin{align}
    H_D^1(\Omega_i) := \{ u \in H^1(\Omega_i) \mid u|_{\Gamma_i \setminus \partial \Omega}  = 0 \},
\end{align}
the set of $H^1$ functions vanishing only on the interior boundary. 
Note that all functions in $H^\gamma_B$ can be written as a sum of functions in $H_D^1(\Omega_i)$.
Hence, a new decomposition can be written 
\begin{align}
    H^1(\Omega) = H_D^1(\Omega_1) \oplus \cdots \oplus H_D^1(\Omega_n) \oplus H_D^\gamma(\Omega). 
\end{align} 

With the spaces above, we can introduce a mortar that is equivalent to \cref{eqn:primal-big-neumann}, up to a constant: for $1 \le i \le n$, let $(\mat u_i, p_i, \lambda) \in (L^2(\Omega_i)^2, H^1(\Omega_i), H^\gamma_D(\Omega))$ satisfy
\begin{equation}\label{eqn:neumann}
\begin{aligned}
    	(\mat u_i, \mat v_i) - (K\nabla p_i, \mat v_i) &= 0, &\forall \mat v_i \in L^2(\Omega_i)^2 \\
	(\mat u_i, \nabla w_i)  &= (f, w_i)+ (g, w_i)_{\Gamma_i \cap \partial \Omega}, &\forall w_i \in H^1_D(\Omega_i)
\end{aligned}
\end{equation}
with the boundary condition that $p|_{\Gamma_i \setminus \partial \Omega} = \lambda|_{\Gamma_i \setminus \partial \Omega}$, and 
\begin{align}\label{eqn:basic-h10-neumann}
	\sum_{i=1}^n (\mat u_i, \nabla w)_{\Omega_i}  &= (f, w)+ (g, w_i)_{\partial \Omega}, &\forall w \in H_D^\gamma(\Omega).
\end{align}
Finally, we can impose $\int_\Gamma \lambda = 0$ for uniqueness. 
The proof is similar to that of \cref{lem:cont-subdomains} and is omitted. 

With the above, it's easy to define the variational problem as before.
Small changes are needed in the bilinear form \cref{eqn:bilinear-form-1} and linear functional \cref{eqn:linear-functional-1}: the definition of $(\mat u^*(\lambda), p^*(\lambda)), (\bar{\mat u}, \bar p)$ should be changed to
\begin{equation}
\begin{aligned}\label{eqn:star-neumann}
	(\mat u^*( \lambda), \mat v) - (K\nabla p^*( \lambda), \mat v) &= 0, &\forall \mat v \in L^2(\Omega_i)^2 \\
	(\mat u^*( \lambda), \nabla w)  &= 0,  &\forall w \in H^1_{D}(\Omega_i)
\end{aligned}
\end{equation}
with boundary conditions $p^*(\lambda)|_{\Gamma_i \setminus \partial \Omega} = \lambda|_{\Gamma_i \setminus \partial \Omega}$, and  
\begin{equation}
\begin{aligned}\label{eqn:bar-neumann}
	(\bar{\mat u}, \mat v) - (K\nabla \bar p, \mat v) &= 0, &\forall \mat v \in L^2(\Omega_i)^2 \\
	(\bar{\mat u}, \nabla w)  &= (f, w) + (g,w)_{\Gamma_i \cap \partial \Omega}, &\forall w \in H^1_D(\Omega_i)
\end{aligned}
\end{equation}
with boundary condition $\bar p|_{\Gamma_i \setminus \partial \Omega} = 0$.
We note that \cref{eqn:star-neumann} and \cref{eqn:bar-neumann} are both well-defined for all subdomains due to the Dirichlet boundary conditions on the mortar space, except for the degenerate case where there is only one subdomain.
Finally, the variational form is similar, where we seek $\lambda \in H^\gamma_D$
\begin{align*}
   \sum_{i=1}^n (\mat u^*(\lambda), \nabla \mu)_{\Omega_i} =(f, \mu)_\Omega + (g, \mu)_{\partial \Omega} - \sum_{i=1}^n (\bar{\mat u}, \nabla \mu)_{\Omega_i}, \qquad \forall \mu \in H^\gamma_D.
\end{align*}

Turning to the discrete case, let $W_{hi}, \mat V_{hi}$ be as before and let $W_{hi, D}$ be the discretization of $H^1_D(\Omega_i)$.
Let $\Lambda_{H, D} \subset \Lambda_H$ be the discretized mortar space consisting of continuous, piecewise linear functions that vanish where $H^\gamma_D$ is zero.
The projection $Q_i$ should be changed to $Q_i : H^\gamma_D \to W_{hi}$ with the same alteration to \cref{eqn:projection-h12}. 

Thus, the discetized bilinear form and linear functional is similar to before, with the exception that $p_h^*(Q_i \lambda_H) \in \oplus_{i=1}^n W_{hi}, \mat u^*_h(Q_i \lambda_H) \in \oplus_{i=1}^n \mat V_{hi}$ satisfies, for $1 \le i \le n$, 
\begin{align}
	(\mat u^*_h(Q_i \lambda_H), \mat v_h) - (K\nabla p_h^*(Q_i \lambda_H), \mat v_h) &= 0, &\forall \mat v_h \in \mat V_{hi} \\
	(\mat u_h^*(Q_i \lambda_H), \nabla w_h)  &=0,  &\forall w_h \in W_{hi, D}\label{eqn:star-h-neumann}
\end{align}
with $p^*_h(Q_i\lambda_H) = Q_i\lambda_H$ on $\Gamma_i \setminus \partial \Omega$, 
and $\bar p_h \in \oplus_{i=1}^n W_{hi}, \bar{\mat u}_h \in \oplus_{i=1}^n V_{hi}$ satisfying
\begin{align}
	(\bar{\mat u}_h, \mat v_h) - (K\nabla \bar p_h, \mat v_h) &= 0, &\forall \mat v_h \in \mat V_{hi}  \\
	(\bar{\mat u}_h, \nabla w_h)  &= (f, w_h) + (g, w_h)_{\partial \Omega}, &\forall w_h \in W_{hi, D}\label{eqn:bar-h-neumann}
\end{align}
with $\bar p_h = 0$ on $\Gamma_i \setminus \partial \Omega_i$.
The variational form (written explicitly) is to find $\lambda_H \in \Lambda_{H, D}$, with mean zero, such that
\begin{equation}\label{eqn:disc-bilinear-form-neumann}
    \begin{aligned}
   \sum_{i=1}^n (\mat u_h^*(Q_i \lambda_H), \nabla (Q_i\mu_H))_{\Omega_i} &= \sum_{i=1}^n (f, Q_i\mu_H)_{\Omega_i} + (g, Q_i \mu_H)_{\Omega_i \cap \partial \Omega}\\
   &\qquad - (\bar{\mat u}_h, \nabla (Q_i\mu_H))_{\Omega_i}, \qquad \forall \mu_H \in \Lambda_{H, D}.
\end{aligned}
\end{equation}

Turning to \cref{eqn:conservation} and {\bf R1}, it's easy to see that for a function $\mu_H \in \Lambda_{H, D}$, the projection is exact $Q_i \mu_H|_{\Gamma_i \setminus \partial \Omega} = \mu_H|_{\Gamma_i \setminus \partial \Omega}$ if $\mu_{H}$ is constant on the interior edges (e.g. $\mu_H = C$ on $\Gamma_i \setminus \partial \Omega$).
Without loss of generality, let $\Lambda_{H, D} \ni \mu_H = 1$ on all $\Gamma_i \setminus \partial \Omega$, then we can choose $w_h$ such that $w_h + Q_i \mu_H = 1$ on $W_{hi, D}$ for each $1 \le i \le n$. 
Thus, adding \cref{eqn:star-h-neumann} and \cref{eqn:bar-h-neumann} for $1 \le i \le n$ to \cref{eqn:disc-bilinear-form-neumann} and rearranging, we obtain
\begin{align*}
    \sum_{i=1}^n (\mat u_h^*(Q_i \lambda_H) + \bar{\mat u}_h, \nabla (w_h + Q_i\mu_H))_{\Omega_i} &= \sum_{i=1}^n (\mat u_h^*(Q_i \lambda_H) + \bar{\mat u}_h, \nabla 1)_{\Omega_i} \\
    &= \sum_{i=1}^n (f, w_h + Q_i \mu_H)_{\Omega_i} + (g, w_h + Q_i \mu_H)_{\partial \Omega} \\
    &= \sum_{i=1}^n (f,1)_{\Omega_i} + (g, 1)_{\partial \Omega} = 0,
\end{align*}
meaning \cref{eqn:conservation} is valid even in the discrete case with non-matching mortars. 

\section{Numerical Results}\label{sec:numerical-results}
In this section, we present numerical results obtained from applying the above mortar method to several representative examples. 
All except the first example will involve using pre-trained FEEC elements as local subdomain solvers as discussed in \cref{sec:whitneysec}. 

\subsection{Example 1: Pure Finite Elements}\label{sec:example-1}
We start by validating the accuracy and well-posedness of the mortar method in the classical setting by using only finite element solvers on each subdomain.
In particular, no model training is used for this particular example and we only seek to show that the above mortar method converges in the forward problem.  
We examine the problem \cref{eqn:primal-big} with true solution $p(x, y) = xy + y^2$ on the domain $\Omega = [0, 2]^2$ with
\begin{align}
    K = \left(
\begin{array}{cc}
 (x+1)^2 & 0.5 \\
 0.5 & y^2+1 \\
\end{array}
\right).
\end{align}
The domain is subdivided into four equal squares.
Our initial mesh is depicted in \cref{fig:mesh-example1} with only one degree of freedom on the mortar (with the remaining four fixed due to the homogeneous Dirichlet boundary condition).
For refinement, we divide each subdomain diameter and the mortar diameter by half; see the right hand side of \cref{fig:mesh-example1} for a figure of the first refinement. 

On each subdomain, we will use the standard $\mathbb{Q}_1$ space for pressure with Nedelec elements of the lowest order for the velocity. 
The quantities in \cref{lem:error-pre-vel} can be replaced with results from standard FEM a priori estimates \cite{roberts1991mixed}. 
As a result, we obtain a convergence result of 
\begin{align}\label{eqn:fem-convergence}
    \norm{\mat u - \mat u_h}_{\Omega} + \norm{p - p_h}_{\Omega} \le CH
\end{align}
where $H$ is the size of the mortar. 
The $\mathcal{O}(H)$ convergence in the velocity is clearly illustrated in \cref{tab:convergence-fem} while we obtain $\mathcal{O}(H^2)$ superconvergence in the pressure, which was observed in smooth solutions using mortar methods \cite{arbogast2007multiscale,arbogast2000mixed}. 
Furthermore, since the estimate in \cref{lem:error-pre-vel} is in the $H^{1/2}$ norm, we expect convergence of $\mathcal{O}(H^{3/2})$ as we are measuring the $L^2$ norm but we also observe a level of superconvergence.

\begin{figure}[tb]
    \centering
    \begin{tikzpicture}[scale=1]
\draw[thick, step=1cm] (0, 0) grid (2, 2);
\draw[thick, step=.666cm, xshift=2.2cm] (0, 0) grid (2, 2);
\draw[thick, step=.666cm, yshift=-2.2cm] (0, 0) grid (2, 2);
\draw[thick, step=1cm,  xshift=2.2cm, yshift=-2.2cm] (0, 0) grid (2, 2);

\coordinate (A0) at (0, -.1);
\coordinate (A1) at (2.1, -.1);
\coordinate (A2) at (4.2, -.1);
\coordinate (A3) at (2.1, 2);
\coordinate (A4) at (2.1, -2.2);
\draw[thick] (A0) -- (A1) -- (A2);
\draw[thick] (A3) -- (A1) -- (A4);
\filldraw [gray] (A0) circle (2pt);
\filldraw [gray] (A1) circle (2pt);
\filldraw [gray] (A2) circle (2pt);
\filldraw [gray] (A3) circle (2pt);
\filldraw [gray] (A4) circle (2pt);
\end{tikzpicture}
\qquad
    \begin{tikzpicture}[scale=1]
\draw[thick, step=.5cm] (0, 0) grid (2, 2);
\draw[thick, step=.333cm, xshift=2.2cm] (0, 0) grid (2, 2);
\draw[thick, step=.333cm, yshift=-2.2cm] (0, 0) grid (2, 2);
\draw[thick, step=.5cm,  xshift=2.2cm, yshift=-2.2cm] (0, 0) grid (2, 2);

\coordinate (A0) at (0, -.1);
\coordinate (A1) at (2.1, -.1);
\coordinate (A2) at (4.2, -.1);
\coordinate (A3) at (2.1, 2);
\coordinate (A4) at (2.1, -2.2);
\draw[thick] (A0) -- (A1) -- (A2);
\draw[thick] (A3) -- (A1) -- (A4);
\filldraw [gray] (A0) circle (2pt);
\filldraw [gray] (A1) circle (2pt);
\filldraw [gray] (A2) circle (2pt);
\filldraw [gray] (A3) circle (2pt);
\filldraw [gray] (A4) circle (2pt);
\filldraw [gray] (1, -.1) circle (2pt);
\filldraw [gray] (3.2, -.1) circle (2pt);
\filldraw [gray] (2.1, -1.2) circle (2pt);
\filldraw [gray] (2.1, 1) circle (2pt);
\end{tikzpicture}
    \caption{Figure illustrating the initial mesh, the corresponding mortar space, and their first refinement for the example in \cref{sec:example-1}.}
    \label{fig:mesh-example1}
\end{figure}
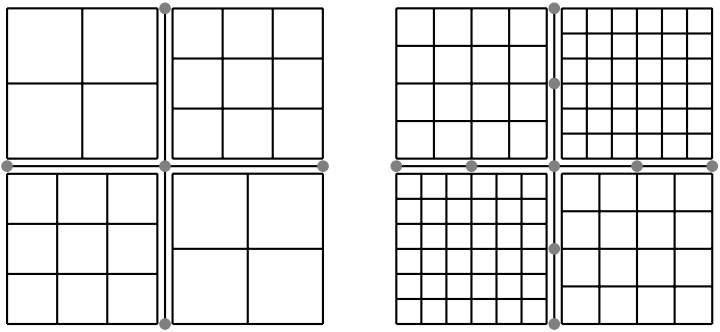

\begin{table}[H]
    \centering
    \caption{Table illustrating the absolute errors, and the convergence rates for Example 1.
    The rates are in agreement with \cref{lem:error-pre-vel}.
    }
    \begin{tabular}{|c|ccc|}
       \hline
       $H$  & $\norm{p - p_h}_{\Omega}$ & $\norm{\mat u -  \mat u_h}_{\Omega}$ & $\norm{\lambda - \lambda_h}_{L^2(\Gamma)}$ \\
       \hline
       1  & 2.73E-01 & 4.66E+00        &2.44E-01 \\
       $1/2$  & 6.23E-02 & 2.16E+00    &5.75E-02 \\
       $1/4$  & 1.49E-02  & 1.04E+00   &1.43E-02 \\
       $1/8$  & 3.66E-03 & 5.12E-01    &3.56E-03 \\
       $1/16$  & 9.07E-04 & 2.54E-01   &8.91E-04 \\
       $1/32$  &  2.31E-04 & 1.26E-01  &2.41E-04 \\ 
       \hline
       Rate & $\mathcal{O}(H^{2.04})$ & $\mathcal{O}(H^{1.04})$& $\mathcal{O}(H^{2.00})$  \\
       \hline
    \end{tabular}
    \label{tab:convergence-fem}
\end{table}


\subsection{Example 2: Pure FEEC and Pure FEM Elements Comparison}\label{sec:sine-cosine-problem}
We now consider the data arising from the problem \cref{eqn:primal-big} with $\Omega = [0, 3] \times [0, 3]$, 
\begin{align}\label{eqn:sine-cosine-forcing}
    f := 2\pi^2 \cos(\pi x)\sin(\pi y),  \qquad K = \mat I
\end{align}
with boundary condition determined by the true solution $p(x, y) = \cos(\pi x) \sin(\pi y)$.

The domain $\Omega$ is split into $9$ uniform squares whereby either a $\mathbb{Q}_1$ FEM or a pretrained FEEC element is used in each subdomain. 
The FEEC element is trained on 20480 uniformly drawn points from $[0, 1]^2$ with 16 POUs on the interior and 16 on the boundary with varying number of fine-scale knots. 
As discussed in \cite{actoradata}, increasing the number of fine-scale grids is akin to $h$-refinement in the FEM sense. 

To train the FEEC eleemnts, we use data arising from different boundary conditions and forcing terms which corresponds to approximating \cref{eqn:bar-h} and \cref{eqn:star-h}:
\begin{enumerate}
    \item a problem with the same forcing term as in \cref{eqn:sine-cosine-forcing} but homogeneous zero Dirichlet boundary condition.
    This corresponds to \cref{eqn:bar-h}).
    \item Sixteen different boundary conditions consisting of the Bernstein polynomials of forth order on the boundary (e.g. $x^4y^4$, $\binom{4}{1}x^4y(1-y)^3$ etc) and forcing term of $f = 0$. 
    This is needed so that \cref{eqn:star-h} can be approximated accurately on the FEEC elements when different boundary conditions are passed in from the mortar. 
\end{enumerate}
The solutions to the above boundary value problems were calculated by a low-order finite element solver. 
For more details regarding the training, we refer the reader to the appendix \cref{sec:feec-training}.

The mortar refinement level was chosen to be $H = 4h$ in for the FEM case.
For the FEEC local solvers, we note that the fine-scale nodes can move, resulting in non-uniform meshes; nevertheless, we still choose the same $H$ as the FEM case for comparison's sake.

In \cref{tab:sine-error} and \cref{tab:sine-error-fem}, we show the error resulting from using purely FEEC elements or purely FEM elements on all the subdomain respectively. 
The convergence rates among the two different different solvers are similar, and reflect superconvergence due to the smoothness of the problem. 
In \cref{fig:sine-sol}, we plot the true solution and its fluxes, and the approximate solution and its fluxes on the whole $[0, 3]^2$ domain solved using FEEC elements, while \cref{fig:sine-line} plots the quantities on the diagonal line from $(0, 0)$ through $(3, 3)$.
In both cases, the true solution is well-approximated. 

\begin{table}[ht]
    \centering
    \begin{tabular}{|c|ccc|}
        \hline
        FEEC fine-scale grid and mortar size & $\norm{p - p_h}_{\Omega}$ & $\norm{\mat u - \mat u_h}_{\Omega}$ & $\norm{\lambda - \lambda_H}_{L^2(\Gamma)}$\\
        \hline
        $8 \times 8, H=1/2$  & 1.47E-01	&1.48E+00	&2.20E-01\\
        $12 \times 12, H=1/3$& 6.66E-02	&8.97E-01	&8.47E-02 \\
        $16 \times 16, H=1/4$& 4.07E-02	&5.41E-01	&4.41E-02\\
        $20 \times 20, H=1/5$& 2.76E-02	&4.40E-01	&2.71E-02\\
        $24 \times 24, H=1/6$& 2.10E-02	&3.97E-01	&1.97E-02\\
        \hline 
         &$\mathcal{O}(h^{1.78})$ & $\mathcal{O}(h^{1.25})$& $\mathcal{O}(h^{2.22})$ \\
       \hline
    \end{tabular}
    \caption{Table of absolute error for the sine-cosine problem \cref{sec:sine-cosine-problem} using trained FEEC elements as the subdomain solver.
    The convergence rates are similar to the method using pure FEM elements.
    We note that the fine-scale grid roughly correspond to $h$-scaling in a standard FEM method \cite{actoradata}.}
    \label{tab:sine-error}
\end{table}

\begin{table}[ht]
    \centering
    \begin{tabular}{|c|ccc|}
        \hline
        FEM fine-scale grid and mortar size & $\norm{p - p_h}_{\Omega}$ & $\norm{\mat u - \mat u_h}_{\Omega}$ & $\norm{\lambda - \lambda_H}_{L^2(\Gamma)}$\\
        \hline
        $8 \times 8, H=1/2$   & 1.67E-01    & 1.54E+00&	2.16E-01\\
        $12 \times 12, H=1/3$ & 7.04E-02	& 8.54E-01&	8.16E-02 \\
        $16 \times 16, H=1/4$ & 3.91E-02	& 5.76E-01&	4.31E-02\\
        $20 \times 20, H=1/5$ & 2.49E-02	& 4.30E-01&	2.68E-02\\
        $24 \times 24, H=1/6$ & 1.72E-02	& 3.41E-01&	1.83E-02\\
        $28 \times 28, H=1/7$ & 1.26E-02	& 2.82E-01&	1.33E-02\\
        $32 \times 32, H=1/8$ & 9.64E-03	& 2.40E-01&	1.01E-02\\
        $36 \times 36, H=1/9$ & 7.61E-03	& 2.08E-01&	7.98E-03\\
        $40 \times 40, H=1/10$& 6.15E-03	& 1.84E-01&	6.45E-03\\
        \hline 
         &$\mathcal{O}(h^{2.04})$ & $\mathcal{O}(h^{1.31})$& $\mathcal{O}(h^{2.16})$ \\
       \hline
    \end{tabular}
    \caption{Table of absolute error for the sine-cosine problem \cref{sec:sine-cosine-problem} with FEM elements as the local solvers.}
    \label{tab:sine-error-fem}
\end{table}

\begin{figure}[ht]
     \centering
     \begin{subfigure}[b]{0.32\textwidth}
         \centering
         \includegraphics[trim={3cm 1cm 1cm 1cm},clip, width=\textwidth]{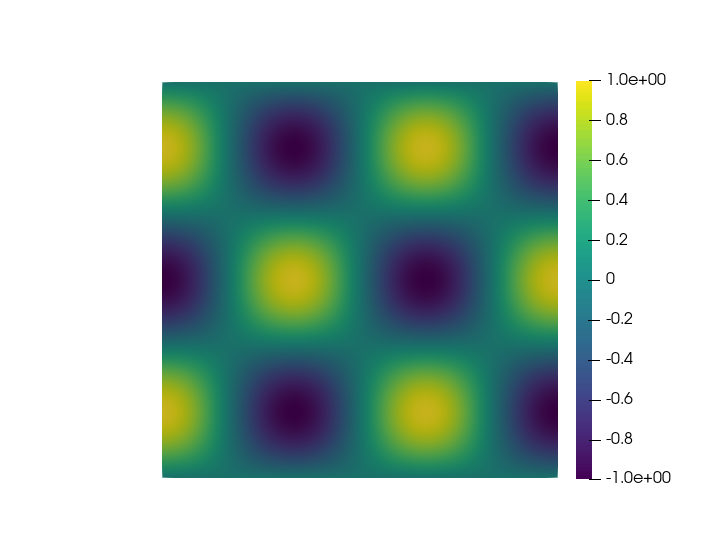}
         \caption{True pressure}
     \end{subfigure}
     \hfill
     \begin{subfigure}[b]{0.32\textwidth}
         \centering
         \includegraphics[trim={3cm 1cm 1cm 1cm},clip,width=\textwidth]{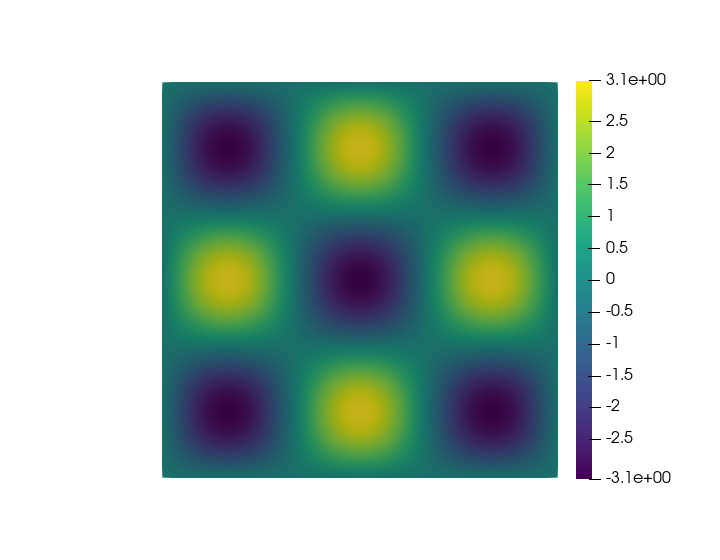}
        \caption{True gradient in $x$}
     \end{subfigure}
     \hfill
     \begin{subfigure}[b]{0.32\textwidth}
         \centering
         \includegraphics[trim={3cm 1cm 1cm 1cm},clip,width=\textwidth]{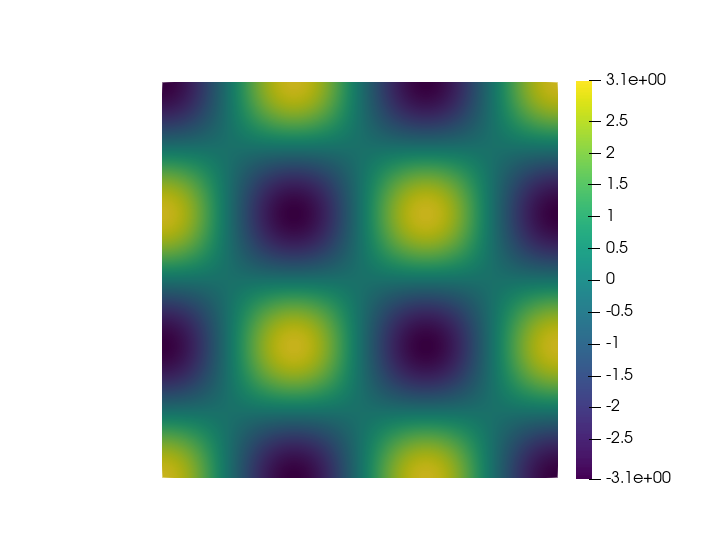}
         \caption{True gradient in $y$}
     \end{subfigure} 
     \\
          \begin{subfigure}[b]{0.32\textwidth}
         \centering
         \includegraphics[trim={3cm 1cm 1cm 1cm},clip, width=\textwidth]{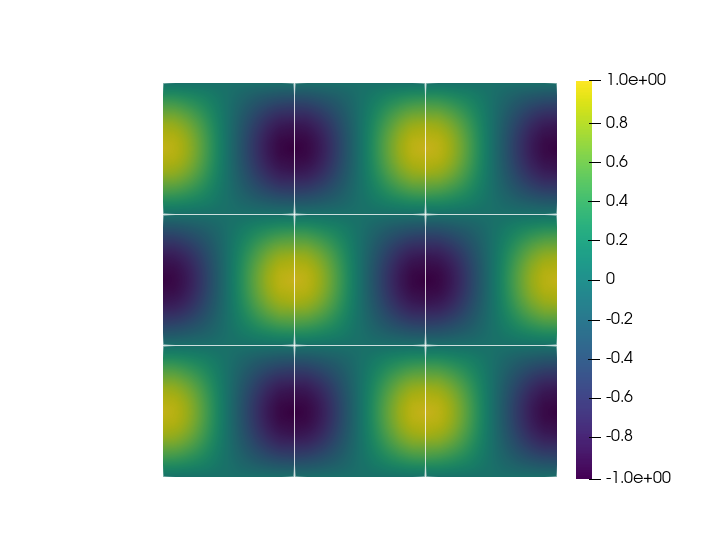}
         \caption{Estimated pressure}
     \end{subfigure}
     \hfill
     \begin{subfigure}[b]{0.32\textwidth}
         \centering
         \includegraphics[trim={3cm 1cm 1cm 1cm},clip,width=\textwidth]{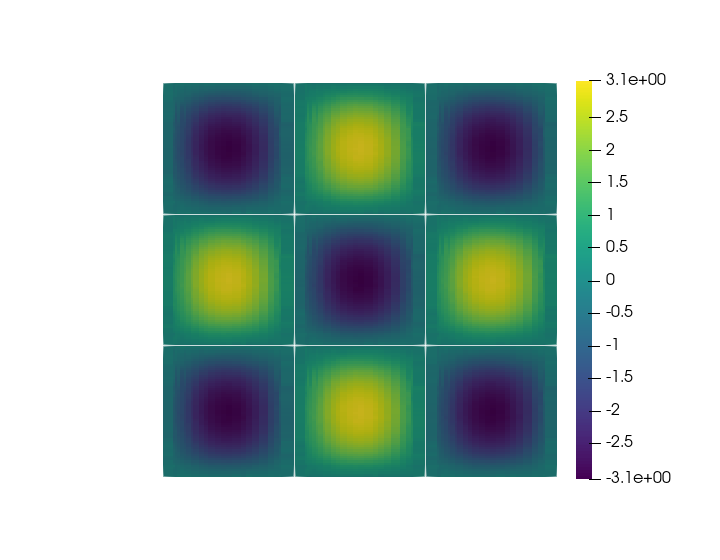}
         \caption{Estimated gradient in $x$}
     \end{subfigure}
     \hfill
     \begin{subfigure}[b]{0.32\textwidth}
         \centering
         \includegraphics[trim={3cm 1cm 1cm 1cm},clip,width=\textwidth]{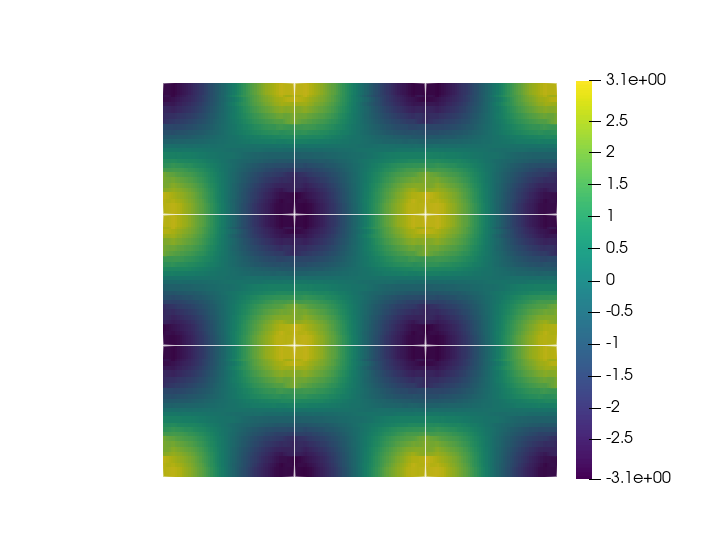}
         \caption{Estimated gradient in $y$}
     \end{subfigure}
        \caption{
        Plot of the true (first row) and estimated (second row) solution for the sine-cosine problem \cref{sec:sine-cosine-problem} with pure FEEC elements consisting of $24 \times 24$ fine scale nodes, and $H= 1/6$. 
        As expected, the solution is well-approximated by the FEEC elements. 
        }
        \label{fig:sine-sol}
\end{figure}

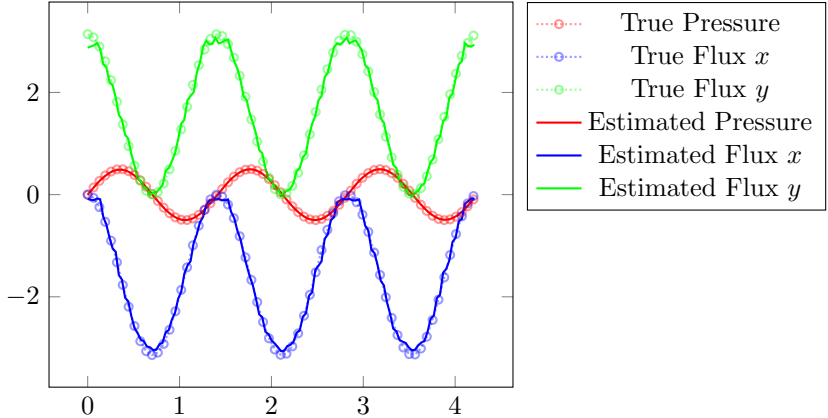
\begin{figure}
    \centering
    \begin{tikzpicture}
    \begin{axis}[scale=.9,
        legend pos=outer north east]
    
    \addplot[red, mark=o, mark options=solid, densely dotted, thick, opacity=.4, mark size=1.5] table[col sep=comma, each nth point=15, x=arclength, y=funcvalues] {sine-cosine-data/true_vals.csv};
    \addlegendentry{True Pressure};
    \addplot[blue, mark=o,  mark options=solid, densely dotted, thick, opacity=.4, mark size=1.5] table[col sep=comma, each nth point=15, x=arclength, y=gradx] {sine-cosine-data/true_vals.csv};
    \addlegendentry{True Flux $x$};
    \addplot[green, mark=o,  mark options=solid, densely dotted, thick, opacity=.4, mark size=1.5] table[col sep=comma, each nth point=15, x=arclength, y=grady] {sine-cosine-data/true_vals.csv};
    \addlegendentry{True Flux $y$};
    \addplot[red, mark=none, thick] table[col sep=comma, each nth point=8, x=arclength, y=pressure] {sine-cosine-data/est_vals.csv};
    \addlegendentry{Estimated Pressure};
    \addplot[blue, mark=none, thick] table[col sep=comma, each nth point=8, x=arclength, y=velx] {sine-cosine-data/est_vals.csv};
    \addlegendentry{Estimated Flux $x$};
    \addplot[green, mark=none, thick] table[col sep=comma, each nth point=8, x=arclength, y=vely] {sine-cosine-data/est_vals.csv};
    \addlegendentry{Estimated Flux $y$};
    
    \end{axis}
    \end{tikzpicture}
    \caption{Profile of the true and estimate solutions for the sine-cosine problem \cref{sec:sine-cosine-problem} on the line $(0, 0)$ to $(3, 3)$ using pure FEEC elements with $24 \times 24$ fine scale nodes and $H=1/6$.
    While there are small fluctuations in the FEEC approximation, it is clear that both the pressure and fluxes are captured.}
    \label{fig:sine-line}
\end{figure}

\subsection{Example 3: Hybrid Methods}\label{sec:cylinder}
We next showcase the ability to use a hybrid approach whereby standard finite elements are interfaced to FEEC elements allowing for areas with unknown features to be learned using FEEC elements, and smooth areas using classical FEM methods. 

We assume data is obtained from the problem \cref{eqn:primal-big} on $\Omega = [-1.5, 1.5]^2$ with the parameters 
\begin{align}
    f := 0, \qquad K(\mat x) = \begin{cases}
            \begin{pmatrix}
        k & 0 \\
        0 & k 
    \end{pmatrix}, & \norm{\mat x} \le b \\
    \begin{pmatrix}
        1 & 0 \\
        0 & 1
    \end{pmatrix}, & \norm{\mat x} > b \\
    \end{cases}
    \label{eqn:cylinder-kappa}
\end{align}
with $b = .2, k = 10$. 
The Dirichlet boundary imposed such that the true solution is
\begin{align*}
    u := \begin{cases}
            x \left(1-\frac{b^2 (k-1)}{(k+1) \left(x^2+y^2\right)}\right)  & \norm{\mat x} > b \\
            \frac{2}{k+1}x & \norm{\mat x} \le b
    \end{cases}
\end{align*}
This particular equation arises in electrostatics when examining the case where a conducting cylinder with radius $b$ and capacitance $k$ is placed within a uniform field of strength 1 \cite[\S 4.03]{Smythe1989-tk}. 
Note that outside of a radius around the origin, the diffusion problem is easy to solve. 

We split the domain is split into 9 congruent squares with the center square $[-0.5, 0.5]^2$ consisting of a FEEC element to capture the change in material coefficients while the remaining eight subdomains utilizing a simple, low-order FEM space with $8 \times 8$ quads. 
The FEEC element is trained on 12 different boundary conditions corresponding to the 12 third-order Bernstein polynomials on the boundary as in the previous example.
We note that in training, only the solution and its fluxes are provided, meaning the material coefficient \cref{eqn:cylinder-kappa} is not fully exposed to the FEEC element. 
A total of 16 POUs are used on the interior and the boundary. 
We choose to use a mortar of $H = 1/4$. 

\begin{table}[ht]
    \centering
    \begin{tabular}{|c|ccc|}
        \hline
        FEEC fine-scale grid  & $\norm{p - p_h}_{\Omega}$ & $\norm{\mat u - \mat u_h}_{\Omega}$ & $\norm{\lambda - \lambda_H}_{L^2(\Gamma)}$\\
        \hline
        $8 \times 8$ & 5.69E-03 ($0.224\%$) & 1.64E-01 ($5.48\%$) & 6.41E-3\\
        $16 \times 16$ & 3.07E-03 ($0.121\%$) &1.16E-01 ($3.87\%$) &4.66E-3\\
        $24 \times 24$ & 2.01E-03 ($0.079\%$) & 8.29E-02 ($2.77\%$)&2.50E-3\\
        \hline 
         &$\mathcal{O}(h^{.939})$ & $\mathcal{O}(h^{.607})$& $\mathcal{O}(h^{.814})$ \\
       \hline
    \end{tabular}
    \caption{Table of absolute and relative errors for the cylinder problem \cref{sec:cylinder} using a hybrid approach.
    While we do not expect a full convergence as we are only refining the singular FEEC element on  $[-.5, .5]^2$ while keeping the mortar spaces and FEM spaces constant, we do observe that using the finest FEEC element gives significantly better results.}
    \label{tab:cylinder-error}
\end{table}

We show the error over the whole domain in \cref{tab:cylinder-error} from only refining the fine-scale grid of the FEEC element in $[-.5, .5]^2$. 
A full rate of convergence is not expected since \cref{thm:err-est} requires both the mortar space and the local subdomain solvers to be refined in tandem.
We do not consider refinement with the mortar here as Assumption 1 might be violated from either the movement of fine-scale knots of the FEEC elements, or the fact that the mesh size of the FEM solvers are fixed to be very coarse.

In \cref{fig:cyl-sol}, we plot the true and estimated solution to the problem. 
Note that the trained FEEC element managed to resolve the circular inclusion and the subtleties in the fluxes when the true solution is not explicitly given in the training data. 
Furthermore, we plot the true and estimated solution profiles in \cref{fig:cyl-sol-line}.
From the plots, it is clear that while there are small spurious fluctuations in the estimated solutions, that the error decreases as we refine the FEEC model. 
In \cref{fig:cyl-fem-comparison}, we compare the FEEC profiles to the profile obtained using a $24 \times 24$ FEM on $[-.5, .5]^2$ instead. 
Note that the oscillations are greatly reduced by using the FEEC elements due to the adaptivity of the fine-scale mesh. 

\begin{figure}[ht]
     \centering
     \begin{subfigure}[b]{0.32\textwidth}
         \centering
         \includegraphics[trim={3cm 1cm 0cm 1cm},clip, width=\textwidth]{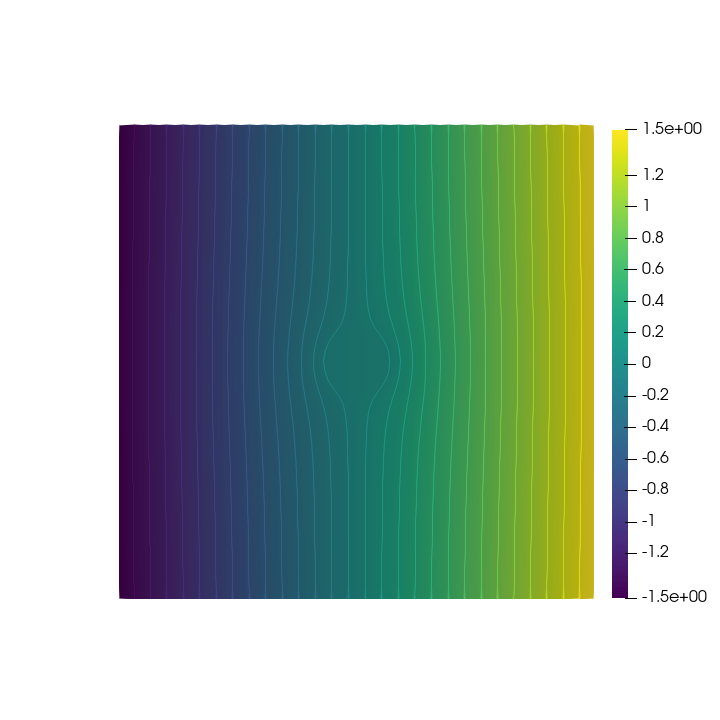}
         \caption{True pressure}
     \end{subfigure}
     \hfill
     \begin{subfigure}[b]{0.32\textwidth}
         \centering
         \includegraphics[trim={3cm 1cm 0cm 1cm},clip,width=\textwidth]{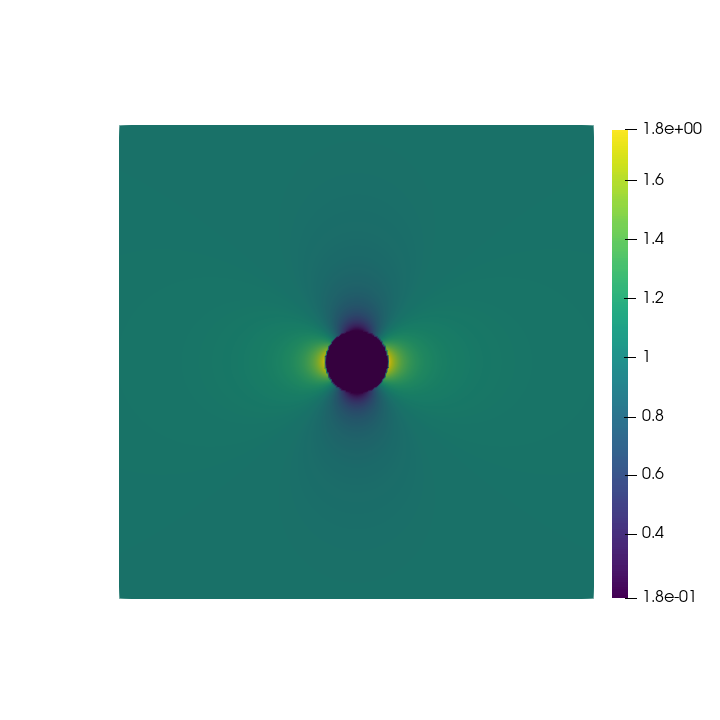}
        \caption{True gradient in $x$}
     \end{subfigure}
     \hfill
     \begin{subfigure}[b]{0.32\textwidth}
         \centering
         \includegraphics[trim={3cm 1cm 0cm 1cm},clip,width=\textwidth]{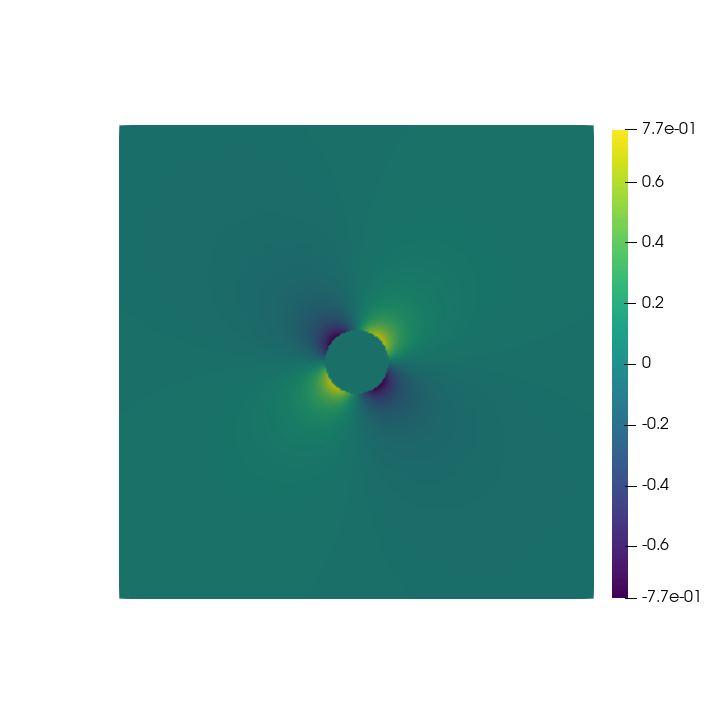}
         \caption{True gradient in $y$}
     \end{subfigure} 
     \\
          \begin{subfigure}[b]{0.32\textwidth}
         \centering
         \includegraphics[trim={3cm 1cm 0cm 1cm},clip, width=\textwidth]{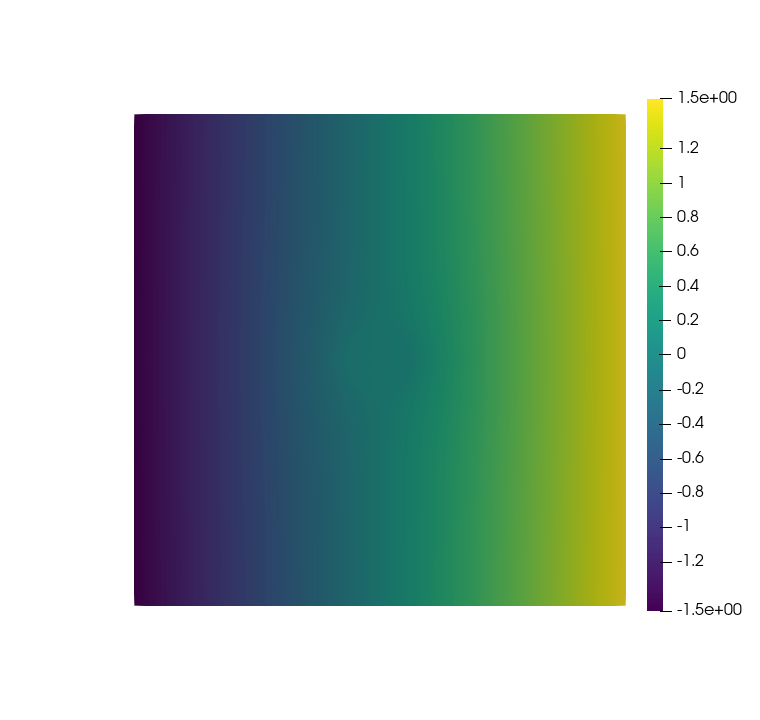}
         \caption{Estimated pressure}
     \end{subfigure}
     \hfill
     \begin{subfigure}[b]{0.32\textwidth}
         \centering
         \includegraphics[trim={3cm 1cm 0cm 1cm},clip,width=\textwidth]{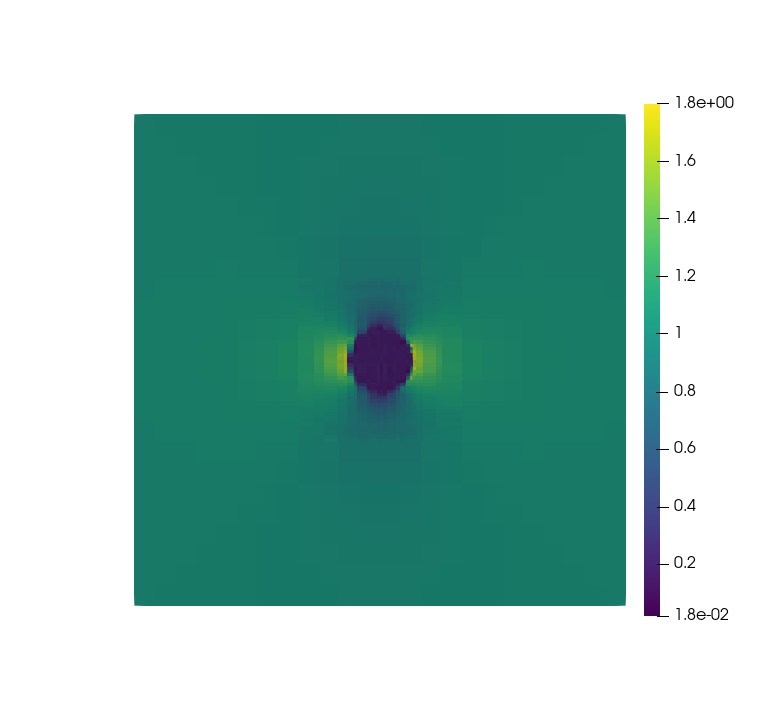}
         \caption{Estimated gradient in $x$}
     \end{subfigure}
     \hfill
     \begin{subfigure}[b]{0.32\textwidth}
         \centering
         \includegraphics[trim={3cm 1cm 0cm 1cm},clip,width=\textwidth]{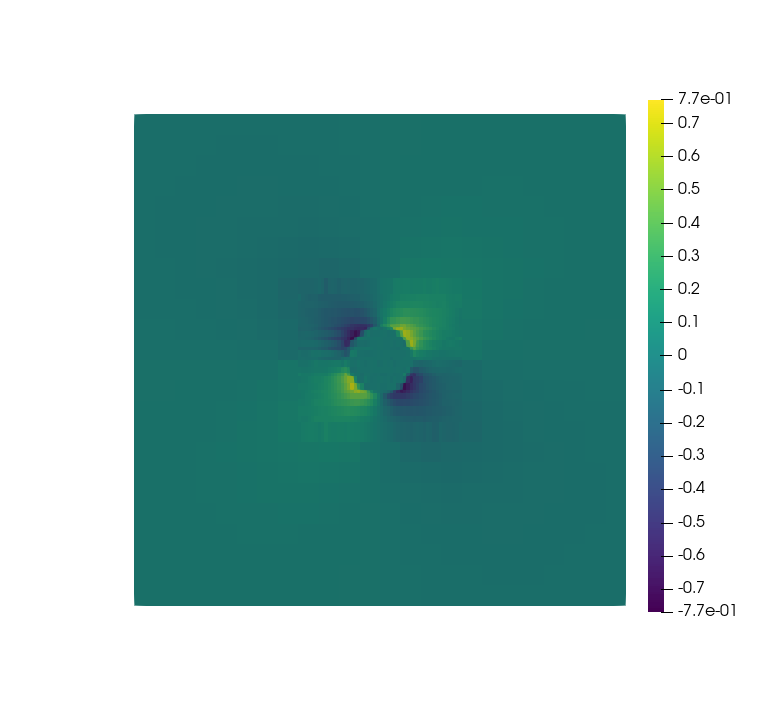}
         \caption{Estimated gradient in $y$}
     \end{subfigure}
        \caption{
        Figure of the true solution and estimated value for the cylinder problem \cref{sec:cylinder}. 
        The estimated solution uses a single FEEC element with $24 \times 24$ fine-scale knots in the center-most subdomain with the remaining subdomains using FEM of just $8 \times 8$ elements. 
        }
        \label{fig:cyl-sol}
\end{figure}

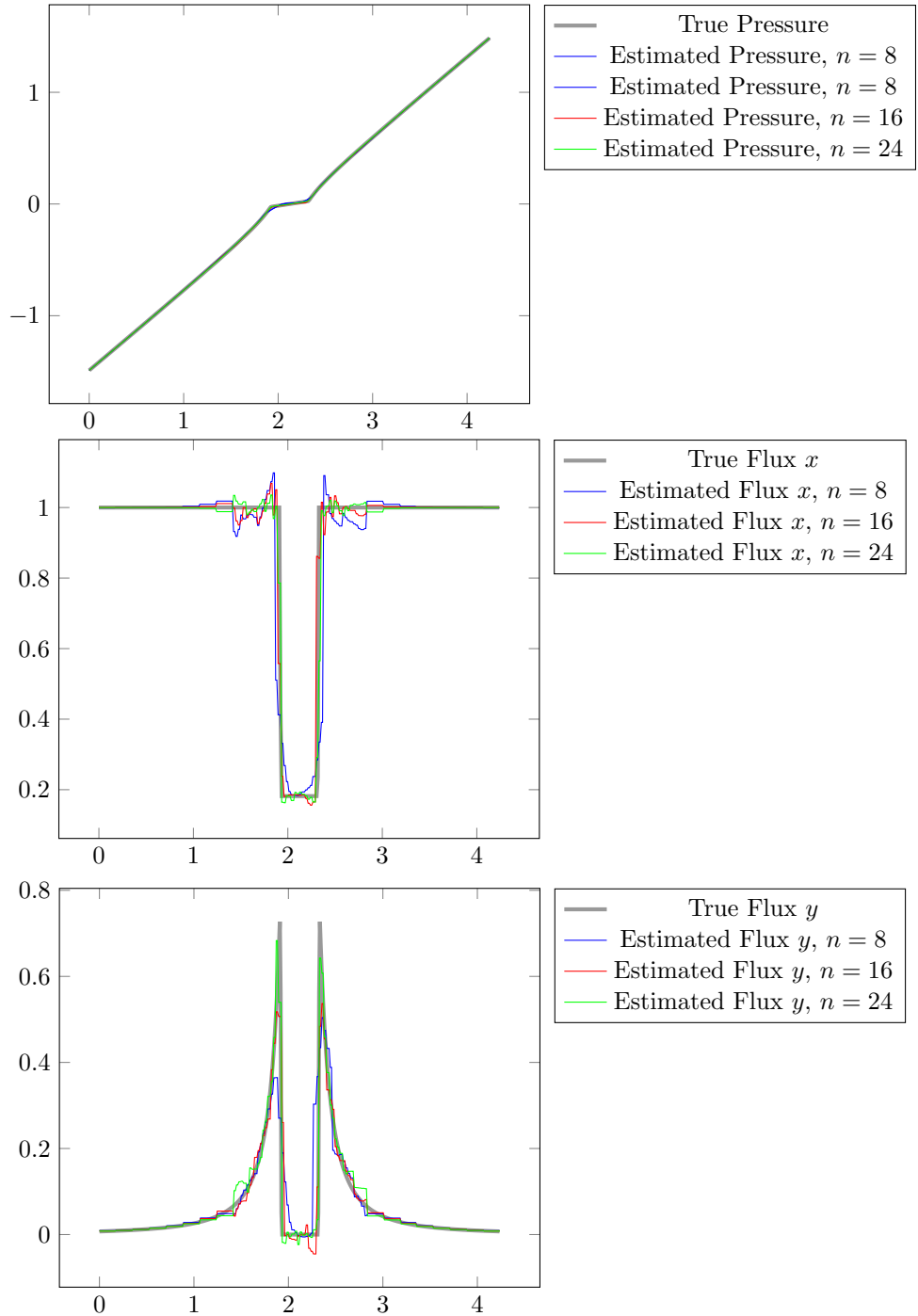
\begin{figure}
     \centering
        \begin{tikzpicture}
        \begin{axis}[scale=.99,
            legend pos=outer north east,
              ]
        \addplot[black, mark=circle, opacity=.4, ultra thick] table[col sep=comma, each nth point=3, x="arc_length", y="func_values"] {cylinder-data/true_data.csv};
        \addlegendentry{True Pressure};
        \addplot[blue, mark=none] table[col sep=comma, each nth point=3, x="arc_length", y expr=\thisrow{"pressure"}] {cylinder-data/est_data8.csv}; 
        \addlegendentry{Estimated Pressure, $n=8$};
        \addplot[blue, mark=none] table[col sep=comma, each nth point=3, x="arc_length", y="pressure"] {cylinder-data/est_data8.csv};
        \addlegendentry{Estimated Pressure, $n=8$};
        \addplot[red, mark=none] table[col sep=comma, each nth point=3, x="arc_length", y="pressure"] {cylinder-data/est_data16.csv};
        \addlegendentry{Estimated Pressure, $n=16$};
        \addplot[green, mark=none] table[col sep=comma, each nth point=3, x="arc_length", y="pressure"] {cylinder-data/est_data24.csv};
        \addlegendentry{Estimated Pressure, $n=24$};
        \end{axis}
        \end{tikzpicture}
    \begin{tikzpicture}
        \begin{axis}[scale=.99,
            legend pos=outer north east,
              ]
        \addplot[black, mark=circle, opacity=.4, ultra thick] table[col sep=comma, each nth point=3, x="arc_length", y="grad_x"] {cylinder-data/true_data.csv};
        \addlegendentry{True Flux $x$};
        \addplot[blue, mark=none] table[col sep=comma, each nth point=3, x="arc_length", y="vel_x"] {cylinder-data/est_data8.csv};
        \addlegendentry{Estimated Flux $x$, $n=8$};
        \addplot[red, mark=none] table[col sep=comma, each nth point=3, x="arc_length", y="vel_x"] {cylinder-data/est_data16.csv};
        \addlegendentry{Estimated Flux $x$, $n=16$};
        \addplot[green, mark=none] table[col sep=comma, each nth point=3, x="arc_length", y="vel_x"] {cylinder-data/est_data24.csv};
        \addlegendentry{Estimated Flux $x$, $n=24$};
        \end{axis}
        \end{tikzpicture}
        \begin{tikzpicture}
        \begin{axis}[scale=.99,
            legend pos=outer north east,
              ]
        \addplot[black, mark=circle, opacity=.4, ultra thick] table[col sep=comma, each nth point=2, x="arc_length", y="grad_y"] {cylinder-data/true_data.csv};
        \addlegendentry{True Flux $y$};
        \addplot[blue, mark=none] table[col sep=comma, each nth point=3, x="arc_length", y="vel_y"] {cylinder-data/est_data8.csv};
        \addlegendentry{Estimated Flux $y$, $n=8$};
        \addplot[red, mark=none] table[col sep=comma, each nth point=3, x="arc_length", y="vel_y"] {cylinder-data/est_data16.csv};
        \addlegendentry{Estimated Flux $y$, $n=16$};
        \addplot[green, mark=none] table[col sep=comma, each nth point=3, x="arc_length", y="vel_y"] {cylinder-data/est_data24.csv};
        \addlegendentry{Estimated Flux $y$, $n=24$};
        \end{axis}
        \end{tikzpicture}
        \caption{
        Trace plot from $[-1.5, -1.5]$ to $[1.5, 1.5]$ of the cylinder problem \cref{sec:cylinder} for FEEC models with different fine-scale nodes $n$ in the subdomain $[-.5, .5]^2$. 
        As we refine the number of fine-scale nodes we use, the jumps in the fluxes are increasingly more well-resolved with less fluctuations. 
        }
        \label{fig:cyl-sol-line}
\end{figure}

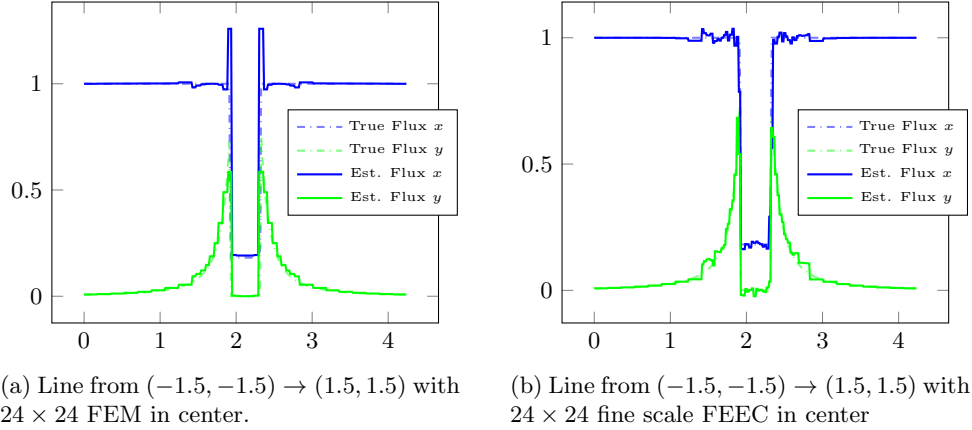
\begin{figure}[ht]
     \centering
     \begin{subfigure}[b]{0.47\textwidth}
         \centering
        \begin{tikzpicture}
        \begin{axis}[scale=.75,
             ticklabel style = {font=\small},
              legend style={font=\tiny,
              at={(1.05,0.5)},anchor=east}]
        
        \addplot[blue, mark=circle, dash dot, thick, opacity=.5] table[col sep=comma, each nth point=3, x="arc_length", y="grad_x"] {cylinder-data/true_data.csv};
        \addlegendentry{True Flux $x$};
        \addplot[green, mark=circle, dash dot, thick, opacity=.5] table[col sep=comma, each nth point=3, x="arc_length", y="grad_y"] {cylinder-data/true_data.csv};
        \addlegendentry{True Flux $y$};
        \addplot[blue, mark=none, thick] table[col sep=comma, each nth point=3, x="arc_length", y="vel_x"] {cylinder-data/fem_data.csv};
        \addlegendentry{Est. Flux $x$};
        \addplot[green, mark=none, thick] table[col sep=comma, each nth point=3, x="arc_length", y="vel_y"] {cylinder-data/fem_data.csv};
        \addlegendentry{Est. Flux $y$};
        \end{axis}
        \end{tikzpicture}
         \caption{Line from $(-1.5, -1.5) \to (1.5, 1.5)$ with $24 \times 24$ FEM in center. }
     \end{subfigure} 
     \hfill
     \begin{subfigure}[b]{0.47\textwidth}
         \centering
        \begin{tikzpicture}
        \begin{axis}[scale=.75,
             ticklabel style = {font=\small},
              legend style={font=\tiny,
              at={(1.05,0.5)},anchor=east}]
        
        \addplot[blue, mark=circle, dash dot, thick, opacity=.5] table[col sep=comma, each nth point=2, x="arc_length", y="grad_x"] {cylinder-data/true_data.csv};
        \addlegendentry{True Flux $x$};
        \addplot[green, mark=circle, dash dot, thick, opacity=.5] table[col sep=comma, each nth point=2, x="arc_length", y="grad_y"] {cylinder-data/true_data.csv};
        \addlegendentry{True Flux $y$};
        \addplot[blue, mark=none, thick] table[col sep=comma, each nth point=2, x="arc_length", y="vel_x"] {cylinder-data/est_data24.csv};
        \addlegendentry{Est. Flux $x$};
        \addplot[green, mark=none, thick] table[col sep=comma, each nth point=2, x="arc_length", y="vel_y"] {cylinder-data/est_data24.csv};
        \addlegendentry{Est. Flux $y$};
        \end{axis}
        \end{tikzpicture}
         \caption{Line from $(-1.5, -1.5) \to (1.5, 1.5)$ with $24 \times 24$ fine scale FEEC in center}
     \end{subfigure}
    \hfill
        \caption{
        Comparison between using FEM (left) and FEEC (right) solvers in the material discontinuity region $[-.5, .5]^2$ for the cylinder problem \cref{sec:cylinder}. 
        Note the overshoot in the discontinuity in the $x$ component of the flux for the pure finite elements case, resulting in a relative error of over $25\%$ near the discontinuity.
        On the other hand, the FEEC element is able to reduce that fluctuation near the discontinuity to less than $5\%$ using the same number of fine-scale knots due to the adaptivity. 
        }
        \label{fig:cyl-fem-comparison}
\end{figure}

\subsection{Example 4: Subdomain Refinement with FEEC}\label{sec:subdomain-refinement-examples}
In this next class of examples, we will consider three separate problems whereby the number of subdomains is increased with no further refinement in either the subdomain-level solver, or the number of mortar degrees of freedom per subdomain. 
This is a non-standard example case in the context of domain decomposition methods, but is extremely useful in the case where machine-learned elements are used.

We hypothesize that smaller subdomains means that there are fewer features for each FEEC element to learn, meaning that the optimization procedure will usually result in smaller local losses. 
The smaller number of features to capture also means that we can use FEEC elements without as many fine-scale nodes, decreasing computational costs in training. 
Furthermore, in the case with large amount of data points, smaller subdomains means that one can speed up the training tremendously as all the training points can now fit on a single GPU. 

In the first two examples, we perform a similar training procedure as before where on each subdomain, a suite of boundary conditions are used to train the local Whitney elements. 
The last example is more representative of a possible usage case where only a single reference solution is provided with realistic multiscale features.

\subsubsection{Stripe Problem}\label{sec:stripes-problem}
Consider data arising from the problem \cref{eqn:primal-big} with $\Omega = [0, n] \times [0, n]$ for $n$ a positive integer, 
\begin{align}
    f := 0, \quad g := x, \qquad K = \kappa_i \mathbf{I}
\end{align}
where $\mathbf{I}$ is the $\mathbb{R}^{2 \times 2}$ identity matrix, and where if $\lfloor y \rfloor$ is even, then 
\begin{align*}
    \kappa_i = \begin{cases}
        1 & 0 \le y < .4\\ 
        .4 & .4 \le y < .8\\
        .8 & .8 \le y < 1
    \end{cases}
\end{align*}
otherwise, 
\begin{align*}
    \kappa_i = \begin{cases}
        .8 & 0 \le y < .2\\ 
        .3 & .4 \le y < .6\\
        .9 & .8 \le y < 1\\
    \end{cases}
\end{align*}
While the true solution for the pressure is trivially $p(x) = x$ for all $n$, the difficulty lies in the ability of the discrete solution to capture the discontinuous velocities 
\begin{align*}
    \mat u(x) := \begin{pmatrix}
        \kappa_i \\ 0
    \end{pmatrix}
\end{align*}
which arises. 

Two FEEC elements of size $[0, 1]^2$ are trained: one to capture the case where $\lfloor y \rfloor$ is even, and another for the odd case. 
For both FEEC elements, a total of $20 \times 20$ fine scale nodes were used, which was subsequently compressed down to 14 POUs on the interior and 14 on the boundary. 
To train the two FEEC systems, we minimize the MSE against only four PDEs corresponding to the Laplace equation $f = 0$ with the boundary conditions $xy, x(1-y), (1-x)y, (1-x)(1-y)$ on 20480 randomly sampled points on $[0, 1]^2$.
As for the mortar space, the lowest order space $H=1$ is used. 
Note that in this case, Assumption 1 is trivially satisfied.

In \cref{fig:stripes-pressure-solution}, we show the solutions of the pressure for $n = 2, 3, 5$.
We see that we recover the true pressure easily as it is just a simple linear function. 
We note that the notion of convergence is not applicable in this case since the domain and problem itself are actually changing as we increase $n$. 

In \cref{fig:stripes-fx-solution}, we show the $x$-component of the gradient; it is clear that the stripes structure is well-preserved even as we introduce more subdomains into the mortar space. 
While the error estimates \cref{lem:error-pre-vel} cannot support this statement due to the usage of crude $L^\infty$ norms, this is indication that, at least numerically, requirement {\bf R2} is satisfied.
We also plot the estimate solution profile on the line $(2.5, 0)$ to $(2.5, 5)$ in \cref{fig:stripes-line} for the case of $n=5$. 
From this view, it's clear that the actual numerical values are in good agreement with the true solution. 

\begin{figure}[ht]
     \centering
     \begin{subfigure}[b]{0.32\textwidth}
         \centering
         \includegraphics[trim={21cm 6cm 20cm 6cm},clip, width=\textwidth]{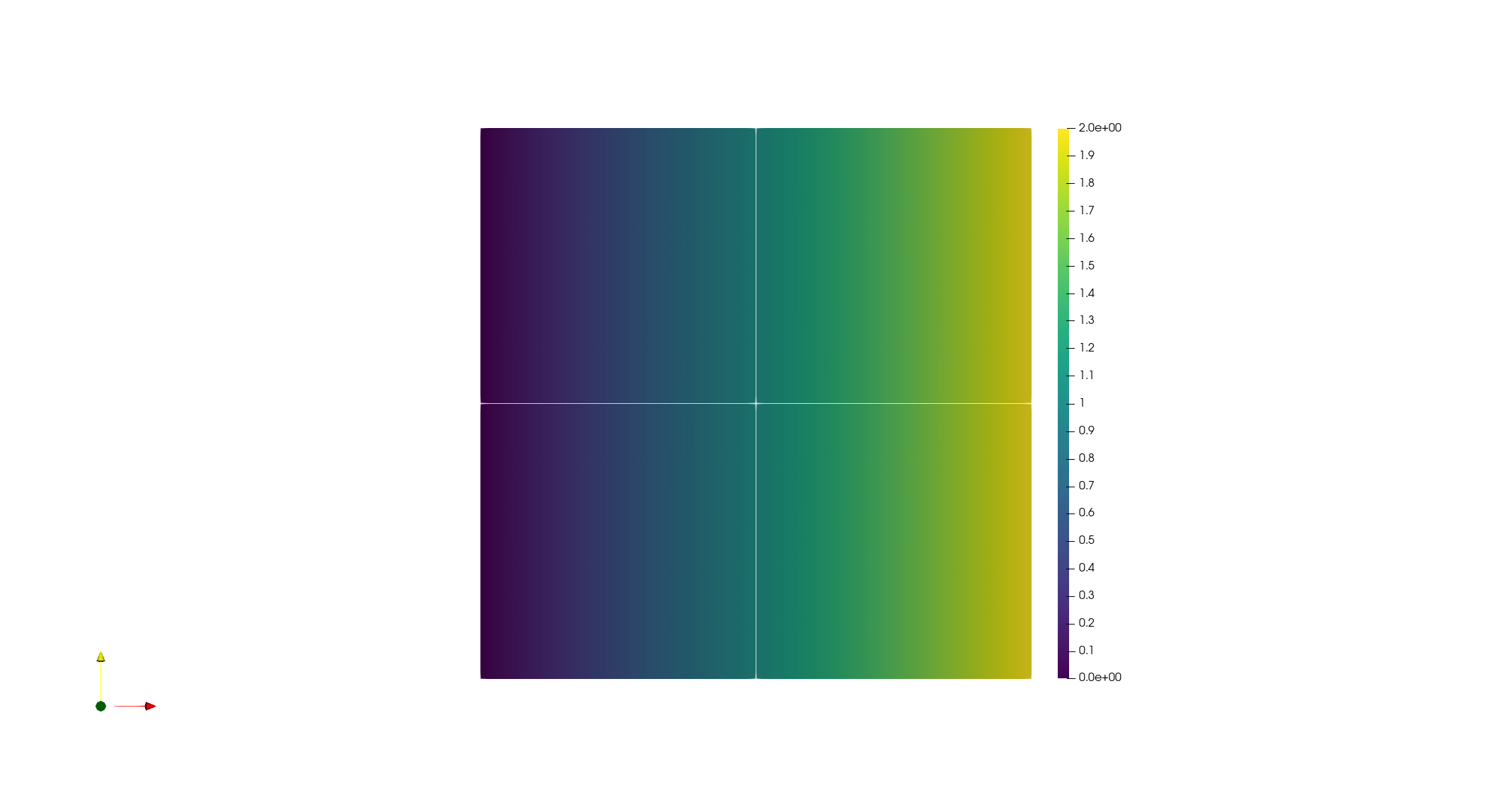}
     \end{subfigure}
     \hfill
     \begin{subfigure}[b]{0.32\textwidth}
         \centering
         \includegraphics[trim={21cm 6cm 20cm 6cm},clip,width=\textwidth]{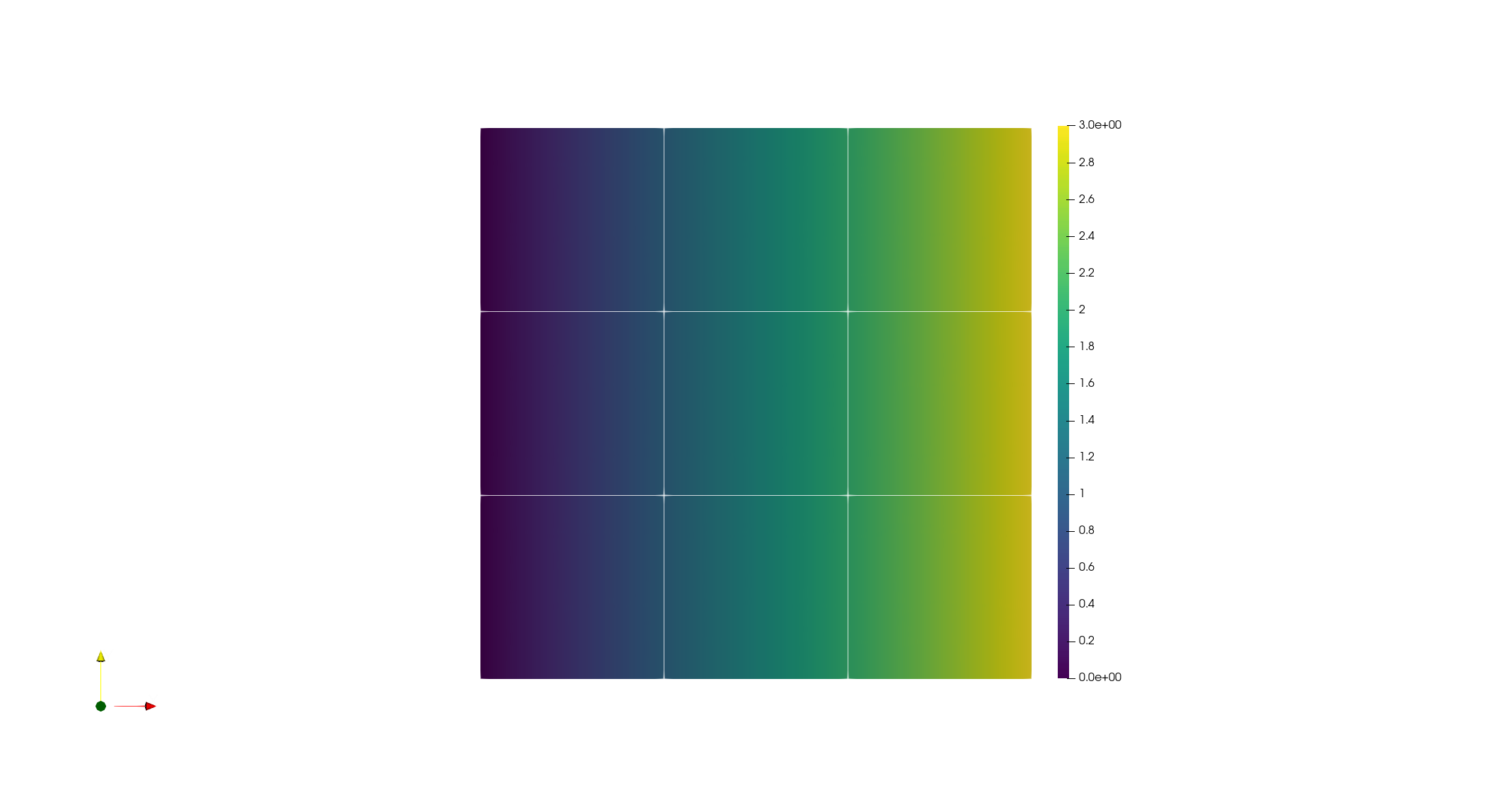}
     \end{subfigure}
     \hfill
     \begin{subfigure}[b]{0.32\textwidth}
         \centering
         \includegraphics[trim={21cm 6cm 20cm 6cm},clip,width=\textwidth]{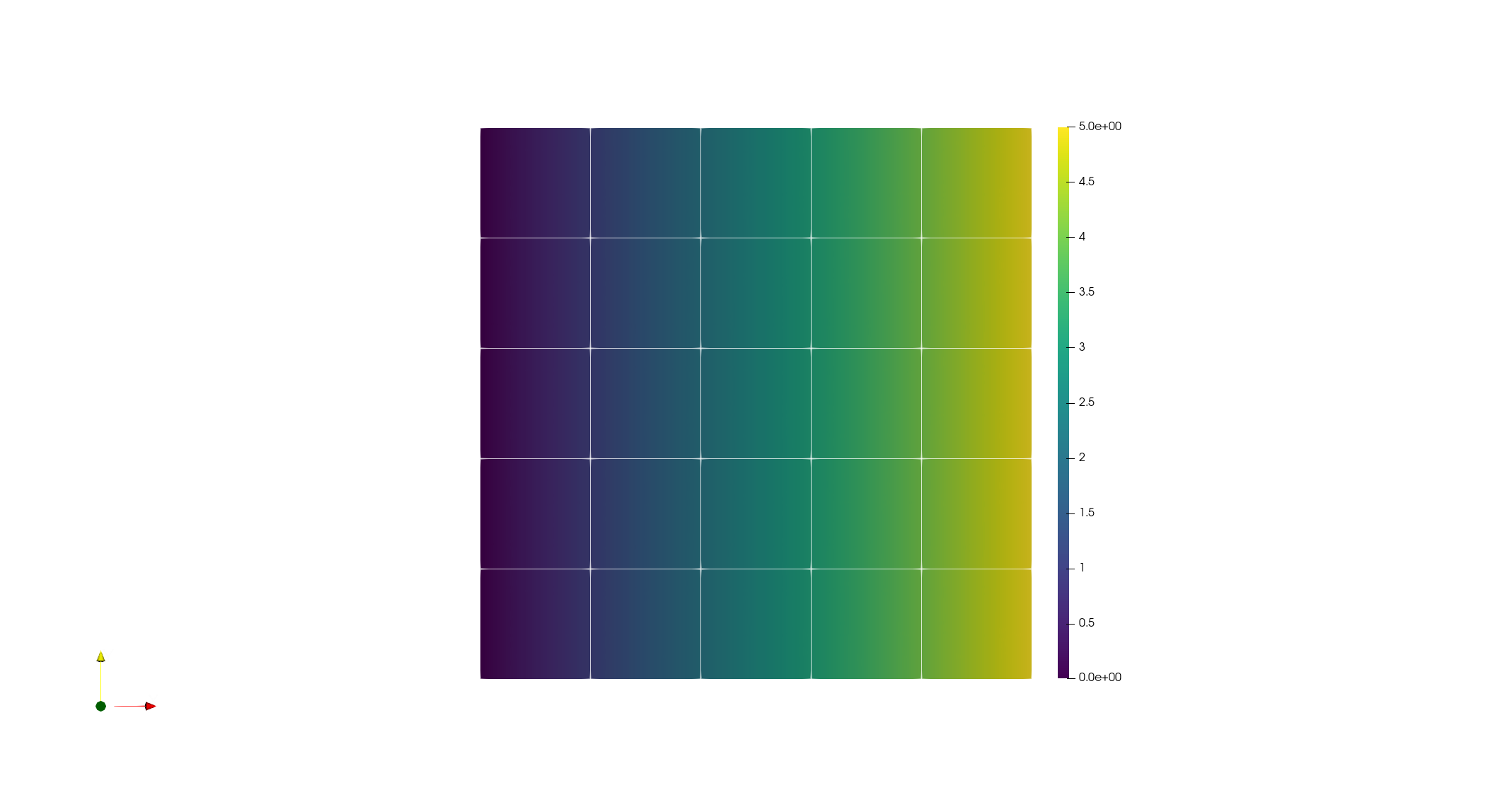}
     \end{subfigure} 
        \caption{
            Figure of the pressure solutions obtained for the stripes problem \cref{sec:stripes-problem} for on increasingly larger domain $[0, 2]^2, [0, 3]^2, [0, 5]^2$ using FEEC elements of $14 \times 14$ fine-scale knots and a very coarse mortar of $H = 1$. 
            Importantly, we note that, from left to right, the domain $\Omega$ of the problem is being increased and we are not depicting a refinement process.
        }
        \label{fig:stripes-pressure-solution}
\end{figure}

\begin{figure}[ht]
     \centering
     \begin{subfigure}[b]{0.32\textwidth}
         \centering
         \includegraphics[trim={21cm 6cm 20cm 6cm},clip, width=\textwidth]{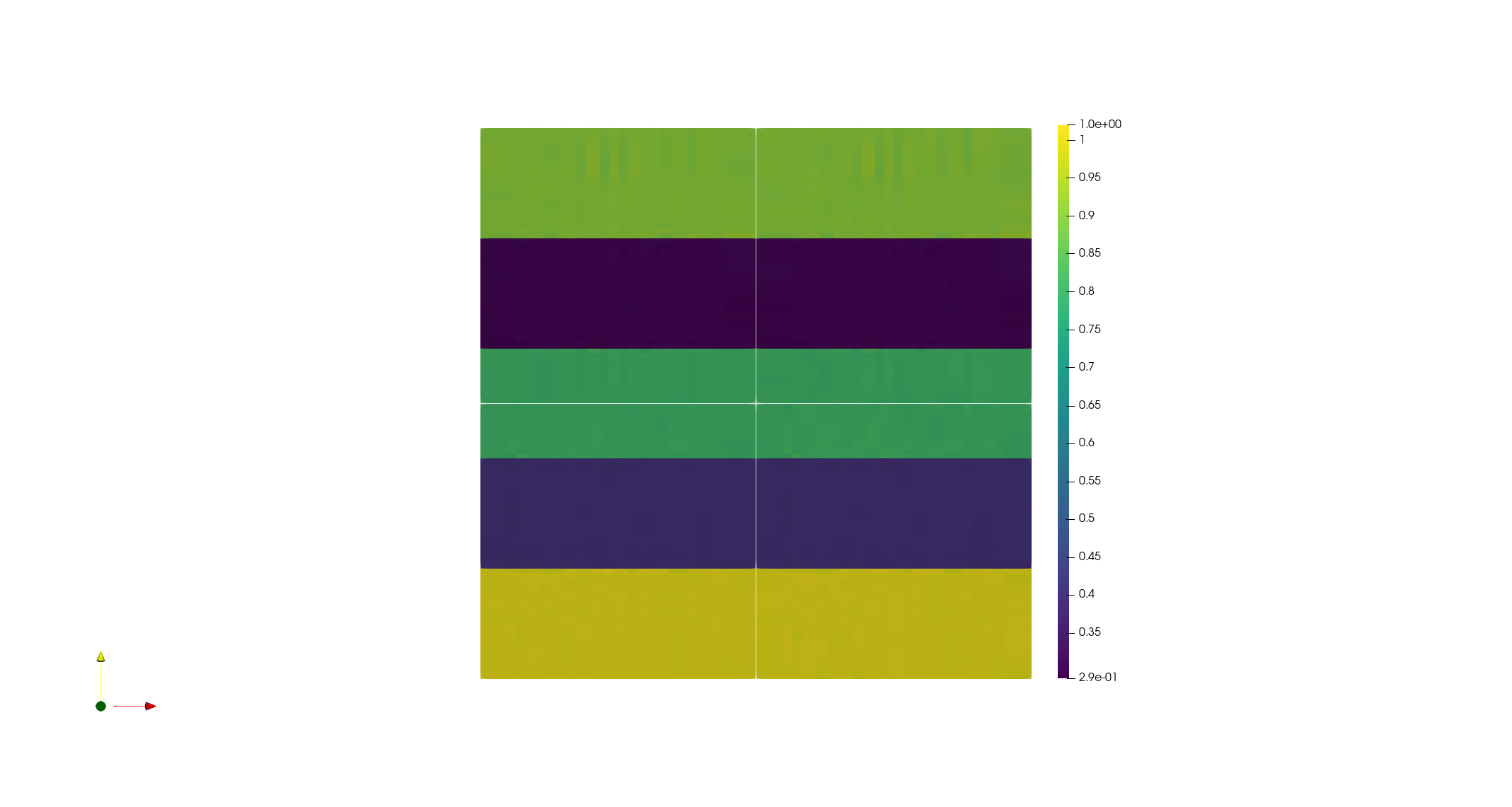}
     \end{subfigure}
     \hfill
     \begin{subfigure}[b]{0.32\textwidth}
         \centering
         \includegraphics[trim={21cm 6cm 20cm 6cm},clip,width=\textwidth]{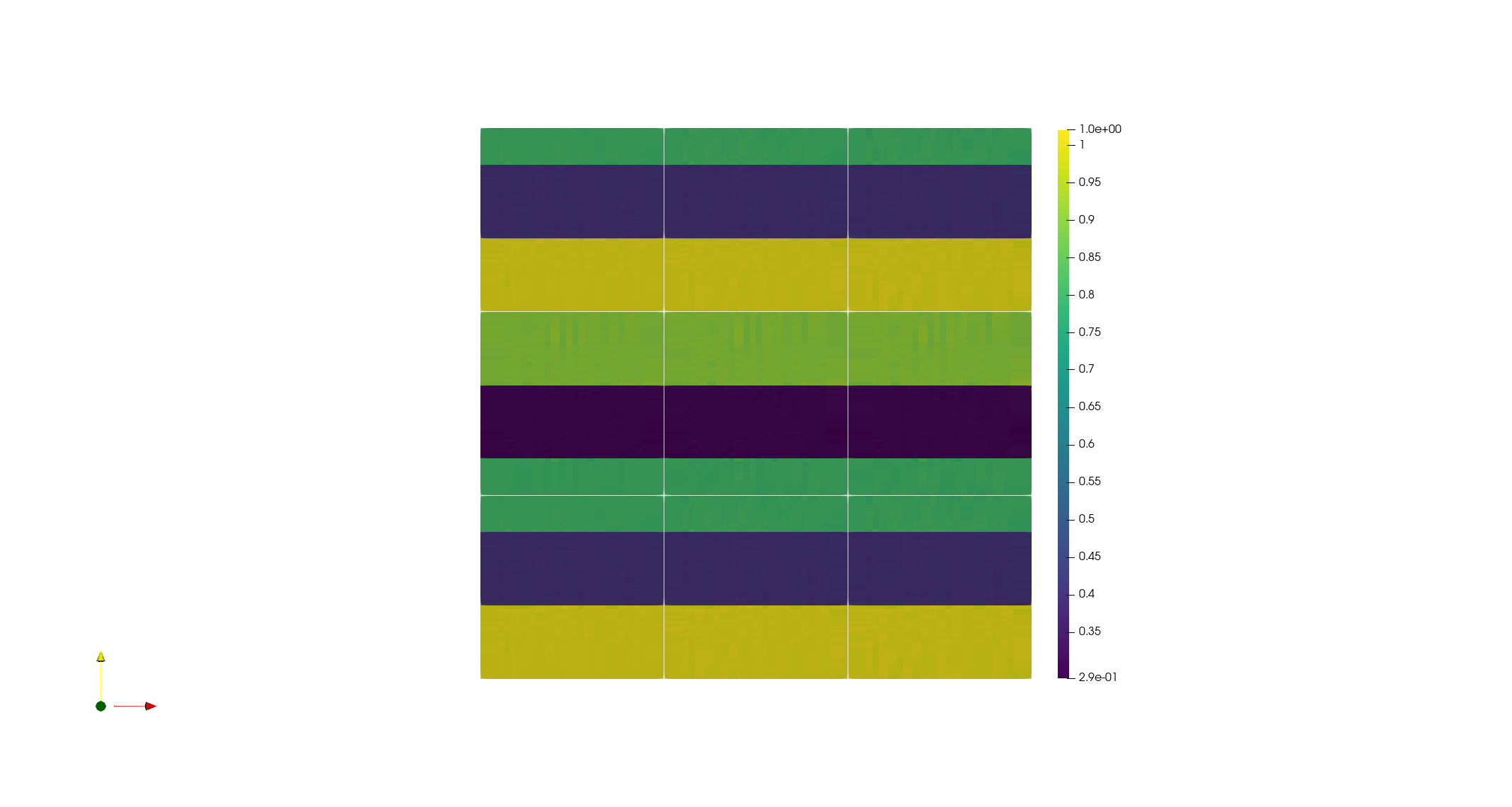}
     \end{subfigure}
     \hfill
     \begin{subfigure}[b]{0.32\textwidth}
         \centering
         \includegraphics[trim={21cm 6cm 20cm 6cm},clip,width=\textwidth]{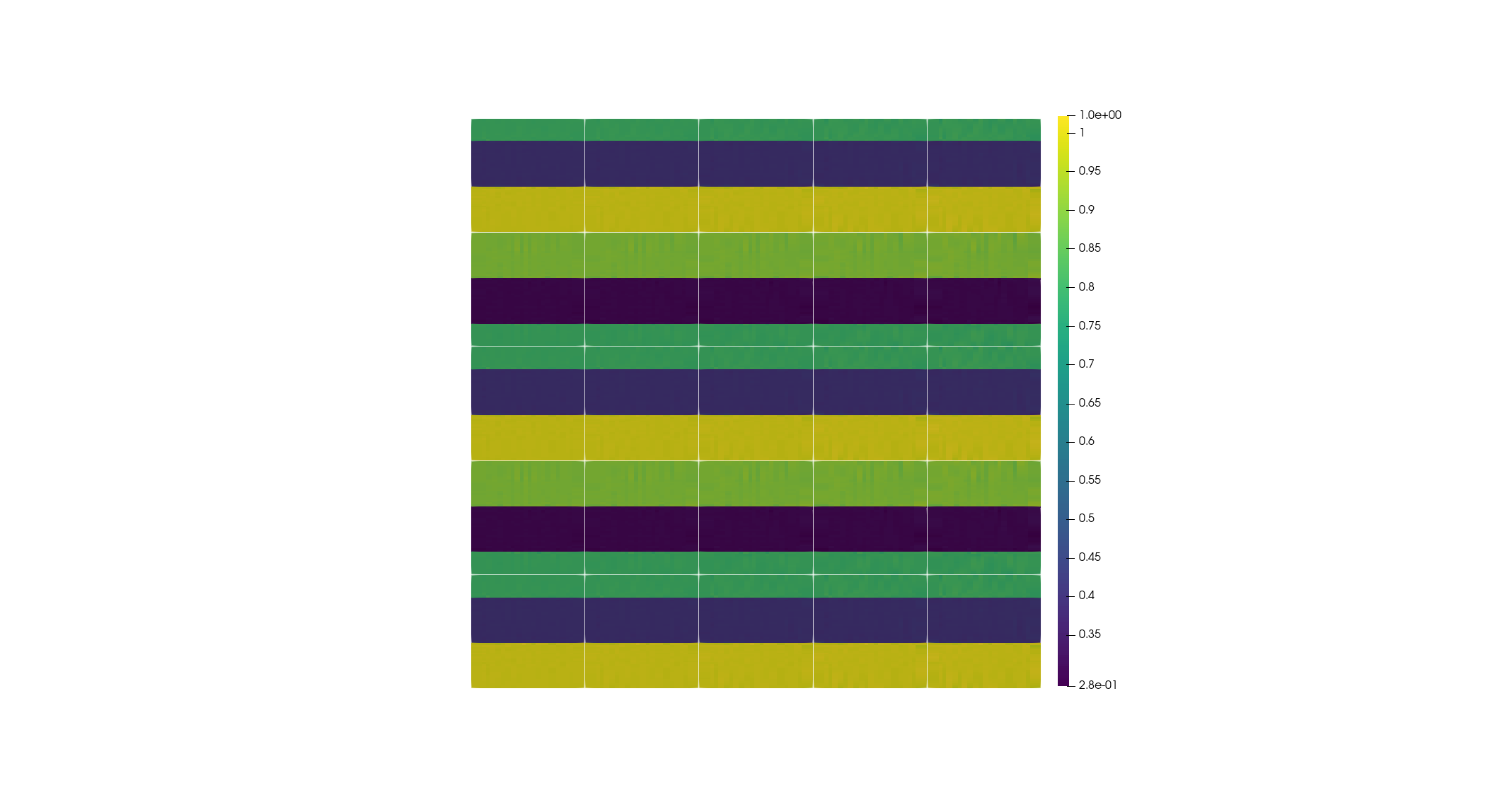}
     \end{subfigure} 
        \caption{
        Figure of the flux in $x$ of the solutions obtained for the stripes problem \cref{sec:stripes-problem} on $[0, 2]^2, [0, 3]^3, [0, 5]^2$ with FEEC elements of $14 \times 14$ fine-scale knots and $H = 1$.  
        We remark that the discontinuities are well-preserved using the FEEC elements even as the domain of the problem is increased. 
        }
        \label{fig:stripes-fx-solution}
\end{figure}

\begin{figure}[ht]
\centering
        \begin{tikzpicture}
        \begin{axis}[scale=.8,
            legend pos=outer north east]

        \addplot[blue, mark=none, thick] table[col sep=comma, each nth point=3, x="arc_length", y="vel_x"] {stripes-data/stripes_data.csv};
        \addlegendentry{Estimated Flux $x$};
        \addplot[black, mark=none, thin] table[x=x, y=y] {stripes-data/stripes_true.data};
        \addlegendentry{True Flux $x$};
        \addplot[green, mark=none, thick] table[col sep=comma, each nth point=3, x="arc_length", y="vel_y"] {stripes-data/stripes_data.csv};
        \addlegendentry{Estimated Flux $y$};
        \addplot[mark=none, thin, red] coordinates {(0, 0) (5, 0)};
        \addlegendentry{True Flux $y$};
        \end{axis}
        \end{tikzpicture}
        \caption{
            Profile on the line  $(2.5, 0)$ to $(2.5, n)$ for $n=5$ of the fluxes of the estimated and true stripes problem \cref{sec:stripes-problem} obtained from the FEEC elements with $14 \times 14$ fine-scale knots and $H=1$ mortar space. 
        }
        \label{fig:stripes-line}
\end{figure}
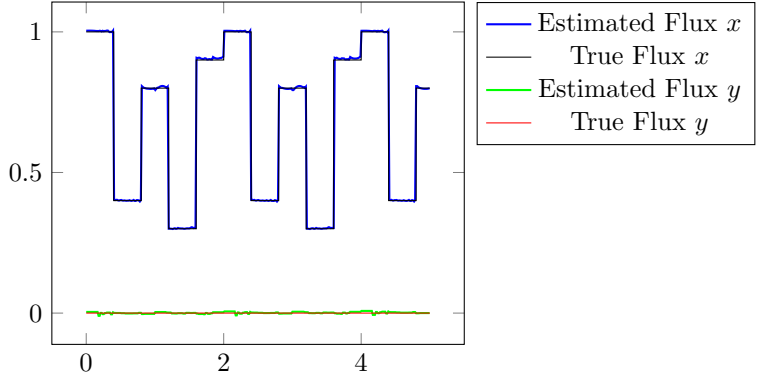

\subsubsection{Path Problem}\label{sec:path-example}
Consider data arising from the problem \cref{eqn:primal-big} on $\Omega = [0, 1] \times [0, 1]$ with $f = 0, g = x$ and 
\begin{align}\label{eqn:path-coef}
    K = \begin{cases}
        \frac{1}{5} \mathbf{I} & x \in \Omega_{\text{path}} \\
        \frac{1}{2} \mathbf{I} & x \in \Omega_{\text{circ}}\\
        \mathbf{I} & x \in \text{elsewhere}
    \end{cases}
\end{align}
where $\Omega_{\text{path}}$ is defined as the region lying in 
\begin{align*}
    R((0, .625), (.375, .875)) \cup R((.375, .125), (.625, .875)) \cup R((.625, .125), (1, .375))
\end{align*}
with $R(p_1, p_2)$ is the rectangle with lower left point $p_1$ and upper right corner $p_2$,
and $\Omega_{\text{circ}}$ are two circles centered at $(.125, .25)$ and $(.875, .75)$ with radius .075. 
See the first column of \cref{fig:subdomain-comparison} for figures of the true solution. 


Let our domain $\Omega$\textbf{} be subdivided into $n^2$ equal squares as our subdomains, and let $H = \frac{1}{4n}$ meaning each subdomain has a total of 16 mortar degrees of freedom.
On each of the subdomains, we train a FEEC element on 20480 uniformly sampled points from the subdomain with 10 fine scale nodes and 14 POUs on the interior and boundary. 
As before, the FEEC elements are trained on 12 total boundary conditions corresponding to the third order Bernstein polynomials on squares.
Rather than refining the mortar discretization relative to the number of subdomains, or increasing the fine-scale nodes on the local solvers, we \emph{strictly increase the number of subdomains in this study.}
We reiterate the fact that as the number of subdomains increases, the number of mortar degrees of freedom per subdomain remains the same at 16 and each FEEC element has the same number of parameters (e.g. 10 fine scale nodes and 14 POUs on the interior).

In \cref{tab:path-error}, we show the average error resulting from increasing the number of subdomains over five different random seeds for training.
We note that while the error in the pressure is already captured quite accurately by a single FEEC element owing to its almost linear nature on the whole domain, the error in the gradient decreases much more dramatically, due to the higher resolution by increasing the number of subdomains.

In \cref{fig:plot-error-path}, we plot the $H^1$ norm errors of both the individual seeds and the mean.
We observe a first-order convergence in the number of subdomains, supporting the notion that our mortar method satisfies requirement {\bf R2} as we increase the number of elements. 
Unfortunately, the error analysis performed in the previous section is not fine enough to show convergence in this case where we increase the number of subdomains due to the usage of crude triangle inequalities.

\begin{table}[ht]
    \centering
    \begin{tabular}{|c|cc|}
        \hline
        Subdomains  & Mean $\norm{p - p_h}_{\Omega}$ & Mean $\norm{\mat u - \mat u_h}_{\Omega}$ \\
        \hline
        $2 \times 2$ & 3.53E-03 $(0.596 \%)$ & 3.47E-02 $(5.21 \%)$\\
        $3 \times 3$ & 3.25E-03 $(0.549 \%)$ & 2.49E-02 $(3.74 \%)$\\
        $4 \times 4$ & 3.13E-03 $(0.528 \%)$ & 1.56E-02 $(2.34 \%)$\\
        $5 \times 5$ & 2.97E-03 $(0.501 \%)$ & 1.33E-02 $(1.99 \%)$\\
        $6 \times 6$ & 2.70E-03 $(0.456 \%)$ & 9.80E-03 $(1.47 \%)$\\
        $8 \times 8$ & 3.28E-03 $(0.554 \%)$ & 6.50E-03 $(0.97 \%)$\\
       \hline
    \end{tabular}
    \caption{Table of average absolute and relative errors for the path problem \cref{sec:path-example} whereby the domain is increasingly subdivided into finer pieces.
    While the pressure does not exhibit convergence, the fluxes converges at a rate of $\mathcal{O}(h)$ and so does the full $H^1$ norm (see \cref{fig:plot-error-path}).}
    \label{tab:path-error}
\end{table}
\begin{figure}[ht]
    \centering
    \begin{tikzpicture}
    \begin{loglogaxis}[scale=.9, xlabel={Number of subdivisions}, ylabel={Absolute $H^1$ error},
    xtick={2, 3, 4, 5, 6, 8},
        xticklabels={2, 3, 4, 5, 6, 8},
        x dir=reverse,
        legend pos=outer north east]
    
    \addplot[only marks, black] table[x=N, y=avgcalc] {data/path.data};
    \addlegendentry{Average $H^1$ error}

    \addplot[only marks, black, opacity=0.4, mark=x, mark size=4] table[x=N, y=seed1] {data/path.data};
    \addplot[only marks, black, opacity=0.4, mark=x, mark size=4] table[x=N, y=seed2] {data/path.data};
    \addplot[only marks, black, opacity=0.4, mark=x, mark size=4] table[x=N, y=seed3] {data/path.data};
    \addplot[only marks, black, opacity=0.4, mark=x, mark size=4] table[x=N, y=seed4] {data/path.data};
    \addplot[only marks, black, opacity=0.4, mark=x, mark size=4] table[x=N, y=seed5] {data/path.data};
    \addlegendentry{Individual $H^1$ error}
    \logLogSlopeTriangle{0.9}{0.2}{0.1}{1}{black};
    
    \end{loglogaxis}
    \end{tikzpicture}
    \caption{Plot of the absolute $H^1$ error resulting from refinement for the path problem \cref{sec:path-example}. 
    We observe a linear convergence rate by dividing the domain into increasingly smaller domains for the FEEC problem.}
    \label{fig:plot-error-path}
\end{figure}
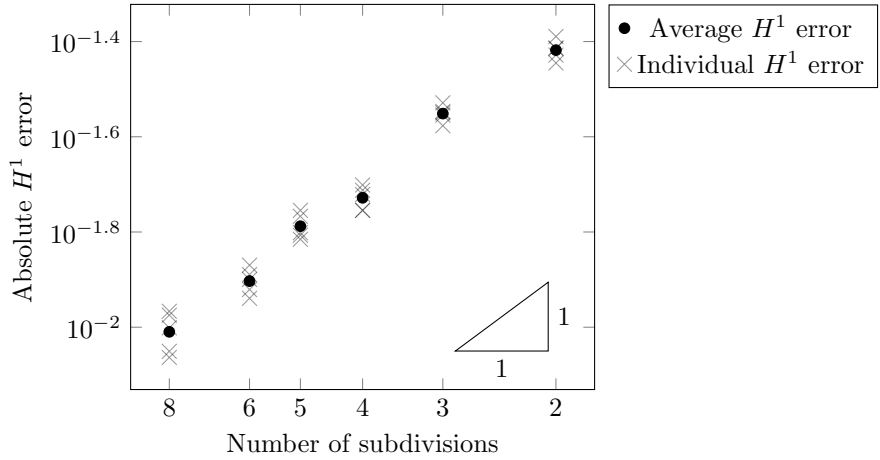

\begin{figure}[ht]
     \centering
     \begin{subfigure}[b]{0.32\textwidth}
         \centering
         \includegraphics[trim={5cm 2cm 2cm 1.5cm},clip, width=\textwidth]{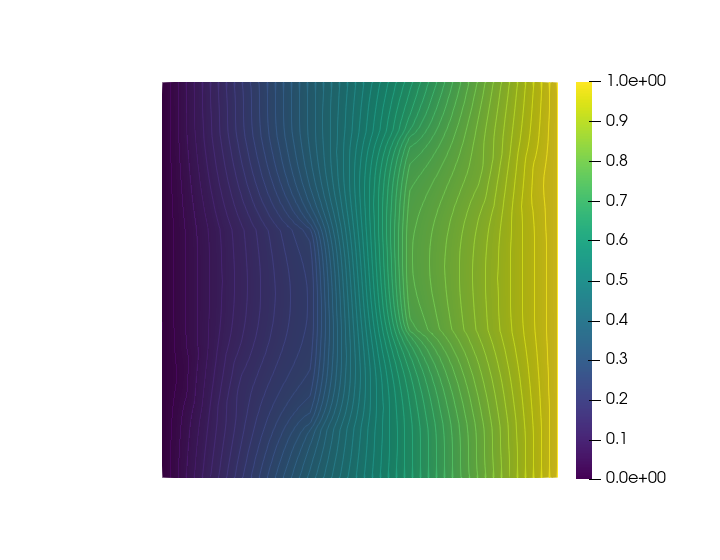}
     \end{subfigure}
     \hfill
     \begin{subfigure}[b]{0.32\textwidth}
         \centering
         \includegraphics[trim={5cm 2cm 2cm 1.5cm},clip,width=\textwidth]{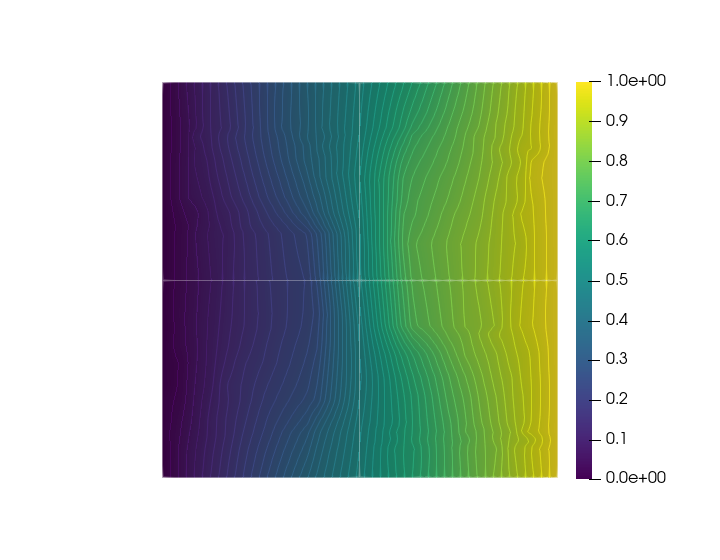}
     \end{subfigure}
     \hfill
     \begin{subfigure}[b]{0.32\textwidth}
         \centering
         \includegraphics[trim={5cm 2cm 2cm 1.5cm},clip,width=\textwidth]{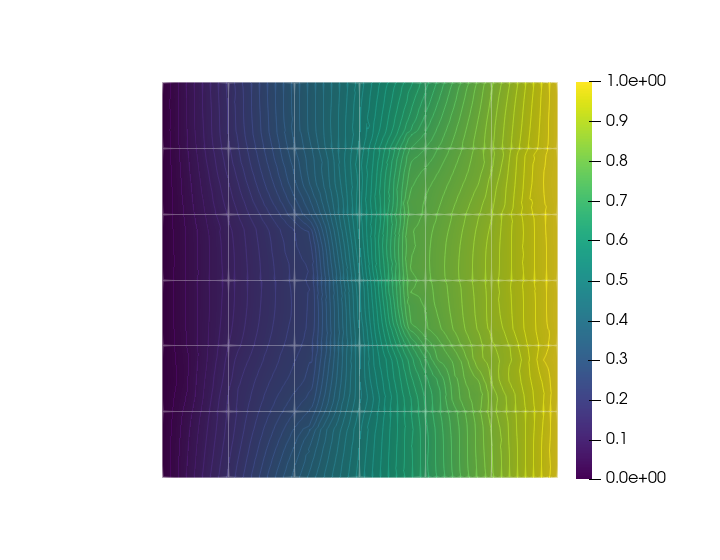}
     \end{subfigure} \\
          \begin{subfigure}[b]{0.32\textwidth}
         \centering
         \includegraphics[trim={5cm 2cm 2cm 1.5cm},clip, width=\textwidth]{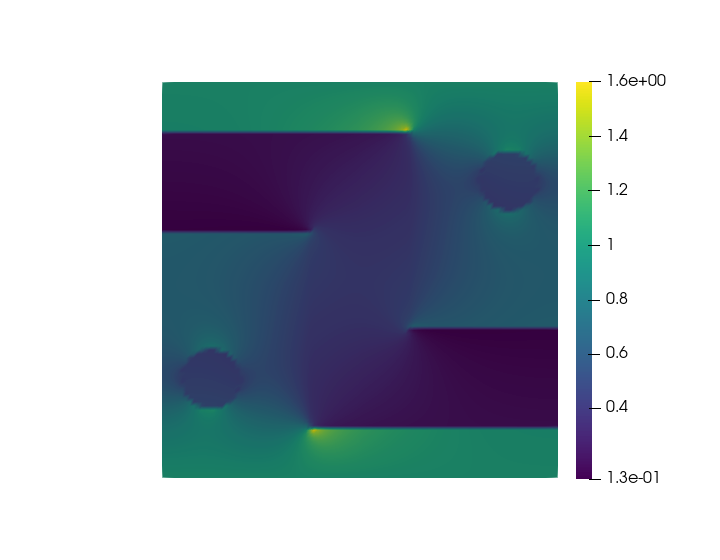}
     \end{subfigure}
     \hfill
     \begin{subfigure}[b]{0.32\textwidth}
         \centering
         \includegraphics[trim={5cm 2cm 2cm 1.5cm},clip,width=\textwidth]{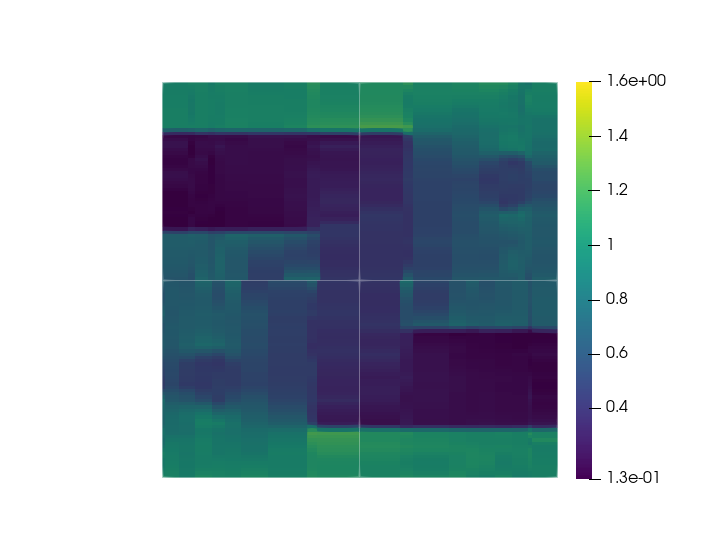}
     \end{subfigure}
     \hfill
     \begin{subfigure}[b]{0.32\textwidth}
         \centering
         \includegraphics[trim={5cm 2cm 2cm 1.5cm},clip,width=\textwidth]{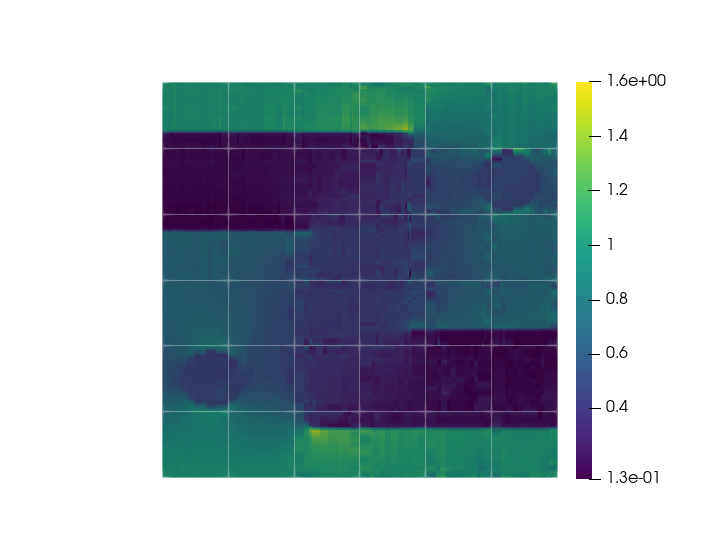}
     \end{subfigure} \\
               \begin{subfigure}[b]{0.32\textwidth}
         \centering
         \includegraphics[trim={5cm 2cm 2cm 1.5cm},clip, width=\textwidth]{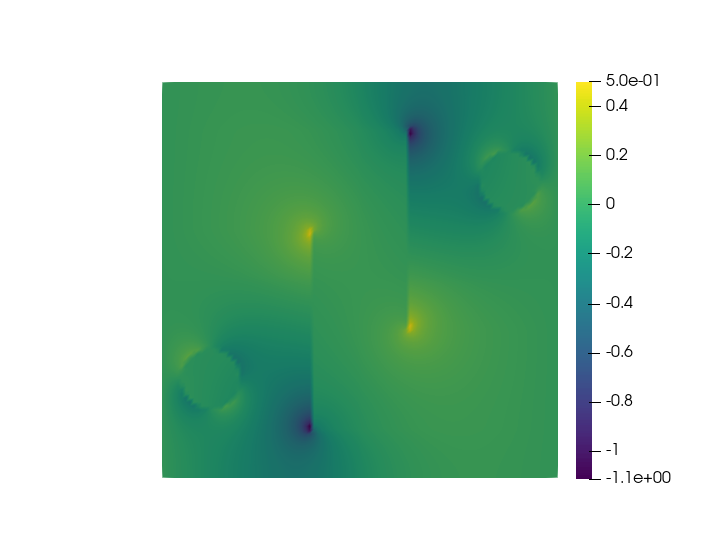}
     \end{subfigure}
     \hfill
     \begin{subfigure}[b]{0.32\textwidth}
         \centering
         \includegraphics[trim={5cm 2cm 2cm 1.5cm},clip,width=\textwidth]{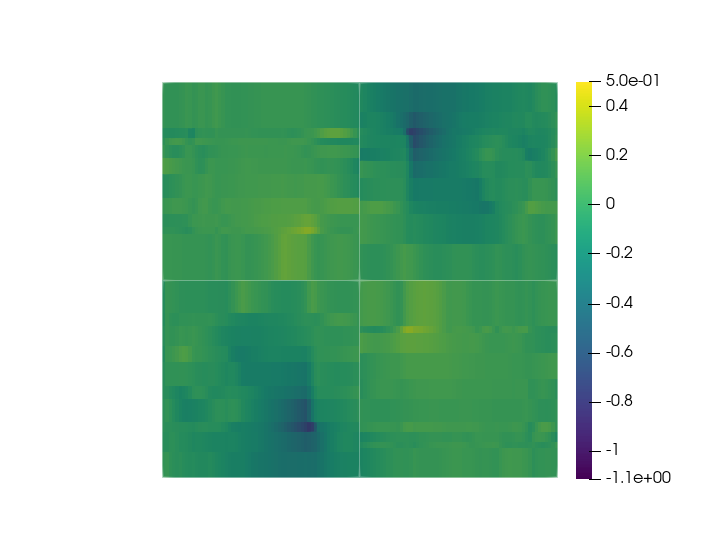}
     \end{subfigure}
     \hfill
     \begin{subfigure}[b]{0.32\textwidth}
         \centering
         \includegraphics[trim={5cm 2cm 2cm 1.5cm},clip,width=\textwidth]{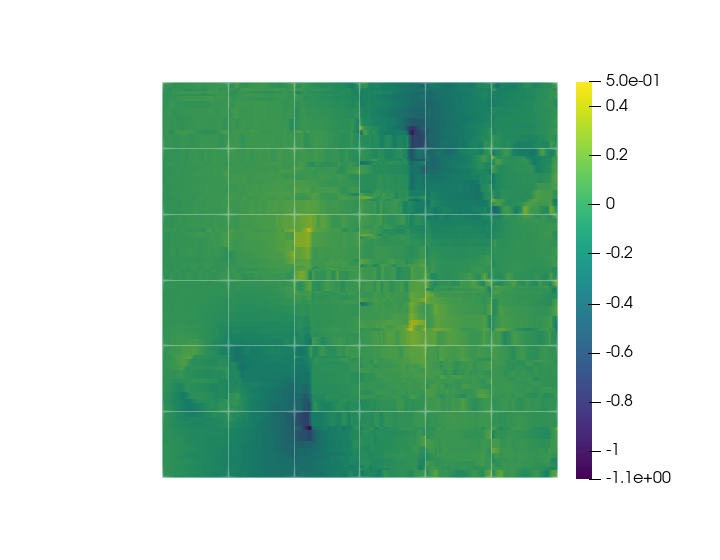}
     \end{subfigure} 
        \caption{
            Plot of the true solution (first column), and subdomain with 2 (second column) and 6 (third column) refinements for the path problem \cref{sec:path-example}. 
            Note that the features are increasingly more refined and matches the true solutions as the number of subdomains are increased.
        }
        \label{fig:subdomain-comparison}
\end{figure}

\subsubsection{Battery Problem: Single solution training}\label{sec:battery-example}
We now consider data from the problem \cref{eqn:primal-big} on $\Omega = [0, 1] \times [0, 1]$ with $f = 0$ and a nontrivial material data and boundary condition corresponding to a voltage difference across a lithium-ion battery.
The true pressure and fluxes, which are sampled at 5.89 million points, are provided via a high-fidelity solver SIERRA/ARIA \cite{notz2016sierra} and will be treated as the \emph{only} source of provided data with no additional methods of augmentation. 
In other words, we assume a full simulation of the response for the subdomains to arbitrary mortars is not available, meaning the local FEEC elements will have to extrapolate the correct Dirichlet-to-Neumann maps. 
For a figure of the true pressure and flux, see \cref{fig:battery-comparison}.
More details regarding the data can be found in appendix B of \cite{actoradata}; note that for simplicity, we consider the problem as a purely Dirichlet boundary condition problem whilst \cite{actoradata} included Neumann boundary conditions. 

We again split the domain $[0, 1]^2$ into $n^2$ uniform squares, but only employ four mortar degrees of freedom per subdomain with $H = \frac 1 n$ (i.e. the mortar degrees of freedom lie on the corners of the subdomain).\footnote{The coarsest mortar mesh is chosen since the fine scale nodes may move substantially, due to only one training set, and violate assumption \cref{eqn:injective-assumption}.}
A FEEC element with 12 fine scale nodes in both the $x$ and $y$ direction, and 12 POUs on the interior and boundary are used on each subdomain.

Since only a single reference solution is provided, we train the FEEC element with boundary condition obtained from interpolating the given solution and the data given (e.g. $g_i = p|_{\partial \Omega_i}$).
For example, suppose $n=2$, then the FEEC element on subdomain corresponding to $\Omega' = [0, .5]^2$ will have $\frac{5.89}{4} \approx 1.5$ million data points, and boundary conditions corresponding to the nearest neighbor interpolation of those points on $\partial \Omega'$. 
This is unlike \cref{sec:path-example} or even \cref{sec:stripes-problem} where each FEEC element was provided with multiple examples to train on. 
Note that the number of training data points per FEEC element decrease as we increase the number of subdomains, we have found that it can lead to some instability in pretraining.

In \cref{tab:battery-error}, we show the absolute MSE of the $L^2$ and the $H^1$ semi-norm resulting from solving the Darcy's flow equation with the trained FEEC elements. 
In the case of $2 \times 2$ refinement, the error is quite large since the mortar only has one degree of freedom in the interior (cf. \cref{fig:mesh-example1}) and the boundary conditions are not even well-resolved; however, it's clear that as additional refinements are made that the relative error decreases. 
In addition, we also show the absolute MSE of the ``true mortar'' (TM) which is obtained by setting the mortar degrees of freedom to be the interpolant from the data set.
This ``true mortar'' indicates how much of the error is due to the training procedure as no actual solves of the bilinear form is performed and allows us to see how much error arises from the actual mortar coupling.
Since this true mortar errors are similar to the errors obtained from solving the bilinear form, this suggests that very little error arises due to the mortar coupling.
In \cref{fig:plot-error-battery}, we observe that the error obtained from solving the Darcy flow equation decreases as we increase the number of subdomains, with the finest level obtaining a better $H^1$ error than the errors obtained in \cite{actoradata}.

\begin{table}[ht]
    \centering
    \begin{tabular}{|c|cc|cc|}
        \hline
        Subdomains  & $L^2(\Omega)$ & $H^1$-seminorm &  TM $L^2(\Omega)$ & TM $H^1$-seminorm\\
        \hline
        $2 \times 2$ & 7.15E-03 ($1.27 \%$)& 1.22E+00 ($67.6\%$) & 5.37E-03 & 1.39E+00\\
        $3 \times 3$ & 3.01E-03	($0.54 \%$)& 6.45E-01 ($35.7\%$) & 2.76E-03 & 6.03E-01\\
        $4 \times 4$ & 2.71E-03	($0.48 \%$)& 4.57E-01 ($25.3\%$) & 2.42E-03 & 3.27E-01\\
        $6 \times 6$ & 2.46E-03	($0.44\%$)& 2.44E-01 ($13.5\%$) & 1.67E-03 & 1.37E-01 \\
        $8 \times 8$ & 2.40E-03	($0.43\%$)& 1.43E-01 ($7.92\%$) & 1.41E-03 & 1.19E-01\\
       \hline
    \end{tabular}
    \caption{Table of absolute and relative errors for the battery problem \cref{sec:battery-example}. 
    The right ``true mortar'' (TM) columns essentially capture the training error by simply fixing the mortar space to the true values, while the left columns result from actually solving the Darcy's flow equations. 
    Similar to the \cref{sec:path-example}, the error in pressure only decreases slightly with most of the benefits arising from the convergence in the $H^1$-seminorm. 
    }
    \label{tab:battery-error}
\end{table}

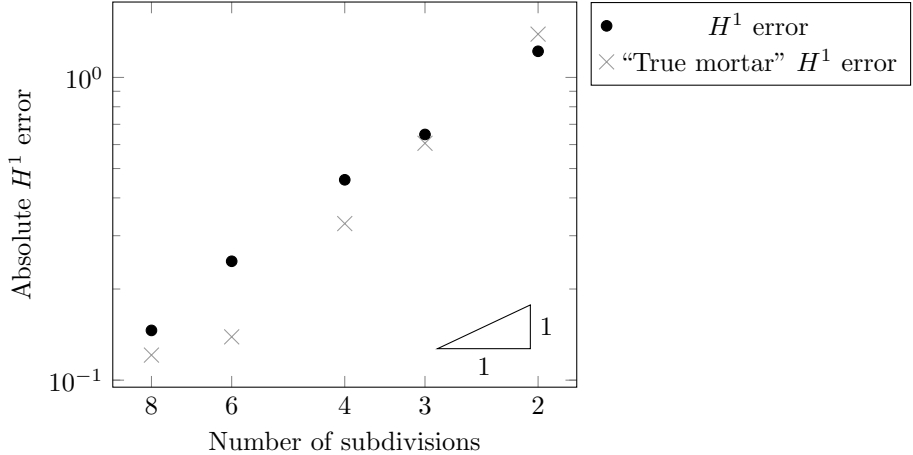
\begin{figure}
    \centering
    \begin{tikzpicture}
    \begin{loglogaxis}[scale=.9, xlabel={Number of subdivisions}, ylabel={Absolute $H^1$ error},
    xtick={2, 3, 4, 6, 8},
        xticklabels={2, 3, 4, 6, 8},
        x dir=reverse,
        legend pos=outer north east]
    
    \addplot[only marks, black] table[x=N, y=est] {data/battery.data};
    \addlegendentry{$H^1$ error}

    \addplot[only marks, black, opacity=0.4, mark=x, mark size=4] table[x=N, y=true] {data/battery.data};
    \addlegendentry{``True mortar'' $H^1$ error}
    \logLogSlopeTriangle{0.9}{0.2}{0.1}{1}{black};
    
    \end{loglogaxis}
    \end{tikzpicture}
    \caption{Plot of the $H^1$ error and true mortar error resulting from refinement for the battery problem \cref{sec:battery-example}.
    For an explanation of what the true mortar error is, we refer the reader to the corresponding discussion \cref{sec:battery-example}.
    We note that the error is quite close to the true mortar error, meaning that the coarse mortar space does not negatively effect the convergence that much.}
    \label{fig:plot-error-battery}
\end{figure}

\begin{figure}[ht]
     \centering
     \begin{subfigure}[b]{0.32\textwidth}
         \centering
         \includegraphics[trim={4.5cm 1.5cm 0.5cm 1cm},clip, width=\textwidth]{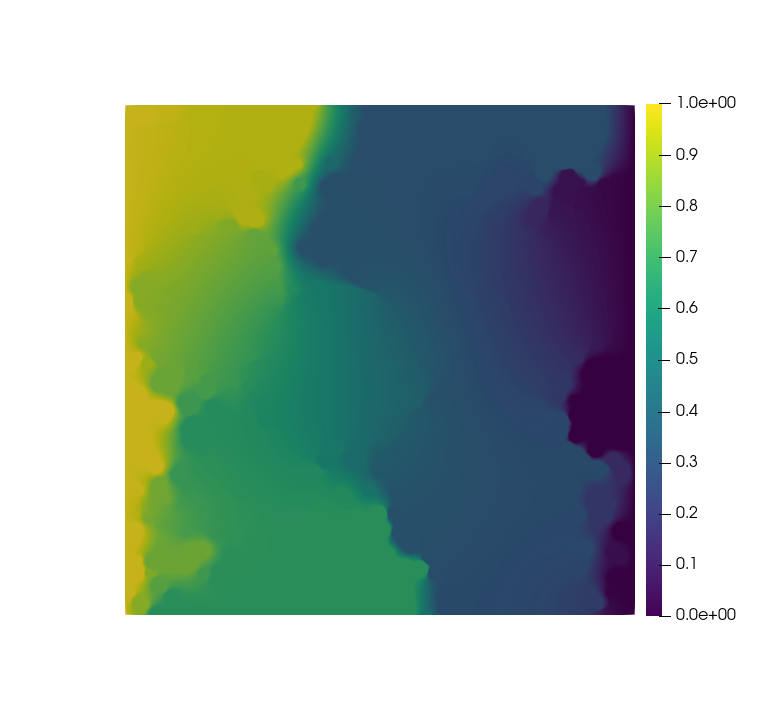}
     \end{subfigure}
     \hfill
     \begin{subfigure}[b]{0.32\textwidth}
         \centering
         \includegraphics[trim={4.5cm 1.5cm 0.5cm 1cm},clip,width=\textwidth]{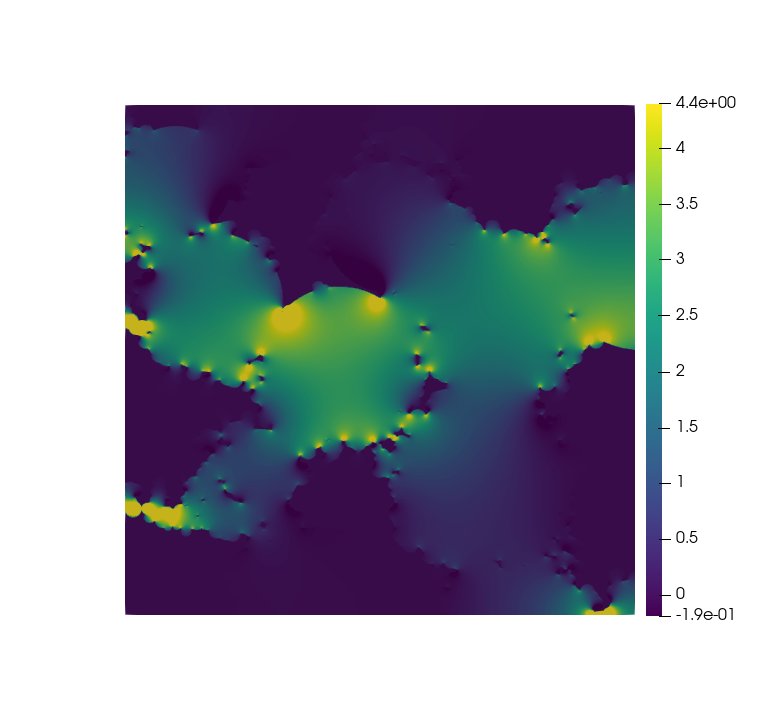}
     \end{subfigure}
     \hfill
     \begin{subfigure}[b]{0.32\textwidth}
         \centering
         \includegraphics[trim={4.5cm 1.5cm 0.5cm 1cm},clip,width=\textwidth]{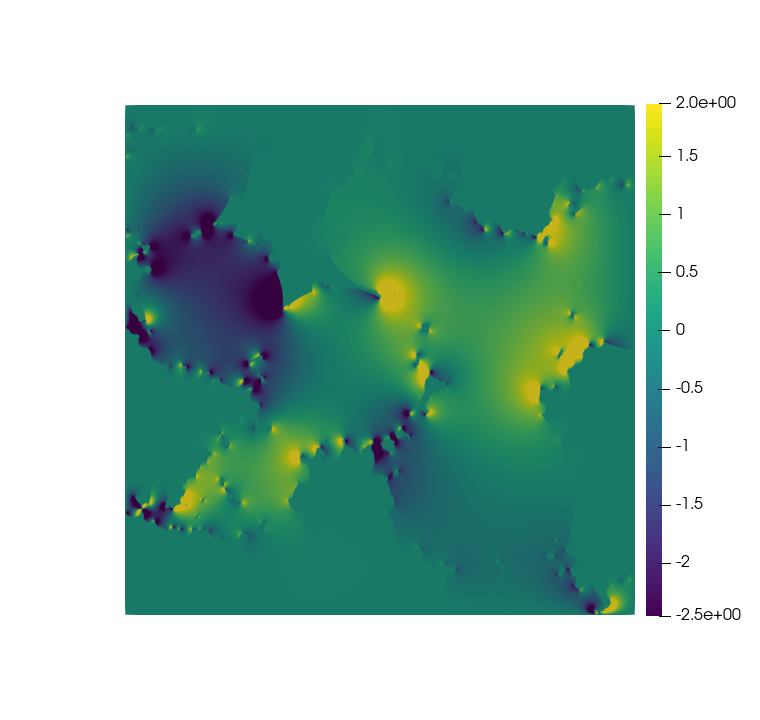}
     \end{subfigure} \\
     \begin{subfigure}[b]{0.32\textwidth}
         \centering
         \includegraphics[trim={4.5cm 1.5cm 0.5cm 1cm},clip, width=\textwidth]{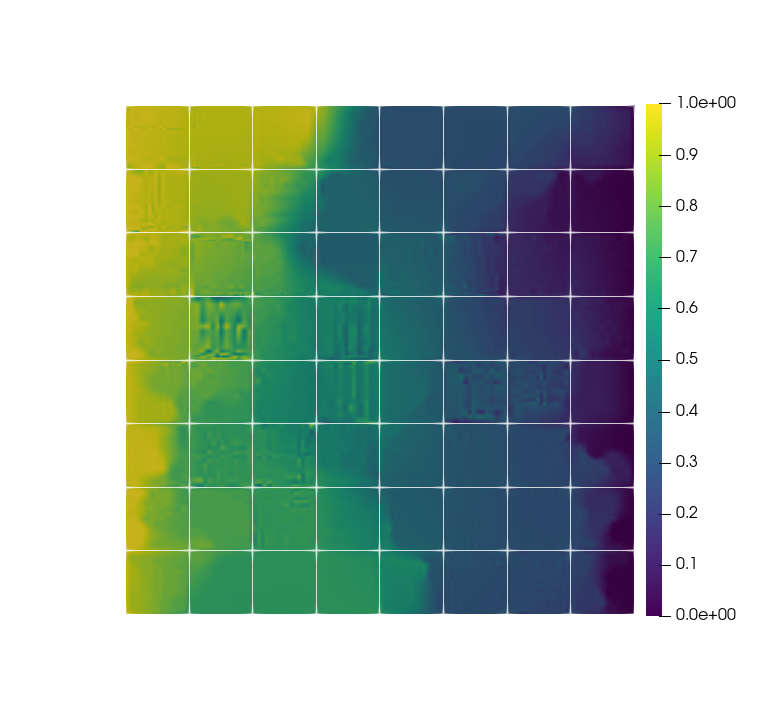}
     \end{subfigure}
     \hfill
     \begin{subfigure}[b]{0.32\textwidth}
         \centering
         \includegraphics[trim={4.5cm 1.5cm 0.5cm 1cm},clip,width=\textwidth]{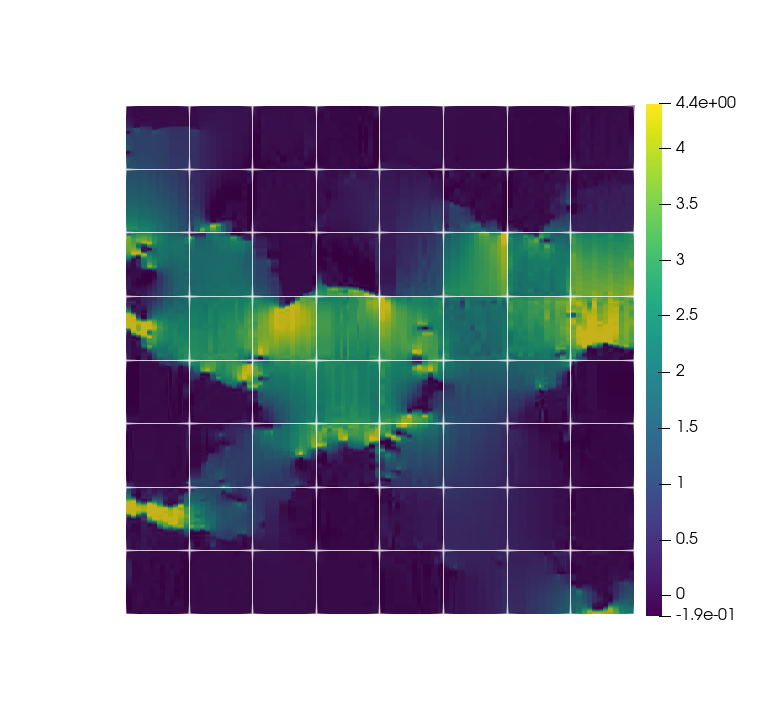}
     \end{subfigure}
     \hfill
     \begin{subfigure}[b]{0.32\textwidth}
         \centering
         \includegraphics[trim={4.5cm 1.5cm 0.5cm 1cm},clip,width=\textwidth]{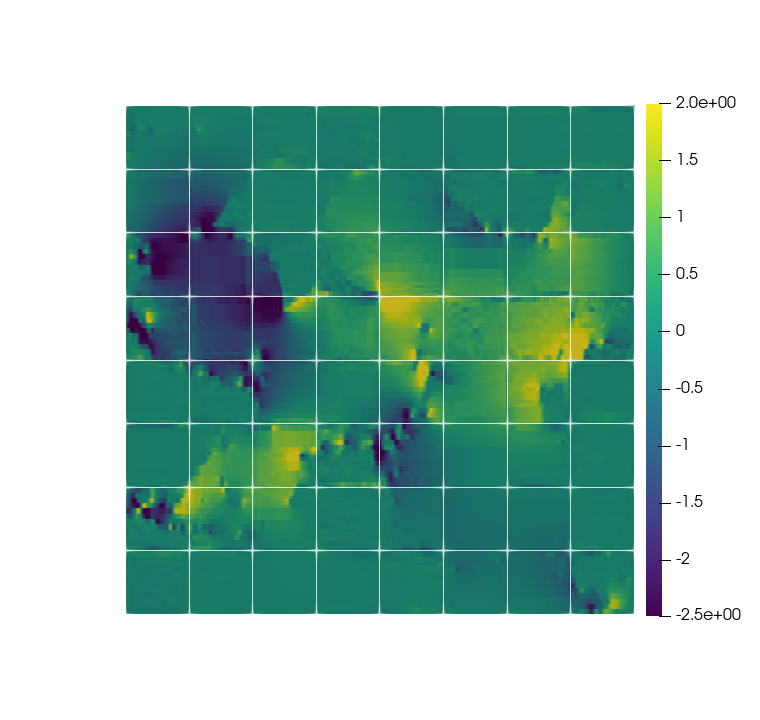}
     \end{subfigure} 
     \caption{Figures of the given pressure/fluxes for the battery problem \cref{sec:battery-example} in the first row, and the approximation obtained from solving the Darcy's flow problem in the second row for $8 \times 8$ subdivisions. 
     Overall, the estimated solution matches the data fairly accurately with many most small details captured.}
    \label{fig:battery-comparison}
\end{figure}

\section{Appendix}\label{sec:appendix}
\subsection{Technical Proofs}
\begin{proof}[Proof of \cref{lem:continuous-coer}]
 By \cref{eqn:decomposition}, for any $\mu \in H_0^\gamma(\Omega)$, we can decompose it as 
 \begin{align*}
 	\mu = p^*(\mu) + \sum_{i=1}^n p_{i0}
 \end{align*}
 where $p^*(\mu) \in H^1_0(\Omega)$ satisfying \cref{eqn:star} (hence $p^*(\mu)|_{\Gamma} = \mu|_{\Gamma}$) and $p_{i0} \in H_0^1(\Omega_i)$ are bubble functions. 

 Thus,  
\begin{align}\label{eqn:b-useful-identity}
\begin{split}
	 	b(\lambda, \mu) &= \sum_{i=1}^n \left(\mat u^*(\lambda), \nabla  p^*(\mu) + \sum_{i=1}^n \nabla p_{i0}\right)_{\Omega_i} \\
 	&= \sum_{i=1}^n \left(\mat u^*(\lambda), \nabla p^*(\mu) \right)_{\Omega_i} = \sum_{i=1}^n \left(K \nabla p^*(\lambda), \nabla p^*(\mu) \right)_{\Omega_i}
\end{split}
 \end{align}
 since \cref{eqn:star} implies the inner products of $\mat u^*(\lambda)$ with the gradient of bubble functions are zero. 
From the above, the bilinear form is clearly symmetric and positive definite.

For coercivity, using \cref{eqn:b-useful-identity}, Poincare inequality and trace inequality \cite{evans2022partial}, 
\begin{align*}
 	b(\lambda, \lambda) &\ge \norm{\nabla p^*(\lambda)}^2 
	\ge \frac{1}{C}\norm{p^*(\lambda)}_{H^1(\Omega)}^2 = \frac{1}{C} \sum_{i=1}^n \norm{p^*(\lambda)}_{H^1(\Omega_i)}^2 
 \ge \frac{1}{C} \sum_{i=1}^n \norm{\lambda}_{H^{1/2}(\Gamma_i)}^2. 
\end{align*}
meaning $b(\lambda, \lambda) \ge \alpha\sum_{i=1}^n\norm{\lambda}_{H^{1/2}(\Gamma_i)}^2 \sim \alpha\norm{\lambda}_{H^\gamma}^2$ for some constant $\alpha$ independent of $\lambda$. 
\end{proof}

The remaining proofs are for the coercivity and the error estimate for the discrete mortar.
We introduce the shorthand notation $p_h^*(Q\lambda_H) := \sum_{i=1}^n p^*_h(Q_i\lambda_H)$, and let $\norm{\cdot}_\Omega$ denote the $L^2$ norm over the domain $\Omega$ unless otherwise stated. 
Before proceeding, we define the inclusion map $P_i: \Lambda_H \subset \Lambda \to H^\gamma|_{\Omega_i}$ through the isomorphism. 
We need a preparatory lemma: 
\begin{lemma}\label{lem:helper}
Let $\delta := \frac{2}{C_p + 1}$, where $C_p$ is the constant arising from Corollary 6.3 of \cite{brenner2003poincare}, then
\begin{align*}
	\frac{\delta}{2}\norm{p_h^*(Q \lambda_H)}_{H^1(\Omega)}^2 &\le 
	\left[\norm{\nabla p^*_h(Q\lambda_H)}_{\Omega}^2 + \frac{C_p}{C_p + 1}\sum_{\Gamma_{ij}} \frac{1}{\abs{\Gamma_{ij}}} \norm{Q_i\lambda_H - Q_j \lambda_H}^2_{L^2(\Gamma_{ij})} \right].
\end{align*}
\end{lemma}
\begin{proof}
By a simple application of Corollary 6.3 of \cite{brenner2003poincare}: 
	\begin{align*}
	\norm{p_h^*(Q \lambda_H)}_{H^1(\Omega)}^2 &\le C_p\left[(1 + \frac{1}{C_p})\norm{\nabla p^*_h(Q\lambda_H)}_{\Omega}^2 + \sum_{\Gamma_{ij}} \frac{1}{\abs{\Gamma_{ij}}^2} \left(\int_{\Gamma_{ij}} Q_i\lambda_H - Q_j \lambda_H \, ds \right)^2 \right] \\
	 &\le C_p\left[(1 + \frac{1}{C_p})\norm{\nabla p^*_h(Q\lambda_H)}_{\Omega}^2 + \sum_{\Gamma_{ij}} \frac{1}{\abs{\Gamma_{ij}}} \norm{Q_i\lambda_H - Q_j \lambda_H}^2_{L^2(\Gamma_{ij})} \right]
\end{align*}
where we used Cauchy-Scwharz inequality on $(\int_\sigma f)^2 \le \abs{\sigma}\norm{f}^2$.
\end{proof}

\begin{proof}[Proof of \cref{lem:dis-stability}]
    An identity like \cref{eqn:b-useful-identity} can also be verified for the discrete version as well since on any subdomain $i$ and $\mu_H \in \Lambda_H$, $Q_i\mu_H = p_h^*(Q_i\mu_H) + p_{hi}$ where $p_{hi} \in W_{hi, 0}$ bubble functions, one has
    \begin{align}\label{eqn:b-useful-identity-disc}
\begin{split}
	 	b_h(\lambda_H, \mu_H) &= \sum_{i=1}^n \left(\mat u_h^*(Q_i\lambda_H), \nabla  p_h^*(Q_i\mu_H) + \nabla p_{hi}\right)_{\Omega_i} \\
 	&= \sum_{i=1}^n \left(\mat u_h^*(Q_i\lambda_H), \nabla p_h^*(Q_i\mu_H) \right)_{\Omega_i} = \sum_{i=1}^n \left(K \nabla p_h^*(Q_i\lambda_H), \nabla p_h^*(Q_i\mu_H) \right)_{\Omega_i}.
\end{split}
 \end{align}
Thus, the bilinear form $b_h$ is symmetric, and, at least, positive semi-definite. 
Coercivity requires a bit more work. 

Since for each subdomain $i$, $p^*_h(Q_i\lambda_H)|_{\Gamma_i} = Q_i\lambda_H|_{\Gamma_i}$, we add by zero and expand to obtain
\begin{align*}
	b_h( \lambda_H,  \lambda_H) 
	&= \sum_{i=1}^n \left(K\nabla p_h^*(Q_i\lambda_H), \nabla p_h^*(Q_i\lambda_H) \right)_{\Omega_i} + \delta \ipbc{Q_i \lambda_H - p_h^*(Q_i \lambda_H)}{Q_i \lambda_H}_{H^{1/2}(\Gamma_i)} \\
	&\ge \sum_{i=1}^n\norm{\nabla p_h^*(Q_i\lambda_H)}^2_{\Omega_i} + \delta \norm{Q_i\lambda_H}_{H^{1/2}(\Gamma_i)}^2 - \frac{\delta}{2} \norm{p_h^*(Q_i\lambda_H)}_{H^{1/2}(\Gamma_i)}^2 - \frac{\delta}{2} \norm{Q_i\lambda_H}_{H^{1/2}(\Gamma_i)}^2 \\
	&\ge \sum_{i=1}^n\norm{\nabla p_h^*(Q_i\lambda_H)}^2_{\Omega_i} + \frac{\delta}{2} \norm{Q_i\lambda_H}_{H^{1/2}(\Gamma_i)}^2 - \frac{\delta}{2} \norm{p_h^*(Q_i\lambda_H)}_{H^{1}(\Omega_i)}^2 \\
	&= \norm{\nabla p_h^*(Q\lambda_H)}^2_{\Omega} - \frac{\delta}{2} \norm{p_h^*(Q\lambda_H)}_{H^{1}(\Omega)}^2+ \sum_{\Gamma_{ij} }\frac{\delta}{2} (\norm{Q_i \lambda_H}_{H^{1/2}(\Gamma_{ij})}^2+ \norm{Q_j \lambda_H}_{H^{1/2}(\Gamma_{ij})}^2)
\end{align*}
by using Cauchy-Schwarz, the trace inequality, and the trivial inequality $ab \le 2a^2 + 2b^2$. 

Now, using the assumption \cref{eqn:jump-assumption}
\begin{align*}
	b_h(\lambda_H, \lambda_H) 
	&\ge \norm{\nabla p_h^*(Q\lambda_H)}^2 - \frac{\delta}{2} \norm{p_h^*(Q\lambda_H)}_{H^{1}(\Omega)}^2+ \sum_{\Gamma_{ij}} \frac{\delta C_p}{2\abs{\Gamma_{ij}}} \norm{Q_i \lambda_H - Q_j \lambda_H}_{L^2(\Gamma_{ij})}^2 + \\
	&\qquad \sum_{\Gamma_{ij} }\frac{\delta}{4} (\norm{Q_i \lambda_H}_{H^{1/2}(\Gamma_{ij})}^2+ \norm{Q_j \lambda_H}_{H^{1/2}(\Gamma_{ij})}^2) \\
	&= \norm{\nabla p_h^*(Q\lambda_H)}^2 - \frac{\delta}{2} \norm{p_h^*(Q\lambda_H)}_{H^{1}(\Omega)}^2+ \sum_{\Gamma_{ij}} \frac{ C_p}{(C_p + 1)\abs{\Gamma_{ij}}} \norm{Q_i \lambda_H - Q_j \lambda_H}_{L^2(\Gamma_{ij})}^2 + \\
	&\qquad \sum_{\Gamma_{ij} }\frac{\delta}{4} (\norm{Q_i \lambda_H}_{H^{1/2}(\Gamma_{ij})}^2+ \norm{Q_j \lambda_H}_{H^{1/2}(\Gamma_{ij})}^2) 
\end{align*}
Then, by \cref{lem:helper} and the assumption \cref{eqn:injective-assumption}
\begin{align*}
	b_h(\lambda_H, \lambda_H) &\ge \sum_{\Gamma_{ij} }\frac{\delta}{4} (\norm{Q_i \lambda_H}_{H^{1/2}(\Gamma_{ij})}^2+ \norm{Q_j \lambda_H}_{H^{1/2}(\Gamma_{ij})}^2) \\
	&\ge \sum_{i=1}^n \frac{\delta}{4} \norm{Q_{i} \lambda_H}^2_{H^{1/2}(\Gamma_i)} \ge \frac{\delta}{4} \sum_{i=1}^n \norm{\lambda_H}_{H^{1/2}(\Gamma_i)}^2.
\end{align*}

\end{proof}

For the sake of notation, we assume that $\norm{\cdot}_{H^{1/2}}$ denote the sum of the $H^{1/2}$ norms over all the $\Gamma_i$ unless otherwise denoted: 
\begin{proof}[Proof of \cref{thm:err-est}]
    By Strang's second lemma, there exists a constant $C$ such that 
    \begin{align*}
        \norm{ \lambda^* -  \lambda^*_h}_{H^{1/2}} &\le C\biggl(\inf_{\mu_H \in \Lambda_0} \norm{\lambda^* - \mu_H}_{H^{1/2}} + \sup_{\mu_H \in \Lambda_0} \frac{\abs{b_h(\lambda^*, \mu_H) - L_h(\mu_H)}}{\norm{\mu_H}_{H^{1/2}}} \biggr).
    \end{align*}
    The first term, otherwise known as the approximation error, is bounded by our assumption that $\abs{p}_{H^2} < \infty$, meaning that the traces on the interior is at least in $H^{3/2}(\Gamma_i)$ for all $1 \le i \le n$, hence
    \begin{align*}
        \inf_{\mu_H \in \Lambda_0} \norm{\lambda^* - \mu_H}_{H^{1/2}} \le H \sum_{i=1}^n \norm{\lambda^*}_{H^{3/2}(\Gamma_i)} \le H \abs{p}_{H^2(\Omega)}
    \end{align*}
    by standard approximation results.

    For the consistency error, we substitute the definition into the definition of our bilinear form and linear functional in, and noting that $- \nabla K \nabla p = f$ by definition of our problem, we have for all $\mu_H \in \Lambda_0$
    \begin{align*}
        \frac{\abs{b_h(\lambda^*, \mu_H) - L_h(\mu_H)}}{\norm{\mu_H}_{H^{1/2}}} &= \frac{\abs{\sum_{i=1}^n (\mat u_h^*(Q_i \lambda^*) + \mat{\bar u}_h, \nabla (Q_i \mu_H)) - (f, Q_i \mu_H)}}{\norm{\mu_H}_{H^{1/2}}} \\
        &= \frac{\abs{\sum_{i=1}^n (K \nabla (p_h(\lambda^*) - p), \nabla Q_i \mu_H) + (K \nabla p, \nabla Q_i \mu_H) - (-\nabla K \nabla p, Q_i \mu_H)}}{\norm{\mu_H}_{H^{1/2}}} \\
        &= \frac{\abs{\sum_{i=1}^n (K \nabla (p_h(\lambda^*) - p), \nabla Q_i \mu_H) -\int_{\Gamma_{i}} K \nabla p Q_i \mu_H \cdot \mat{n}_i\, ds} }{\norm{\mu_H}_{H^{1/2}}}
    \end{align*}
    where $\mat {n}_i$ is the outward normal to the subdomain $\Omega_i$, and $p_h(\lambda^*) := p^*(Q_i \lambda^*) + \bar p$.
    The first term can be estimate using Cauchy-Schwarz inequality, 
    \begin{align*}
        \frac{\abs{\sum_{i=1}^n (K \nabla (p_h(\lambda^*) - p), \nabla Q_i \mu_H)} }{\norm{\mu_H}_{H^{1/2}}} &\le \frac{\sum_{i=1}^n \norm{K \nabla (p_h(\lambda^*) - p)}_{\Omega_i}\norm{\nabla Q_i \mu_H}_{\Omega_i} }{\norm{\mu_H}_{H^{1/2}}} \\
        &\le n \max_i \norm{K \nabla (p_h(\lambda^*) - p)}_{\Omega_i}\frac{\sum_{i=1}^n \norm{\nabla Q_i \mu_H}_{\Omega_i} }{\norm{\mu_H}_{H^{1/2}}} \\
        &\le C n \max_i \norm{K \nabla (p_h(\lambda^*) - p)}_{\Omega_i} \le C n \delta.
    \end{align*}
    where we use the fact that
    \begin{align*}
        \norm{\nabla Q_i \mu_H}_{\Omega_i} &\le \norm{Q_i \mu_H}_{H^1(\Omega_i)} 
        \le C \norm{Q_i \mu_H}_{H^{1/2}(\Gamma_i)} \le C\norm{\mu_H}_{H^{1/2}(\Gamma_i)}
    \end{align*}
    where we used the properties of discrete harmonic extensions \cite{toselli2004domain}, and the fact that $L^2$ projection is stable in $H^{1/2}$ due to interpolation \cite{bramble1991some}.
    
    As for the second term, we note that if two subdomains $\Omega_i, \Omega_j$ are adjacent, then $\mat n_i = -\mat n_j$ meaning
    \begin{align*}
        \frac{\abs{\sum_{i=1}^n \int_{\Gamma_{i}} K \nabla p Q_i \mu_H \cdot \mat{n}_i\, ds} }{\norm{\mu_H}_{H^{1/2}}} &\le \frac{\sum_{\Gamma_{ij}}\abs{ \int_{\Gamma_{ij}} K \nabla p (Q_i \mu_H - Q_j \mu_H) \cdot \mat n}_i}{\norm{\mu_H}_{H^{1/2}}} \\
         &\le \frac{\sum_{\Gamma_{ij}}\norm{ K \nabla p \cdot \mat n_i}_{H^{1/2}(\Gamma_{ij})} \norm{Q_i \mu_H - Q_j \mu_H}_{H^{-1/2}(\Gamma_{ij})}}{\norm{\mu_H}_{H^{1/2}}} \\
        &\le Cn \abs{p}_{H^2(\Omega)}  \max_i\frac{ \norm{(I - Q_i)\mu_H}_{H^{-1/2}(\Gamma_{ij})}}{\norm{\mu_H}_{H^{1/2}}} \\
        &\le Cn\abs{p}_{H^2(\Omega)} \max_i h_i 
    \end{align*}
    where we used the inequality $\norm{Q_i - Q_j} \le \norm{I - Q_i} + \norm{I - Q_j}$, the trace inequality on normal derivatives \cite[Thm. 1.5.1.2]{grisvard2011elliptic}, $L^2$ projection approximation properties \cite[(3.5)]{arbogast2007multiscale}, and where $h_i$ denotes the maximal mesh-size on each subdomain $\Omega_i$. 
\end{proof}

\subsection{FEEC Element Training}\label{sec:feec-training}
For each of the FEEC elements used in Examples 2 through 4 with the exception of the battery example (discussed below), a ``monolithic'' approach is used.
For concreteness, we will exposit the details fully for Example 3 as the other examples only differ by model hyper-parameters described in the relevant section and the training data.  

The data used to train the FEEC elements are generated from 20480 randomly sampled points from $[0, 1]^2$ evaluated by interpolating the solution of an elementary finite element solver.
In the case of the FEEC element in Example 3, a grand total of 12 different solutions each with different boundary conditions, corresponding to the third-order Bernstein polynomials on the boundary (e.g. $x^3 y^3, \binom{3}{1}x^3y^2(1 - y), \binom{3}{2}x^3y(1-y)^2$ etc), are used alongside the forcing term of $f=0$.
The Bernstein polynomials were used instead of simple hat functions as we found the additional smoothness meant pre-training of the FEEC element was more stable. 
In \cref{fig:training_data}, we plot the first five, out of twelve, of the training data we generated for \cref{sec:cylinder}.

Let $\xi$ correspond to all the hyper-parameters in the FEEC model (e.g. knot location, POU coefficients, scaling coefficients).
The loss function we use is
\begin{align}
    \min_{\xi} \sum_{k = 1}^{12} \frac{\norm{p_{\xi, k} - p_{\text{data}, k}}_{MSE}}{\norm{ p_{\text{data}, k}}_{\ell_2}} + \frac{\norm{\mat u_{\xi, k} - \mat u_{\text{data}, k}}_{MSE}}{\norm{\mat u_{\text{data}, k}}_{\ell_2} + 0.001}
    \label{eqn:loss-function}
\end{align}
where $p_{\xi, k}, \mat u_{\xi, k}$ are the FEEC solutions with the $k$th boundary condition, and $p_{\text{data}, k}$, $\mat u_{\text{data}, k}$ are the data for the $k$th boundary condition subject to the constraint.
This is exactly \cref{eqn:h1loss}, except we summed over all the different boundary conditions and minimized against all the boundary conditions in a single epoch (e.g. a monolithic approach).
The computation of the loss is efficient since $p_{\xi, k}$ for $k = 1, \ldots, 12$ can be solved with a single linear solver step because their systems only differ in their right hand sides from the boundary conditions.
Thus, the expensive stiffness matrix generation only has to be performed once at each optimization step. 
The standard Adams optimizer were used in each case as discussed in \cite{actoradata}.

As a result of the monolithic training and the basis generation of FEEC, the FEEC element will be able to accurately solve for the flux and pressure even when faced with Dirichlet boundary conditions which it has not seen before. 
For example, in \cref{fig:y_bc}, we plot the true and predicted solution of \cref{eqn:cylinder-kappa} with a boundary condition of $y$. 
Note that, while the boundary condition was never explicitly given in the training data, that the FEEC element was able to reproduce the behavior around the material discontinuity quite accurately. 

As noted in \cref{sec:battery-example}, the battery example assumes only a single data set is available, with no additional data generation with varying boundary conditions as above. 
The data for each subdomain are simply obtained via a restriction operator, and the loss is exactly \cref{eqn:h1loss}. 


\begin{figure}
    \centering
    \includegraphics[width=.8\textwidth]{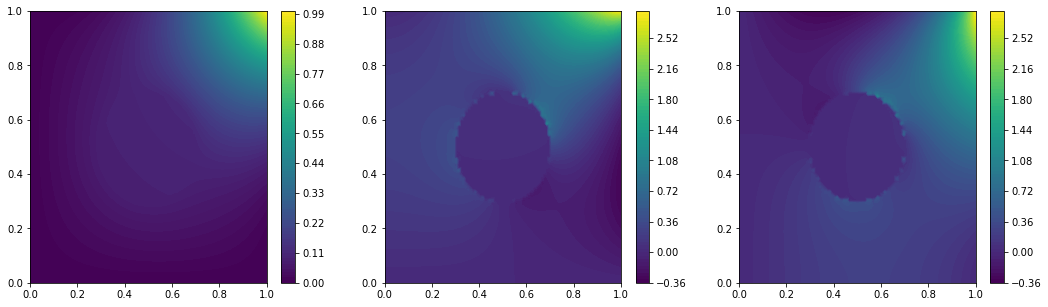}
    \includegraphics[width=.8\textwidth]{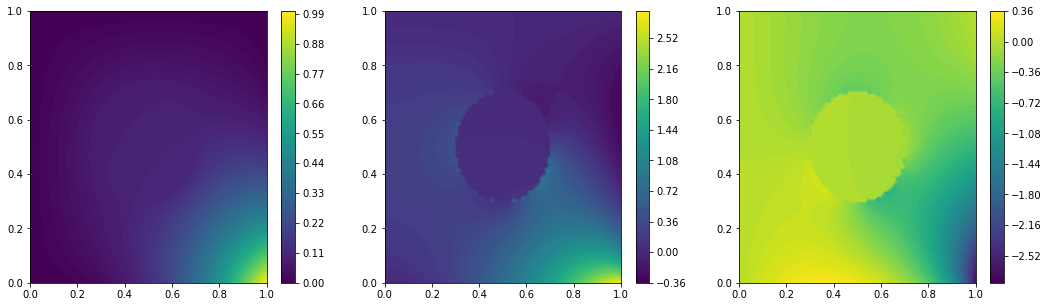}
    \includegraphics[width=.8\textwidth]{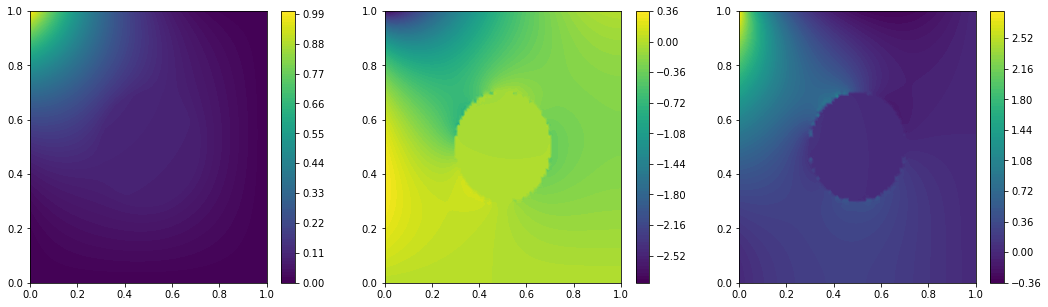}
    \includegraphics[width=.8\textwidth]{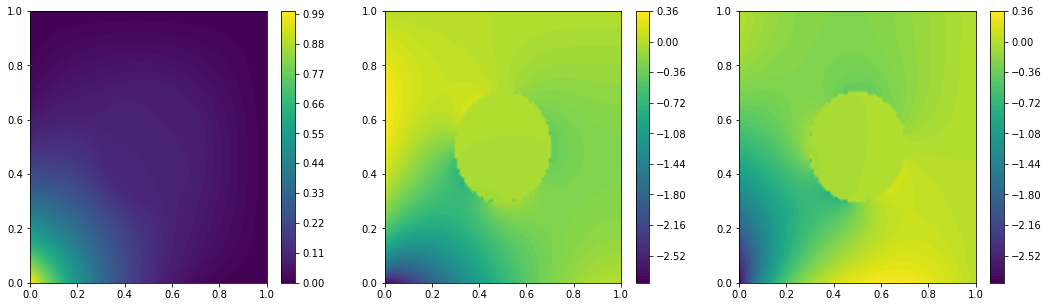}
    \includegraphics[width=.8\textwidth]{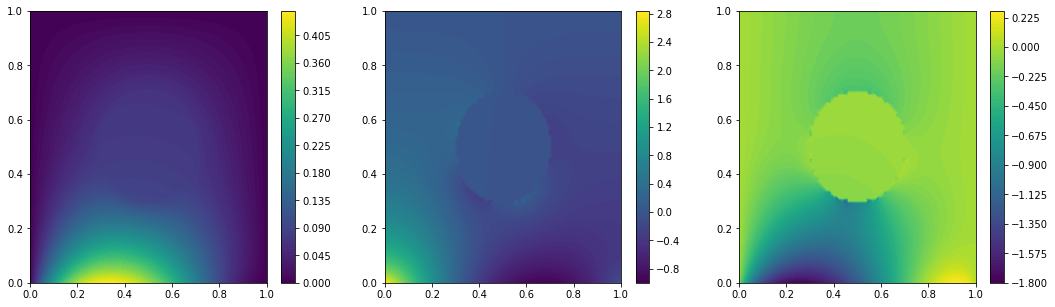} \\
    $\vdots$
    \caption{Plots illustrating some of the training data used for \cref{sec:cylinder} with pressure, flux of $x$ and flux of $y$ in the columns respectively. 
    The data is generated from a low order FEM method with $h = 1/100$.
    The key differences between each data set is that the boundary conditions are varied so that the element can respond to the different mortars.}
    \label{fig:training_data}
\end{figure}

\begin{figure}
    \centering
    \includegraphics[width=.9\textwidth]{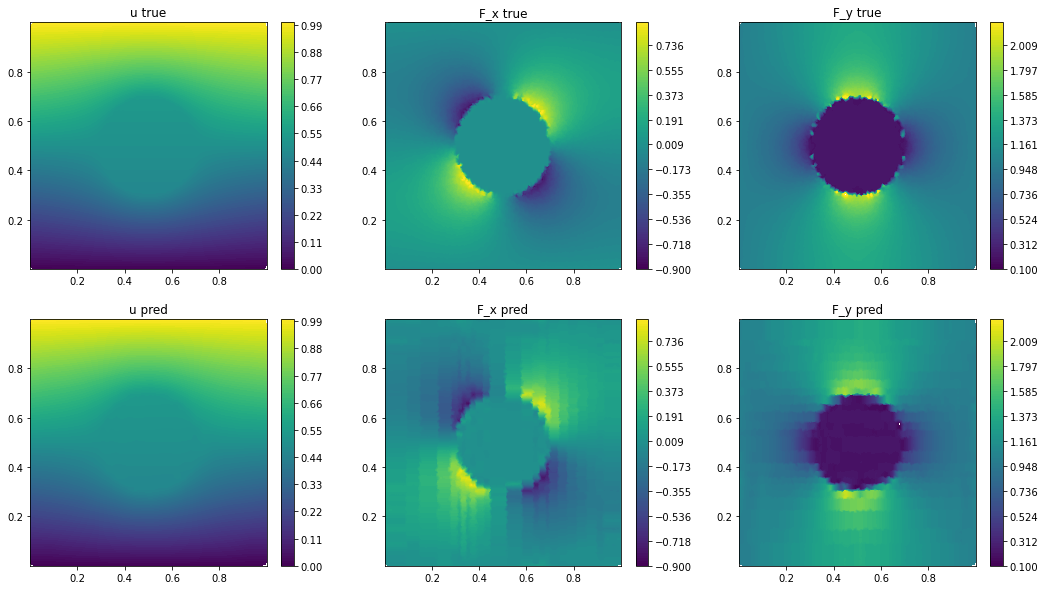}
    \caption{Plot of the true (first row) and predicted using a $24 \times 24$ fine-scale FEEC element (second row) solution to \cref{eqn:cylinder-kappa} with the boundary condition $y$ on the domain $[-.5, .5]^2$. 
    Note that while the boundary condition is not explicitly included in the training data, but rather a linear combination, we are able to reproduce the true solution accurately due to training against a large suite of boundary conditions. }
    \label{fig:y_bc}
\end{figure}

\bibliographystyle{siamplain}
\bibliography{references}
\end{document}